\documentclass[11pt,reqno]{amsart} 
\usepackage[T1]{fontenc}
\usepackage{hyperref}
\usepackage{amsmath,amssymb,amsthm,amsfonts,stmaryrd,rotating,dsfont,bm}
\usepackage[latin9]{inputenc} 
 \usepackage[a4paper]{geometry}
\usepackage[matrix,arrow,curve]{xy}
\usepackage{tensor}
\usepackage{color}
\CompilePrefix{xy/integrals}

\newcommand{\intd}{\,\mathrm{d}}

\newcommand{\eps}{\varepsilon}
\newcommand{\heps}{\hat{\varepsilon}}
\newcommand{\hbeps}{{_{B}\hat{\varepsilon}}}
\newcommand{\hepsc}{\hat{\varepsilon}_{{C}}}
\newcommand{\hceps}{{_{{C}}\hat{\varepsilon}}}
\newcommand{\hepsb}{\hat{\varepsilon}_{{B}}}

\newcommand{\bphi}{{_{B}\phi}}

\newcommand{\cphi}{{_{C}\phi}}
\newcommand{\bpsi}{{_{B}\psi}}
\newcommand{\cpsi}{{_{C}\psi}}
\newcommand{\phib}{{\phi_{B}}}

\newcommand{\psic}{{\psi_{C}}}

\newcommand{\hbphib}{{_{B}\hat{\phi}_{B}}}
\newcommand{\hcphic}{{_{C}\hat{\phi}_{C}}}
\newcommand{\hbpsib}{{_{B}\hat{\psi}_{B}}}
\newcommand{\hbpsi}{{_{{B}}\hat{\psi}}}
\newcommand{\hpsib}{\hat{\psi}_{{B}}}
\newcommand{\hcpsic}{{_{C}\hat{\psi}_{C}}}

\newcommand{\hhat}[1]{\hat{\hat{#1}}}

\newcommand{\dual}[1]{#1^{\vee}}

\newcommand{\dbA}{\dual{(\bA)}}
\newcommand{\dcA}{\dual{(\cA)}}
\newcommand{\dAb}{\dual{(\Ab)}}
\newcommand{\dAc}{\dual{(\Ac)}}

\newcommand{\comega}{{_{C}\omega}}
\newcommand{\bomega}{{_{B}\omega}}
\newcommand{\omegac}{\omega_{C}}
\newcommand{\omegab}{\omega_{B}}
\newcommand{\cupsilon}{{_{C}\upsilon}}
\newcommand{\bupsilon}{{_{B}\upsilon}}
\newcommand{\upsilonc}{\upsilon_{C}}
\newcommand{\upsilonb}{\upsilon_{B}}

\newcommand{\cphic}{\cphi{}_{C}}
\newcommand{\bpsib}{\bpsi{}_{B}}

\newcommand{\GstG}{G {_{s}\times_{t}} G}
\newcommand{\GssG}{G {_{s}\times_{s}} G}
\newcommand{\GttG}{G {_{t}\times_{t}} G}

\numberwithin{equation}{subsection}

\theoremstyle{definition} 

\swapnumbers
\newtheorem{remark}{Remark}[subsection]
\newtheorem{remarks}[remark]{Remarks}

\newtheorem*{notation*}{Notation}

\theoremstyle{plain}
\newtheorem{definition}[remark]{Definition}
\newtheorem{theorem}[remark]{Theorem}
\newtheorem{proposition}[remark]{Proposition}
\newtheorem{corollary}[remark]{Corollary}
\newtheorem{lemma}[remark]{Lemma}

\newtheorem*{assumption*}{Assumption}

\newcommand{\C}{\mathds{C}}

\newcommand{\oo}{\otimes}

\newcommand{\osub}[1]{\underset{#1}{\otimes}}

\newcommand{\mbtimes}{\osub{\mu_{B}}}

\newcommand{\hAcA}{ \hat{A}_{{ C}} \oo {_{{ C}}\hat{A}}}
\newcommand{\hAbA}{ \hat{A}_{{ B}} \oo {_{{ B}}\hat{A}}}
\newcommand{\hAlA}{\hat{ A}^{B} \otimes {_{B}\hat{ A}}}
\newcommand{\hArA}{\hat{ A}_{C} \otimes {^{C}\hat{ A}}}

\newcommand{\gr}[5]{\tensor*[^{#2}_{#3}]{#1}{^{#4}_{#5}}}

\newcommand{\id}{\iota}
\DeclareMathOperator{\Hom}{Hom}
\DeclareMathOperator{\End}{End}

\newcommand{\bB}{{_{B}B}}
\newcommand{\Bb}{B_{B}}

\newcommand{\bsA}{\gr{A}{}{B}{}{}}
\newcommand{\btA}{\gr{A}{}{}{B}{}}
\newcommand{\bAs}{\gr{A}{}{}{}{B}}
\newcommand{\bAt}{\gr{A}{B}{}{}{}}

\newcommand{\csA}{\gr{A}{}{C}{}{}}
\newcommand{\ctA}{\gr{A}{}{}{C}{}}
\newcommand{\cAs}{\gr{A}{}{}{}{C}}
\newcommand{\cAt}{\gr{A}{C}{}{}{}}

\newcommand{\cAc}{\gr{A}{}{C}{}{C}}
\newcommand{\cCc}{\gr{C}{}{C}{}{C}}
\newcommand{\bAb}{\gr{A}{}{B}{}{B}}
\newcommand{\bBb}{\gr{B}{}{B}{}{B}}

\newcommand{\AltkA}{\ctA \overline{\times} \csA} 
\newcommand{\ArtkA}{\bAs \overline{\times} \bAt} 

\newcommand{\cAsA}{\cAs \otimes \csA}
\newcommand{\bAsA}{\bAs \otimes \bsA}

\newcommand{\AlA}{\ctA \otimes \csA}

\newcommand{\ArA}{\bAs \otimes \bAt}

\newcommand{\bATA}{\btA \otimes \bAt}
\newcommand{\cATA}{\ctA \otimes \cAt}

\newcommand{\bA}{{_{B}A}}
\newcommand{\Ab}{A_{B}}
\newcommand{\cA}{{_{C}A}}
\newcommand{\Ac}{A_{C}}
\newcommand{\sbA}{A^{B}}
\newcommand{\Asb}{\gr{A}{B}{}{}{}}
\newcommand{\scA}{A^{C}}
\newcommand{\Asc}{\gr{A}{C}{}{}{}}

\newcommand{\AbA}{\Ab \oo \bA}

\newcommand{\AcA}{\Ac \oo \cA}

\newcommand{\op}{\mathrm{op}}

\newcommand{\Tl}{T_{\lambda}}
\newcommand{\Tr}{T_{\rho}}
\newcommand{\lT}{{_{\lambda}T}}
\newcommand{\rT}{{_{\rho}T}}
\newcommand{\hTl}{\hat{T}_{\lambda}}
\newcommand{\hTr}{\hat{T}_{\rho}}
\newcommand{\hlT}{{_{\lambda}\hat{T}}}
\newcommand{\hrT}{{_{\rho}\hat{T}}}

\newcommand{\beps}{{_{B}\varepsilon}}
\newcommand{\epsc}{\varepsilon_{C}}
\newcommand{\ceps}{{_{C}\varepsilon}}
\newcommand{\epsb}{\varepsilon_{B}}

\newcommand{\feps}{{_{F}\varepsilon}}
\newcommand{\epse}{\varepsilon_{E}}

\newcommand{\actleft}{\triangleleft}
\newcommand{\actright}{\triangleright}

\title[Duality of algebraic quantum groupoids]{On duality of algebraic quantum groupoids}

\author{Thomas Timmermann} 
 \address{FB Mathematik und Informatik, University of Muenster \\ Einsteinstr.\ 62, 48149
   Muenster, Germany}
 \email{timmermt@math.uni-muenster.de}

 \thanks{Supported by the SFB 878
     ``Groups, geometry and actions'' funded by the DFG}
\date{\today}


\begin{document}

\begin{abstract}
  Like quantum groups, quantum groupoids frequently appear in pairs of mutually dual objects. We develop a general Pontrjagin duality theory for quantum groupoids in the algebraic setting that extends Van Daele's duality theory for multiplier Hopf algebras and overcomes the finiteness restrictions of the approach of Kadison, Szlach\'anyi, Böhm and Schauenburg.  Our construction is based on the integration theory for multiplier Hopf algebroids and yields, as a corollary, a duality theory for weak multiplier Hopf algebras with integrals.   We compute the duals in several examples and introduce morphisms of multiplier Hopf algebroids to succinctly describe their structure. Moreover, we show that such morphisms  preserve the antipode.
\end{abstract}
\maketitle

\tableofcontents

\section{Introduction}

A fundamental feature of quantum groups and quantum groupoids is that, like locally compact abelian groups and their Pontrjagin duals, many appear as pairs of mutually dual objects, for example, function and convolution algebras of groups and groupoids, quantizations of function algebras and of universal enveloping algebras of Lie-Poisson groups and groupoids \cite{chari}, \cite{korogodski}, or Galois symmetries for depth 2 inclusions of rings or factors \cite{kadison:inclusions}.

General duality results on existence of a dual object and an identification of the bidual with the initial object were obtained, for example, for finite-dimensional (weak) Hopf algebras \cite{kornel:weak}, for quantum groupoids that are finite relative to the base algebras by Kadison, Szlach\'anyi, Böhm and Schauenburg \cite{boehm:bijective}, \cite{kadison:inclusions},  \cite{schauenburg}, for multiplier Hopf algebras with integrals by Van Daele \cite{daele}, and in the operator-algebraic framework for locally compact quantum groups by Kustermans and Vaes \cite{vaes:1} and for measured quantum groupoids by Enock and Lesieur \cite{enock:action}, \cite{lesieur}.

In this article, we establish a general duality theory for algebraic quantum groupoids in the form of multiplier Hopf algebroids \cite{timmermann:regular}, without any finiteness assumption, following the elegant approach to the duality of algebraic quantum groups of Van Daele \cite{daele}. The key assumption that we impose is existence of integrals \cite{timmermann:integration}, which correspond to integration with respect to the Haar measure on a group or with respect to a Haar system and a quasi-invariant measure on a groupoid, and also play a crucial role in  the operator-algebraic approaches to quantum groups and quantum groupoids mentioned above.  The regular multiplier Hopf algebroids that we use simultaneously generalize the weak (multiplier) Hopf algebras \cite{boehm:weak}, \cite{kornel:weak},  \cite{daele:weakmult0}, \cite{daele:weakmult}, and Hopf algebroids \cite{boehm:hopf}, \cite{lu:hopf},  \cite{xu}. As a special case, we obtain a duality theory for weak (multiplier) Hopf algebras with integrals, again without any finiteness assumption.

As a tool, we introduce morphisms of multiplier Hopf algebroids, show
that they preserve the antipode, and use them to describe the
structure of the dual in several examples.

\medskip

To motivate our approach and explain the problems that we have to solve, let us first outline the afore-mentioned approaches.

If $A$ is a finite-dimensional (weak) Hopf algebra with multiplication $m\colon A\otimes A \to A$ and comultiplication $\Delta \colon A\to A\otimes A$, then the dual vector space $A^{\vee}$ is a (weak) Hopf algebra with respect to the dual maps $\Delta^{\vee} \colon A^{\vee} \otimes A^{\vee} \to A^{\vee}$ and $m^{\vee} \colon A^{\vee} \to A^{\vee} \otimes A^{\vee}$, where finite-dimensionality is used to identify $A^{\vee} \otimes A^{\vee}$ with $(A\otimes A)^{\vee}$.

To extend this duality beyond the finite-dimensional case and from (weak) Hopf algebras to (multiplier) Hopf algebroids, two problems have to be overcome.

The first problem is that in the infinite-dimensional case, the tensor product $A^{\vee} \otimes A^{\vee}$ is a proper subspace of $(A\otimes A)^{\vee}$, and while the multiplication on $A^{\vee}$ can still be defined by the formula $\upsilon \ast \omega := (\upsilon \otimes \omega)\circ \Delta$, the comultiplication may not be well-defined because the dual map $m^{\vee} \colon A^{\vee} \to (A\otimes A)^{\vee}$ need not take values in the subspace $A^{\vee} \otimes A^{\vee}$. To solve this problem, we follow Van Daele \cite{daele} who assumed existence of a left or right integral on $A$, that is, of a functional $\phi$ or $\psi$ satisfying
\begin{align} \label{eq:hopf-invariance}
  (\iota \otimes \phi)\Delta(a) = \phi(a)1_{A}\quad \text{or} \quad (\psi\otimes \iota)\Delta(a) = \psi(a)1_{A}.
\end{align}
Given such integrals, the  spaces of all functionals of the form $a\cdot \phi \colon b\mapsto  \phi(ba)$ or $\psi\cdot a \colon b\to \psi(ab)$, respectively, coincide, form a subalgebra $\hat{A}$ of $A^{\vee}$, and carry a comultiplication  $\hat{\Delta}$ that takes values in the multiplier algebra $M(\hat{A} \otimes \hat{A})$ such that $(\hat{A}\otimes 1)\hat{\Delta}(A)$ and   $\hat{\Delta}(\hat{A})(1\otimes \hat{A})$  lie in $\hat{A} \otimes \hat{A}$ and
\begin{align} \label{eq:dual-mh}
  \begin{aligned}
    ((\upsilon \otimes 1)\hat{\Delta}(\omega)\mid a\otimes b) &= (\upsilon \otimes \omega\mid \Delta(a)(1\otimes b)), \\
    (\hat{\Delta}(\upsilon)(1\otimes \omega)\mid a\otimes b) &= (\upsilon
    \otimes \omega\mid (a\otimes 1)\Delta(b))
  \end{aligned}
\end{align}
for all $a,b\in A$ and $\upsilon,\omega \in \hat{A}$. Then $\hat{A}$ is a so-called multiplier Hopf algebra and has a left and a right integral $\hat{\psi}$ and $\hat{\phi}$ again, given by $a\cdot \phi \mapsto \varepsilon(a)$ and $\psi\cdot a \mapsto \varepsilon(a)$, respectively, where $\varepsilon$ denotes the counit of $A$. Moreover, the associated dual $\hat{\hat{A}}$ is canonically isomorphic to $A$.

The second problem is that if one passes from weak Hopf algebras to Hopf algebroids, the  comultiplication  does no longer map $A$ to $A \otimes A$. Instead, the algebra $A$ can be regarded as a bimodule over two anti-isomorphic base algebras $R$ and $L$, and one has a left comultiplication $\Delta_{L}$ and a right comultiplication $\Delta_{R}$ that map $A$ to the balanced tensor products $A_{L} \otimes {_{L}A}$ or ${A_{R}} \otimes {_{R}A}$, respectively. Hence, we can not define a multiplication on the dual vector space $A^{\vee}$ as before, but only on each of the four dual modules $(_{L}A)^{\vee}$, $(A_{L})^{\vee}$, $(_{R}A)^{\vee}$ and $(A_{R})^{\vee}$, for example, on $(_{L}A)^{\vee}$ and $(A_{R})^{\vee}$ by
\begin{align} \label{eq:dual-modules-convolution}
  _{L}\upsilon \ast {_{L}\omega}  = { _{L}\upsilon} \circ (\iota \otimes {_{L}\omega}) \circ \Delta_{L} \quad \text{and} \quad
\upsilon_{R} \ast \omega_{R} = \omega_{R} \circ (\upsilon_{R} \otimes \iota) \circ \Delta_{R}.
\end{align}
If the modules $_{L}A,A_{L}$ and $_{R}A,A_{R}$ are finitely generated and projective, then one can dualize the multiplication on $A$ and define a one-sided comultiplication on each of the four dual modules so that one obtains two left bialgebroids and two right bialgebroids. Without further assumption, it is not clear how to assemble these four one-sided bialgebroids into one (two-sided) Hopf algebroid. In \cite{boehm:bijective}, Böhm and Szlach\'anyi show that this can be done if one assumes existence of an integral  \emph{in}  $A$ \cite{boehm:integrals}, and obtain a duality theory for Hopf algebroids with cointegrals.\footnote{To avoid confusion between integrals \emph{in} $A$ and integrals \emph{on} $A$, we henceforth reserve the term integral for the latter and use the term \emph{cointegral} for the former.}

\medskip

Let us now explain our approach and the organization of this article.  

Briefly, we carry over Van Daele's duality theory to regular multiplier Hopf algebroids, using the theory of integration on such objects developed in \cite{timmermann:integration}.

\medskip

In Section \ref{section:preliminaries}, we summarize the definition of a  \emph{regular multiplier Hopf algebroid} and the ingredients needed for integration on such an object.

 A regular multiplier Hopf algebroid consists of an algebra $A$, possibly without unit, two   subalgebras $B,C \subseteq M(A)$ with two fixed anti-isomorphisms $C \rightleftarrows B$,  and a left comultiplication $\Delta_{C}$ and a right comultiplication  $\Delta_{B}$ subject to several natural conditions like co-associativity and existence of counits and an antipode. 

 To perform integration,  first need a \emph{partial left integral} and a \emph{partial right integral}, that is, a $C$-bilinear map $\cphic \colon A \to C$ and a $B$-bilinear map $\bpsib\colon A \to B$ satisfying invariance conditions similar to \eqref{eq:hopf-invariance}.  
 In the situation studied by Böhm and Szlach\'anyi, where $A$ is a Hopf algebroid with a cointegral, existence of such partial integrals is automatic \cite[Lemma 5.6]{boehm:bijective}; see also \cite{boehm:integrals}, \cite{boehm:erratum}.
Next, we need a \emph{base weight}, which consists  of a faithful functional $\mu_{B}$ on $B$ and a faithful functional $\mu_{C}$ on $C$ satisfying a few natural conditions. 
If the multiplier Hopf algebroid arises from a weak (multiplier) Hopf algebra as in \cite{daele:relation}, there exists a canonical choice of such a base weight.

 A regular multiplier Hopf algebroid with partial integrals and a base weight as above will be called a \emph{measured multiplier Hopf algebroid}.

\medskip

 In Section \ref{section:dual}, we associate to a measured multiplier Hopf algebroid $A$ as above a dual measured multiplier Hopf algebroid. 

 The underlying \emph{dual algebra} $\hat{A}$ was introduced already in \cite{timmermann:integration} and will be studied in detail in subsection \ref{subsection:dual-algebra}.  As in the case of multiplier Hopf algebras, it is a subspace of the linear dual $A^{\vee}$ defined with the help of the total integrals $\phi=\mu_{C} \circ \cphi$ and $\psi=\mu_{B} \circ \bpsib $.  The base weight allows us to identify this subspace with four different convolution algebras contained in certain dual modules of $A$, thereby solving the second problem outlined above.  The algebras $B$ and $C$ embed naturally into $M(\hat{A})$ so that we can regard $\hat{A}$ as a module over these base algebras in various ways.

We next construct a dual left comultiplication $\hat{\Delta}_{B}$ and a dual right comultiplication $\hat{\Delta}_{C}$  which dualize the multiplication on $A$ in a similar way as  expressed by \eqref{eq:dual-mh}. To do so, we need to define various pairings between balanced tensor squares of $\hat{A}$ and balanced tensor squares of $A$. This is done in subsection \ref{subsection:dual-comultiplication}.

 The full duality result will be obtained  in subsection \ref{subsection:dual-integrals}. We define 
   a partial right and a partial  left integral on $\hat A$ by similar formulas  as in the case of multiplier Hopf algebras \cite{daele},  show that $\hat{A}$ becomes 
 measured multiplier Hopf algebroid again, and prove that the natural map $A \to \dual{(\hat A)}$ identifies $A$ with the bidual $\hat{\hat A}$.

  The dual $\hat{A}$ turns out to be unital if and only if $A$ has a left and a right cointegral. Existence of such cointegrals is the starting assumption in the duality developed in  \cite{boehm:bijective}.

\medskip

In Section \ref{section:morphisms}, we introduce the notion of a morphism between regular multiplier Hopf algebroids, which requires some care when the algebras involved are non-unital and when the base algebras are allowed to change. We show that every morphism preserves the antipode and pay special attention to the case where the target is the dual of a measured regular multiplier Hopf algebroid.

\medskip

 In Section \ref{section:examples}, we finally consider  examples coming from \'etale groupoids, from actions of Hopf algebras,  and from braided-commutative Yetter-Drinfeld algebras, and compute the dual in each of these cases.  Moreover, we show that the duality constructed in Section \ref{section:dual} yields, as a special  case, a duality theory for weak multiplier Hopf algebras with integrals. 

\medskip

The material developed in this article raises several questions that will have to be addressed separately.

Given the theory of integrals developed here, one should be able to construct (operator-algebraic) measured quantum groupoids out of suitable (algebraic) measured multiplier Hopf $*$-algebroids, similarly as it was done for multiplier Hopf algebras by Kustermans and Van Daele in \cite{kustermans:analytic-2} and \cite{kustermans:algebraic}, and for dynamical quantum groups in \cite{timmermann:dynamical}.  A first step will be taken in \cite{timmermann:opalg}. A crucial assumption in the theory of measured quantum groupoids is that the modular automorphism groups of the left and of the right integral commute, and it would be desirable to study the implications of this assumption in the present context.

Another question concerns the implications of existence of integrals on the theory of corepresentations. We expect that in the case where $A$ is proper in the sense that $BC\subseteq A$, and under suitable assumptions on the base algebras $B$ and $C$, many results on the corepresentation theory of compact quantum groups carry over to measured multiplier Hopf $*$-algebroids.

One should also clarify the relation to the duality obtained by Schauenburg \cite{schauenburg}, Kadison and Szlach\'anyi \cite{kadison:inclusions} and Böhm and Szlach\'anyi \cite{boehm:bijective} in the case where $A$ is unital and finitely generated projective as a module over $B$ and $C$. In particular, we plan to study the relation between the cointegrals that  underly the duality developed in  \cite{boehm:bijective}, and the base weights and partial  integrals used in our appoach.

\subsubsection*{Preliminaries}

We shall use the following conventions and terminology.

All algebras and modules will be complex vector spaces and all homomorphisms will be linear maps, but much of the theory developed in this article applies in wider generality.

The identity map on a set $X$ will be denoted by $\iota_{X}$ or simply $\iota$.

Let $B$ be an algebra, not necessarily unital. We denote by $B^{\op}$
the \emph{opposite algebra}, which has the same underlying vector
space as $B$, but the reversed multiplication.

Given a right module $M$ over $B$, we write $M_{B}$ if we want to emphasize that $M$ is regarded as a right $B$-module. We call $M_{B}$ \emph{faithful} if for each non-zero $b\in B$ there exists an $m\in M$ such that $mb$ is non-zero, \emph{non-degenerate} if for each non-zero $m\in M$ there exists a $b \in B$ such that $mb$ is non-zero, \emph{idempotent} if $MB=M$, and we say that $M_{B}$ \emph{has local units in $B$} if for every finite subset $F\subset M$ there exists a $b\in B$ with $mb=m$ for all $m\in F$. Note that the last property implies the preceding two. We denote by
\begin{align*}
  \dual{(M_{B})} :=\Hom(M_{B},B_{B})
\end{align*}
  the dual module, and by $\dual{f} \colon \dual{(N_{B})} \to \dual{(M_{B})}$ the dual of a morphism $f\colon M_{B}\to N_{B}$ of right $B$-modules, given by $\dual{f}(\chi)= \chi \circ f$.  We use the same notation for duals of vector spaces and of linear maps. We furthermore denote by $L(M_{B}) := \Hom(\Bb,M_{B})$ the space of \emph{left multipliers} of the module $M_{B}$.

For left modules, we obtain the corresponding notation and terminology by identifying left $B$-modules with right $B^{\op}$-modules.  We denote by $R(_{B}M):=\Hom(\bB,{_{B}M})$ the space of \emph{right multipliers} of a left $B$-module $_{B}M$.

Given  a right $B$-module $M_{B}$ and a left $B$-module $_{B}N$, we denote by $M_{B} \otimes {_{B}N}$ the (balanced) tensor product. We shall frequently use the following  slice maps associated to elements $\phi_{B} \in (M_{B})^{\vee}$ and $\bpsi \in (_{B}N)^{\vee}$, 
\begin{align} \label{eq:slice}
  \begin{aligned}
    \phi_{B} \odot \iota &\colon M_{B} \otimes {_{B}N} \to N, \ m\otimes n \mapsto \phi_{B}(m)n, \\
    \id \odot \bpsi &\colon M_{B} \otimes {_{B}N} \to M, \ m\otimes n
    \mapsto m\bpsi(n).
  \end{aligned}
\end{align}

We write $B_{B}$ or ${_{B}B}$ when we regard $B$ as a right or left module over itself with respect to right or left multiplication. We say that the algebra $B$ is \emph{non-degenerate}, \emph{idempotent}, or \emph{has local units} if the modules ${_{B}B}$ and $B_{B}$ both are non-degenerate, idempotent or both have local units in $B$, respectively. Note that the last property again implies the preceding two.

We denote by $L(B)=\End(B_{B})$ and $R(B)=\End({_{B}B})^{\op}$ the algebras of left or right multipliers of $B$, respectively, where the multiplication in the latter algebra is given by $(fg)(b):=g(f(b))$. Note that $B_{B}$ (or\ ${_{B}B}$) is non-degenerate if and only if the natural map from $B$ to $L(B)$ (or $R(B)$, respectively) is injective. If $B_{B}$ is non-degenerate, we define the multiplier algebra of $B$ to be the subalgebra $M(B) :=\{ t\in L(B) : Bt\subseteq B\} \subseteq L(B)$, where we identify $B$ with its image in $L(B)$. Likewise we could define $M(B) =\{ t\in R(B) : tB \subseteq B\}$ if ${_{B}B}$ is non-degenerate.  If both definitions make sense, that is, if $B$ is non-degenerate, then they evidently coincide up to a natural identification, and a multiplier is given by a pair of maps $t_{R},t_{L}\colon B\to B$ satisfying $t_{R}(a)b=at_{L}(b)$ for all $a,b\in B$.

Given a left or right $B$-module $M$ and a space $N$, we regard the space of linear maps from $M$ to $N$ as a right or left $B$-module, where $(f \cdot b)(m)=f(bm)$ or $(b\cdot f)(m)=f(mb)$ for all maps $f$ and all elements $b\in B$ and $m\in M$, respectively.

In particular, we regard the dual space $\dual{B}$ of a non-degenerate, idempotent algebra $B$ as a bimodule over $M(B)$, where $(a \cdot \omega \cdot b)(c)=\omega(bca)$, and call a functional $\omega \in \dual{B}$ \emph{faithful} if the maps $B \to \dual{B}$ given by $d \mapsto d\cdot\omega$ and $d \mapsto\omega \cdot d$ are injective, that is, $\omega(dB)\neq 0$ and $\omega(Bd) \neq 0$ whenever $d\neq 0$.

Given a faithful functional $\omega \in \dual{B}$, we say that $\omega$ \emph{admits a modular automorphism} if there exists an automorphism $\sigma$ of $B$ such that $\omega(ab)=\omega(b\sigma(a))$ for all $a,b\in B$. One easily verifies that this condition holds if and only if $B\cdot \omega = \omega \cdot B$, and that then $\sigma$ is characterized by the relation $\sigma(b) \cdot \omega = \omega \cdot b$ for all $b\in B$.

Assume that $B$ is a $*$-algebra. We call a functional $\omega \in \dual{B}$ \emph{self-adjoint} if it coincides with $\omega^{*}=\ast \circ \omega \circ \ast$, that is, $\omega(a^{*})=\omega(a)^{*}$ for all $a\in B$, and \emph{positive} if additionally $\omega(a^{*}a)\geq 0$ for all $a\in A$.

\section{Measured regular multiplier Hopf algebroids}
\label{section:preliminaries}

This section summarizes the definition and main properties of measured multiplier Hopf
algebroids, which are regular multiplier Hopf algebroids with
partial integrals and base weights. 
References are
\cite{timmermann:integration} and
\cite{timmermann:regular}.

\subsection{Regular multiplier Hopf algebroids}

\label{section:multiplier-bialgebroids}
Briefly,
 a {regular multiplier Hopf algebroid} consists of an algebra $A$, possibly without unit, two   subalgebras $B,C \subseteq M(A)$ with two fixed anti-isomorphisms $C \rightleftarrows B$, and a left comultiplication $\Delta_{C}$ and a right comultiplication  $\Delta_{B}$ such that the maps
\begin{align} \label{eq:intro-canonical}
  \begin{aligned}
    \Tl &\colon a \otimes b \mapsto \Delta_{C}(b)(a \otimes 1), &
    \Tr&\colon a \otimes b\mapsto \Delta_{C}(a)(1 \otimes b), \\
    \lT &\colon a\otimes b\mapsto (a \otimes 1)\Delta_{B}(b), & \rT
    &\colon a\otimes b \mapsto (1 \otimes b)\Delta_{B}(a)
  \end{aligned}
\end{align}
induce bijections between certain balanced tensor products of $A$ relative to the base algebras $B$ and  $C$, respectively. This bijectivity condition is equivalent to the existence of a left counit $\ceps \colon A \to C$ for $\Delta_{C}$, a right counit $\epsb \colon A \to B$ for $\Delta_{B}$, and an antipode.

Let us now formulate all assumptions  precisely and  fix some notation.

A (two-sided) multiplier bialgebroid  is a tuple 
\begin{align*}
  (A,B,C,t_{B},t_{C},\Delta_{B},\Delta_{C})
\end{align*}
consisting of the following ingredients:
\begin{enumerate}
\item a \emph{total algebra} $A$, which  is   non-degenerate and idempotent;
\item \emph{base algebras} $B,C \subseteq M(A)$;
\item  anti-isomorphisms $t_{B}\colon B\to C$ and $t_{C}\colon C\to B$;
\item a \emph{right comultiplication $\Delta_{B}$} and a \emph{left comultiplication $\Delta_{C}$}, which are homomorphisms from $A$ into certain extended right or left Takeuchi products.
\end{enumerate}
The base algebras $B$ and $C$ need not be commutative, but have to commute with each other, and the algebra $A$ is assumed to be non-degenerate and idempotent when regarded as a left or as a right module over $B$ and over $C$.

We denote elements of $B$ by $x,x',\ldots$, elements of $C$ by $y,y',\ldots$, and reserve $a,b,c,\ldots$ for elements of $A$.  We write
\begin{align*}
&  \bsA, && \bAs, && \csA, && \cAs &&\text{and}  &&  \btA, && \bAt, && \ctA, && \cAt,
\end{align*}
when we regard $A$ as a left or right module over $B$ or $C$, respectively, using left or right multiplication in the case of subscripts and the base anti-isomorphisms $t_{B}$ and $t_{C}$ in the case of superscripts. Thus, in order of appearance, the module structures are given by 
\begin{align*}
x\cdot a &=xa,   & a\cdot x &=ax,  & y\cdot a&= ya, & a\cdot y &= ay, \\
  a\cdot x&=t_{B}(x)a, & x\cdot a &= at_{B}(x), & a\cdot y &= t_{C}(y)a, & y\cdot a &= at_{C}(y)
\end{align*}
for all $a\in A$, $x\in B$ and $y\in C$, respectively.  By assumption, all of these modules are non-degenerate and idempotent.  Any two of these module structures commute in the sense that one obtains a bimodule, and we use two indices when we regard $A$ as a bimodule.

We denote the balanced tensor product of modules by the ordinary tensor sign but emphasize the module structures involved. 

  Regard $\AlA$ as a right and $\ArA$ as a left module over $M(A)\otimes M(A)$ in the obvious way. 
Since $B$ and $C$ commute, we can also regard $\AlA$ as a left and $\ArA$ as a right module over $C\otimes B$.
We write $(a \otimes b)T$ for the image of an element $a\otimes b$ under an element $T \in \End(\ArA)^{\op}$, so that $(a\otimes b)(ST) = ((a\otimes b) S)T$ for all $a,b\in A$ and $S,T\in \End(\ArA)^{\op}$.

 Denote by
\begin{align*} 
\AltkA \subseteq \End(\ctA \otimes \csA)
\end{align*} 
the subspace  formed by all endomorphisms $T$ of $\ctA \otimes \csA$ satisfying the following condition: 
 for every $a,b \in A$, there exist elements
  \begin{align*} T(a \otimes 1) \in \ctA \otimes \csA \quad \text{and}
    \quad T(1 \otimes b) \in \ctA \otimes \csA
  \end{align*}
  such that 
  \begin{align*}
T(a \otimes b) = (T(a \otimes 1))(1 \otimes b) = (T(1
    \otimes b))(a \otimes 1)    
  \end{align*}
This subspace is a subalgebra and commutes with the right $A\otimes
A$-module action.  Likewise,
denote by
\begin{align*}
  \ArtkA \subseteq \End(\bAs \otimes \bAt)^{\op}
\end{align*}
the subspace formed by all endomorphisms $T$ such that for all $a,b\in
A$, there exist elements $(a\otimes 1)T \in \bAs \otimes \bAt$ and $(1
\otimes b)T \in \bAs \otimes \bAt$ such that
\begin{align*}
  (a\otimes b)T = (1\otimes b)((a\otimes 1)T) = (a\otimes 1)((1 \otimes b)T).
\end{align*}

\begin{definition} \label{definition:mult-hopf-algebroid} A \emph{multiplier bialgebroid} $(A,B,C,t_{B},t_{C},\Delta_{B},\Delta_{C})$ consists of a non-degenerate, idempotent algebra $A$, subalgebras $B,C \subseteq M(A)$, anti-isomorphisms $t_{B} \colon B\to C$ and $t_{C} \colon C\to B$, and homomorphisms
  \begin{align*}
    \Delta_{C}\colon A \to \AltkA \quad \text{and} \quad \Delta_{B} \colon A \to \ArtkA
  \end{align*}
 satisfying the following conditions:
  \begin{enumerate}
  \item $B$ and $C$ commute with each other and the modules $\bA,\Ab,\cA,\Ac$ are non-degenerate and idempotent;
  \item $\AlA$ is non-degenerate as a right module over $1\otimes A$ and over $A\otimes 1$, and $\ArA$ is non-degenerate as a left module over $1\otimes A$ and over $A\otimes 1$;
\item for all $x,x'\in B$, $y,y'\in C$, $a\in A$,
  \begin{align} \label{eq:left-comultiplication-bilinear}
    \Delta_{C}(xyax'y') &= (y\otimes x)\Delta_{C}(a)(y' \otimes x'), \\
    \Delta_{B}(xyax'y') &= (y\otimes x)\Delta_{B}(a)(y' \otimes x'); \label{eq:right-comultiplication-bilinear}
  \end{align}
\item for all $a,b,c\in A$, the following coassociativity conditions hold:
  \begin{align} \label{eq:left-comultiplication-coassociative}
 (\Delta_{C} \otimes \iota)(\Delta_{C}(b)(1 \otimes c)) (a\otimes 1 \otimes 1) &=
(\iota\otimes \Delta_{C})(\Delta_{C}(b)(a \otimes 1))(1 \otimes 1 \otimes c),
\\ \label{eq:right-comultiplication-coassociative}    (a \otimes 1 \otimes 1) ((\Delta_{B} \otimes \id)((1 \otimes
    c)\Delta_{B}(b))) &= (1 
\otimes 1 \otimes c)((\id \otimes \Delta_{B})((a \otimes 1)\Delta_{B}(b))) 
\\ \label{eq:mixed-coassociativity-one}
(\Delta_{C} \otimes \id)((1 \otimes c)\Delta_{B}(b)))(a \otimes 1 \otimes 1), &= (1
\otimes 1 \otimes c)((\id \otimes \Delta_{B})(\Delta_{C}(b)(a \otimes 1))), \\
\label{eq:mixed-coassociativity-two}
 (a \otimes 1 \otimes 1)((\Delta_{B}
\otimes \id)(\Delta_{C}(b)(1 \otimes c))) &= ((\id \otimes \Delta_{C})((a \otimes 1)\Delta_{B}(b)))(1 \oo
1 \otimes c).
  \end{align}
  \end{enumerate}
\end{definition}
Let $(A,B,C,t_{B},t_{C},\Delta_{B},\Delta_{C})$ be a multiplier bialgebroid.

Condition (2) ensures that for all $a,b\in A$, the elements
  $\Delta_{C}(a)(1\otimes b)$ and $ \Delta_{C}(b)(a\otimes 1)$ in $\AlA$ and $(a\otimes 1)\Delta_{B}(b)$ and $(1\otimes b)\Delta_{B}(a)$ in $\ArA$
  are uniquely determined.  Some typical cases where this condition holds are given in \cite[Lemma 2.1]{timmermann:regular}.
Condition (3) implies that the four linear maps which send an element $a\otimes b\in A\otimes A$ to either one of these four elements factorize to well-defined linear maps
\begin{align}
    \label{eq:tl}
    \Tl \colon \bATA &\to \AlA,   \quad  a\otimes b \mapsto \Delta_{C}(b)(a\otimes 1), \\     \label{eq:tr}
    \Tr \colon \bAsA &\to \AlA,     \quad a\otimes b \mapsto \Delta_{C}(a)(1 \otimes b), \\
    \label{eq:lt}
    \lT \colon \cAsA &\to \ArA, \quad a\otimes b\mapsto (a\otimes 1)\Delta_{B}(b), \\
    \label{eq:rt}
 \rT \colon \cATA &\to \ArA, \quad  a\otimes b\mapsto (1\otimes b)\Delta_{B}(a).
\end{align}
We call these  maps  the \emph{canonical maps} of the multiplier bialgebroid.

\begin{definition}
Let $\mathcal{A}=(A,B,C,t_{B},t_{C},\Delta_{C},\Delta_{B})$   be a multiplier bialgebroid. 

A \emph{left counit}  for $\mathcal{A}$  is a morphism $\ceps \in \Hom(\cA^{C}, \cCc)$ satisfying
  \begin{gather}
        \label{eq:left-counit}
        (\ceps \odot\id)(\Delta_{C}(a)(1 \otimes b)) = ab = (\id
        \odot \ceps)(\Delta_{C}(a)(b\otimes 1))
  \end{gather}
for all $a,b\in A$, and  a \emph{right counit} for $\mathcal{A}$ is a morphism $\epsb \in \Hom(\bAt_{B},\bBb)$ satisfying
\begin{gather}
  (\epsb \odot \id)((1\otimes b)\Delta_{B}(a))=ba =  (\id \odot
\epsb)((b\otimes 1)\Delta_{B}(a)).  \label{eq:right-counit}
\end{gather}
\end{definition}
 Here, $(\ceps \odot \id)(a\otimes b)= \ceps(a)b$ and$
(\id\odot \ceps)(a\otimes b)=  t_{C}(\ceps(b))a$ for every elementary
tensor $a \otimes b\in \AlA$, and   $\epsb\odot \id$ and $\id \odot
\epsb$   are defined  similarly.

Let $(A,B,C,t_{B},t_{C},\Delta_{B},\Delta_{C})$ be a multiplier bialgebroid. Then by \cite[Theorem 5.6]{timmermann:regular}, the following conditions are equivalent:
\begin{itemize}
\item[{(H1)}] the canonical maps $\Tl,\Tr,\lT,\rT$ are bijective and the subspaces
\begin{align*}
\gr{I}{}{B}{}{} &= \langle \omega(a) : \omega \in \dbA, a \in A\rangle, &
  \gr{I}{}{}{B}{} &= \langle \omega(a) : \omega \in \dual{(\btA)}, a \in A\rangle, \\
  \gr{I}{}{}{}{C} &= \langle \omega(a) : \omega \in \dAc, a \in A\rangle, &
  \gr{I}{C}{}{}{} &= \langle \omega(a) : \omega \in \dual{(\cAt)}, a \in A\rangle
\end{align*}
of $B$ or $C$, respectively, satisfy
 $A=t_{B}(\gr{I}{}{B}{}{})\cdot A = \gr{I}{}{}{B}{}  \cdot A =A\cdot t_{C}(\gr{I}{}{}{}{C}) = A \cdot \gr{I}{C}{}{}{}$;
\item[{(H2)}] there exist a left counit $\ceps$, a right counit $\epsb$ and an anti-automorphism $S$ of $A$ such that its extension to $M(A)$ satisfies
  \begin{align} \label{eq:antipode-bilinear}
   S_{C}:=S|_{C}= t_{C}^{-1}  \quad \text{and} \quad S_{B}:=S|_{B}=t_{B}^{-1}
  \end{align}
  and the following diagrams commute, where $m$ denotes the multiplication maps:
    \begin{align} \label{dg:antipode}
      \xymatrix@C=30pt@R=15pt{\Ab\otimes \bA \ar[d]_{ T_{\rho}}
        \ar[r]^(0.55){\epsb \odot \id} & A, & \Ac\otimes \cA
        \ar[d]_{_{\lambda} T} \ar[r]^(0.55){\id \odot\ceps}
        &A. \\
        \AlA \ar[r]_{S \otimes \id} & \ar[u]_{m} \Ac
        \otimes \cA & \ArA \ar[r]_{\id \otimes S} & \ar[u]_{m}
        \Ab \otimes \bA}
    \end{align}
\end{itemize}
If these conditions hold, then the maps $\beps,\epsc$ and $S$ are uniquely determined.
\begin{definition}
  A regular multiplier Hopf algebroid is a multiplier bialgebroid satisfying the equivalent conditions {(H1)} and {(H2)}. The map $S$ is called its \emph{antipode}.
\end{definition}

 Let $\mathcal{A}$ be a regular multiplier Hopf algebroid. Then the following diagrams commute,
\begin{gather} \label{dg:galois-inverse}
        \xymatrix@R=15pt@C=25pt{\ctA \otimes \cAt \ar[d]_{\rT} \ar[r]^{\id \oo S} & \AlA,
          \\
          \ArA \ar[r]_{\id \oo S} & \AbA \ar[u]_{\Tr} } \quad
        \xymatrix@R=15pt@C=25pt{\btA \otimes \bAt \ar[d]_{\Tl} \ar[r]^{S \oo \id} & \ArA,
          \\
          \AlA \ar[r]_{S \oo \id} & \AcA \ar[u]_{\lT} } \\ \label{dg:galois-antipode}
    \xymatrix@C=25pt@R=15pt{
      \btA \otimes \bAt \ar[d]_{\Tl} \ar[r]^{\Sigma(S \oo S)} &       \ctA \otimes \cAt
      \ar[d]^{\rT} \\
      \AlA \ar[r]_{\Sigma (S \oo S)} & \ArA, }
  \quad
    \xymatrix@C=25pt@R=15pt{
      \AcA \ar[d]_{\lT} \ar[r]^{\Sigma (S \oo S)} &       \AbA
      \ar[d]^{\Tr} \\
      \ArA \ar[r]_{\Sigma (S \oo S)} & \AlA,  } 
\end{gather}
where $\Sigma$ denotes the flip maps on various tensor products; see Theorem 6.8, Proposition 6.11 and Proposition 6.12 in \cite{timmermann:regular}. Furthermore, by Corollary 5.12 in \cite{timmermann:regular},
\begin{gather} \label{eq:antipode-counits}
   \ceps = t_{B} \circ \epsb\circ S \quad\text{and}\quad   \epsb =
  t_{C} \circ\ceps \circ S.
\end{gather}
We shall also use the following multiplicativity of the  counits, see (3.5) and (4.9) in \cite{timmermann:regular}:
\begin{align} \label{eq:counits-multiplicative}
  \begin{aligned}
    \ceps(ab)&=\ceps(a\ceps(b))=\ceps(at_{C}(\ceps(b))), \\
    \epsb(ab)&=\epsb(\epsb(a)b) = \epsb(t_{B}(\epsb(a))b)
  \end{aligned}
\end{align}
for all $a,b\in A$.

 Recall that an algebra $D$ is \emph{firm} if the multiplication map $D\otimes D\to D$ factorizes to an isomorphism $D_{D} \otimes {_{D}D} \to D$, and that a right module $M_{D}$ over an algebra $D$ is called  \emph{locally projective} \cite{zimmermann-huisgen} if for every finite number of elements $m_{1},\ldots,m_{k} \in M$, there exist finitely many $\upsilon_{i} \in\Hom(M_{D},D_{D})$ and $e_{i} \in \Hom(D_{D},M_{D})$ such that $m_{j} = \sum_{i} e_{i}(\upsilon_{i}(m_{j}))$ for all $j=1,\ldots,k$. The corresponding definition for left modules is obvious. 
\begin{definition} \label{definition:projective}
  We call a regular multiplier Hopf algebroid $(A,B,C,t_{B},t_{C},\Delta_{B},\Delta_{C})$ \emph{locally projective} if 
  the algebras $B$ and $C$ are firm and the modules $\bA,\Ab,\cA,\Ac$ are locally projective.
\end{definition}
Finally, we consider involutions.
\begin{definition} \label{definition:involution} A \emph{multiplier Hopf $*$-algebroid} is a regular multiplier Hopf algebroid $(A,B,C,t_{B},t_{C},\Delta_{B},\Delta_{C})$ with an involution on the underlying algebra $A$ such that
  \begin{enumerate}
  \item $B$ and $C$ are $^*$-subalgebras of $M(A)$;
  \item $ t_{B}\circ \ast \circ t_{C} \circ \ast =\id_{C}$ and $t_{C}\circ \ast \circ t_{B}
\circ \ast =\id_{B}$;
\item $\Delta_{C}(a^{*})(b^{*}
    \otimes c^{*}) = ((b\otimes c)\Delta_{B}(a))^{(-)^{*} \otimes (-)^{*}}$ for all $a,b,c\in A$.
\end{enumerate}
\end{definition}
Here, condition (2) ensures that the map
\begin{align*}
(-)^{*} \otimes (-)^{*}\colon  \AlA \to \ArA, \ a\otimes b \mapsto a^{*} \otimes b^{*}
\end{align*}
is well-defined.  If $\mathcal{A}$ is a multiplier Hopf $^{*}$-algebroid as above, then
  \begin{align} \label{eq:counit-antipode-involution}
   \epsb\circ * &= *\circ
  S_{C}\circ \ceps, & \ceps\circ *&=*\circ S_{B}\circ \epsb, & 
  S\circ *\circ S \circ *&=\id_{A};
  \end{align}
see Proposition 6.2 in \cite{timmermann:regular}.

\subsection{Measured multiplier Hopf algebroids}

Let $\mathcal{A}=(A,B,C,t_{B},t_{C},\Delta_{B},\Delta_{C})$ be a regular multiplier Hopf algebroid with left counit $\epsc$, right counit $\beps$ and antipode $S$.

Then the following conditions on a bimodule map $\bpsib \in \Hom(\bAb,\bBb)$ are equivalent:
 \begin{itemize}
    \item[(RI1)] $(S_{B} \circ \bpsib \odot \iota)(\Delta_{C}(a)(1 \otimes     b)) = \bpsib(a)b$ for all $a,b\in A$,
    \item[(RI2)] $(\bpsib \odot \iota)((1 \otimes     b)\Delta_{B}(a)) = b\bpsib(a)$  for all $a,b\in A$,
    \item[(RI3)] the following diagram commutes, where $\Sigma \colon \Ac\oo \cA \to \btA\oo \bAt$ denotes the flip:
      \begin{align} \label{dg:strong-invariance-right}
        \xymatrix@R=12pt@C=35 pt{ \AcA \ar[r]^{\lT} \ar[d]_{\Tl\Sigma}
          & \ArA
          \ar[d]^{S\circ (\bpsib \odot \iota)}\\
          \AlA \ar[r]_(0.55){(S_{B}\circ \bpsib)\odot \iota} &
          A; }        
      \end{align}
\end{itemize}
see \cite[Proposition 3.1]{timmermann:integration}, Likewise, for every $\cphic \in \Hom(\cAc,\cCc)$, the following conditions are equivalent:
    \begin{itemize}
    \item[(LI1)] $(\iota \odot     S_{C} \circ\cphic)(\Delta_{C}(b)(a\otimes 1))=\cphic(b)a$ for all $a,b\in A$,
    \item[(LI2)] $(\iota \odot \cphic)((a\otimes   1)\Delta_{B}(b)) = a\cphic(b)$  for all $a,b\in A$,
    \item[(LI3)] the following diagram commutes, where $\Sigma \colon \Ab\oo \bA \to \ctA\oo \cAt$ denotes the flip:
      \begin{align} \label{dg:strong-invariance-left}
        \xymatrix@R=12pt@C=35pt{  \AbA \ar[r]^{\Tr}
          \ar[d]_(0.55){\rT\Sigma} & \AlA \ar[d]^(0.55){S\circ (\iota
            \odot \cphic)}\\
      \ArA \ar[r]_(0.55){\iota\odot  (S_{C}\circ \cphic)} & A. }       
      \end{align}
    \end{itemize}
Here, the  slice maps $(S_{B} \circ \bpsib)\odot \id$, $ \bpsib\odot \id$ and $\id \odot \cphic$, $\id \odot (S_{C} \circ \cphic)$ are defined similarly as in the case of  left or right counits.
\begin{definition}
  A \emph{measured regular multiplier Hopf algebroid} consists of a regular multiplier Hopf algebroid
\begin{align*}
  \mathcal{A}=(A,B,C,t_{B},t_{C},\Delta_{B},\Delta_{C}),
\end{align*}
 faithful functionals $\mu_{B}$ on $B$ and $\mu_{C}$ on $C$, and maps $\bpsib \in \Hom(\bAb,\bBb)$ and $\cphic \in \Hom(\cAc,\cCc)$ satisfying the following conditions:
\begin{enumerate}
\item  $\mu_{B} \circ t_{C} = \mu_{C}$, $\mu_{C} \circ t_{B} = \mu_{B}$, and $\mu_{B} \circ \epsb = \mu_{C} \circ \ceps$;
\item  (RI1)--(RI3) hold for $\bpsib$ and (LI1)--(LI3) hold for $\cphic$;
\item the compositions $\psi:=\mu_{B} \circ \bpsib$ and $\phi:=\mu_{C}\circ \cphic$ are faithful,
\item there exist surjective maps  $\bphi 
  \in \dbA$, $\phib \in \dAb$,  $\cpsi \in \dcA$, $\psic
    \in \dAc$ such that
  \begin{align} \label{eq:quasi-invariance} 
\psi = \mu_{C} \circ \cpsi = \mu_{C} \circ   \psic \quad{and} \quad
\phi =
    \mu_{B} \circ \bphi = \mu_{B} \circ \phib.
  \end{align}
\end{enumerate}
We call $(\mu_{B},\mu_{C})$ the  \emph{base weight}, $\bpsib$ the \emph{partial right integral}, $\cphic$ the \emph{partial left integral}, and the compositions $\psi$ and $\phi$ the \emph{total right} and the \emph{total left integral}, respectively. 

We call a  measured regular multiplier Hopf algebroid as above a \emph{measured multiplier Hopf $*$-algebroid} if  $\mathcal{A}$ is a multiplier Hopf $*$-algebroid and the functionals $\mu_{B},\mu_{C}$ and the total integrals $\phi,\psi$ are positive.
\end{definition}
\begin{remark} \label{remark:counits-surjective}
  As a consequence of (4) is that the counits $\ceps$ and $\epsb$ of
  $\mathcal{A}$ are automatically surjective because then
  \begin{align*}
    \ceps(A) = \ceps(\cphic(A)A) = \cphic(A)\ceps(A) = \cphic(A\ceps(A)) =\cphic(A)=C
  \end{align*}
by  \eqref{eq:left-counit} and similarly $\beps(A)=\bpsib(A)=B$.
\end{remark}
 A key property of total integrals is that they admit modular automorphisms:
\begin{theorem} \label{theorem:modular-automorphism}
 Let $(\mathcal{A},\mu_{B},\mu_{C},\bpsib,\cphic)$ be a measured regular multiplier Hopf algebroid. 
 Then the functionals $\mu_{B}$ and $\mu_{C}$ and the total   integrals $ \phi$ and $\psi$ admit modular automorphisms
   $\sigma_{B},\sigma_{C},\sigma^{\phi},\sigma^{\psi}$, respectively,   and
   \begin{align} \label{eq:modular-automorphism}
     \begin{aligned}
       \sigma_{C} = \sigma^{\phi}|_{C} &=  S_{B}S_{C}, &
       \sigma_{B}= \sigma^{\psi}|_{B} &=S_{B}^{-1}S_{C}^{-1},\\
 \Delta_{C} \circ
       \sigma^{\phi} &= (S^{2} \otimes \sigma^{\phi}) \circ
       \Delta_{C}, & \Delta_{C} \circ
       \sigma^{\psi} &= (\sigma^{\psi} \otimes S^{-2}) \circ
       \Delta_{C}, \\ \Delta_{B} \circ \sigma^{\phi} &= (S^{2} \otimes
       \sigma^{\phi})
       \circ \Delta_{B}, 
  & \Delta_{B} \circ \sigma^{\psi} &= (\sigma^{\psi}
       \otimes S^{-2}) \circ \Delta_{B}. 
     \end{aligned}
   \end{align}
If $\mathcal{A}$ is locally projective, then
 \begin{align} \label{eq:modular-automorphism-base}
   \begin{aligned}
     \sigma^{\phi}(B)&=B, & \mu_{B} \circ \sigma^{\phi}|_{B} &= \mu_{B}, & \sigma^{\phi}|_{B} \circ \bphi &= \sigma_{B} \circ \phi_{B}, \\
     \sigma^{\psi}(C) &= C, & \mu_{C}\circ \sigma^{\psi}|_{C} &= \mu_{C}, &
     \sigma^{\psi}|_{C} \circ \cpsi &= \sigma_{C} \circ \psi_{C}.
   \end{aligned}
 \end{align}
 \end{theorem}
 \begin{proof}
   The equations \eqref{eq:modular-automorphism} follow from \cite[Proposition 5.1, Theorem 6.2]{timmermann:integration}. We suppose that $\mathcal{A}$ is locally projective and prove the equations in \eqref{eq:modular-automorphism-base} concerning $\sigma^{\phi}$. By
 \cite[Theorem 6.2]{timmermann:integration} again, $\sigma^{\phi}(M(B))=M(B)$ and
\begin{align} \label{eq:bphi-phib}
 \phib(xa) &= (S^{2} \circ \sigma^{\phi})(x)\phib(a), &
  \bphi(ax) &= \bphi(a) (\sigma^{\phi}\circ S^{2})^{-1}(x)
\end{align}
for all $x\in B$ and $a\in A$, and hence $\phib(a)=0$ if and only if $\phi(BaB)=0$ if and only if $\bphi(a)=0$.  As $\phib$ and $\bphi$ are surjective, we can conclude existence of a linear bijection $\theta$ satisfying $\phib = \theta \circ \bphi$.  Then $\phib(xa) = \theta(x)\phib(a)$ for all $x\in B$, $a\in A$ and hence, by \eqref{eq:bphi-phib}, $\theta = S^{2} \circ \sigma^{\phi}|_{B}$. In particular, $\sigma^{\phi}(B)=B$.  The relation $\mu_{B} \circ \theta \circ \phib = \mu_{B} \circ \bphi = \mu_{B} \circ \phib$ and surjectivity of $\phib$ imply $\mu_{B} \circ \theta = \mu_{B}$. \end{proof} 
One also has a modular element relating the left and the right integral:
\begin{theorem}[{\cite[Corollary 6.1]{timmermann:integration}}]\label{theorem:modular-element}
 Let $(\mathcal{A},\mu_{B},\mu_{C},\bpsib,\cphic)$ be a measured regular multiplier Hopf algebroid.  Then there exist invertible multipliers $\delta,\delta^{+},\delta^{-}\in M(A)$   such that $\psi=\delta\cdot \phi$, $\phi\circ S = \delta^{+} \cdot \phi$ and $\phi \circ S^{-1} = \phi \cdot \delta^{-}$.
\end{theorem}

 \subsection{Factorizable functionals on bimodules}

Many   constructions in this article  become  more transparent when phrased in terms of the
factorizable functionals on bimodules introduced in \cite{timmermann:integration}. Let us recall their  definition  and main properties.

  Let $B$ and $C$ be idempotent algebras with faithful functionals  $\mu_{B}$ and $\mu_{C}$, respectively, and let $M$ be an idempotent  $B$-$C$-bimodule.  Note that $B$ and $C$ are automatically non-degenerate.
\begin{definition}
We call a functional $\omega \in \dual{M}$  \emph{factorizable} (with respect to $\mu_{B}$ and $\mu_{C}$) if  there exist maps $\bomega \in \Hom(_{B}M,{_{B}B})$ and $\omegac \in
  \Hom(M_{C},C_{C})$ such that
  \begin{align}
    \label{eq:factorisation}
    \mu_{B} \circ \bomega = \omega = \mu_{C} \circ \omegac.
  \end{align}
We denote by
\begin{align*}
  (_{B}M_{C})^{\sqcup} \subseteq \dual{M}
\end{align*}
 the subspace of all such factorizable functionals.  
\end{definition}
\begin{remarks}
  \begin{enumerate}
  \item  A functional $\omega \in M^{\vee}$  is factorizable if and only if there exist linear maps $\bomega \colon M\to B$ and $\omegac \colon M \to C$ such that for all $x\in B$, $y\in C$, $m\in M$,
\begin{align*}
  \omega(xm) = \mu_{B}(x\bomega(m)) \quad\text{and} \quad
\omega(my) = \mu_{C}(\omegac(m)y)
\end{align*}
\item If $\omega \in \dual{M}$ is factorizable, the maps $\bomega$ and $\omegac$ are uniquely determined.
\item The assignment ${_{B}M_{C}} \mapsto (_{B}M_{C})^{\sqcup}$ evidently is functorial.
  \end{enumerate}
\end{remarks}
The key property of factorizable functionals is that one can form slice maps and tensor products for such functionals   as follows.  

Let $D$ be a non-degenerate, idempotent algebra with a faithful
functional $\mu_{D}$ and let $_{C}N_{D}$ be an idempotent
$C$-$D$-bimodule. Consider the balanced tensor product $_{B}M_{C}\otimes
{_{C}N_{D}}$ and let $\upsilon \in (_{B}M_{C})^{\sqcup}$ and $\omega\in (_{C}N_{D})^{\sqcup}$.
Then the following slice maps 
    \begin{align*}
      \upsilonc \odot \id & \colon M_{C}\otimes {_{C}N_{D}} \to N_{D}, &    
                                                                     \id \odot \comega &\colon {_{B}M_{C}} \otimes
                                                                                         {_{C}N} \to {_{B}M}, 
    \end{align*}
are morphisms of modules, and
   by composition with $\omega$ or $\upsilon$, respectively, we obtain a factorizable functional
    \begin{align}\label{eq:factorisation-tensor-1}
      \upsilon \underset{\mu_{C}}{\otimes} \omega :=\upsilon \circ (\id \odot \comega) = \omega \circ (\upsilonc \odot \id) \in ({_{B}M_{C}}\otimes {_{C}N_{D}})^{\sqcup}
    \end{align}
which satisfies
\begin{align}
  \label{eq:factorisation-tensor}
  (\upsilon \underset{\mu_{C}}{\otimes} \omega)(m\otimes n) =\mu_{C}(\upsilonc(m)\comega(n))
\end{align}
for all $m\in M$ and $n\in N$; see \cite[Lemma 3.4.2]{timmermann:integration}.


 Suppose that $\mu_{B}$ and $\mu_{C}$ admit modular automorphisms $\sigma_{B}$ and $\sigma_{C}$, respectively.  Then  $(_{B}M_{C})^{\sqcup}$ is an $M(C)$-$M(B)$-sub-bimodule of $M^{\vee}$ and
  \begin{align} \label{eq:factorisation-dual-bimodule-factorisation}
_ {B}   (y \cdot \omega\cdot  x)(m) &= \bomega(my)\sigma_{B}(x),  &
(y\cdot \omega \cdot x)_{C}(m) &= \sigma_{C}^{-1}(y)\omegac(m x)
 \end{align}
 for all $\omega \in (_{B}M_{C})^{\sqcup}$, $x\in B$, $y\in C$;
 see \cite[Lemma 3.4.4]{timmermann:integration}. 

\section{Duality}
\label{section:dual}
Throughout this section, we fix a  measured regular multiplier Hopf algebroid
\begin{align*}
  (\mathcal{A},\mu_{B},\mu_{C},\bpsib,\cphic),  \text{ where } \mathcal{A}=(A,B,C,t_{B},t_{C},\Delta_{B},\Delta_{C}),
\end{align*}
which we assume to be locally projective from subsection \ref{subsection:dual-comultiplication} on, and construct, step by step, a dual measured regular multiplier Hopf algebroid
 \begin{align*}
  (\widehat{\mathcal{A}},\mu_{C}, \mu_{B},\hbpsib,\hcphic),  \text{ where } \widehat{\mathcal{A}}=(\hat A,C,B,t^{-1}_{ B}, t^{-1}_{ C},\hat \Delta_{C},\hat \Delta_{B}).
    \end{align*}
    As before, $\phi=\mu_{C} \circ \cphic$ and $\psi=\mu_{B}\circ \bpsib$ denote the total left and right integrals.

Let us briefly outline the construction.
As a vector space, $\hat{A}$ is just the subspace
\begin{align*}
  A\cdot \phi = \phi \cdot A = A\cdot \psi = \psi\cdot A
\end{align*}
of the linear dual $A^{\vee}$, and the multiplication is given by
\begin{align*}
  \upsilon \ast \omega &= (\upsilon \underset{\mu_{C}}{\otimes} \omega)\circ \Delta_{C} = (\upsilon \underset{\mu_{B}}{\otimes} \omega) \circ \Delta_{B}.
\end{align*}
The details are given in subsection \ref{subsection:dual-bimodule} and \ref{subsection:dual-algebra}.

The dual left comultiplication $\hat{\Delta}_{B}$ and the dual right comultiplication $\hat{\Delta}_{C}$ will be constructed in terms of the associated maps
\begin{align*}
  \hTl &\colon \upsilon \oo \omega \mapsto  \hat\Delta_{ B}(\omega)(\upsilon \otimes 1), & \hTr &\colon
  \upsilon \oo \omega \mapsto \hat\Delta_{ B}(\upsilon)(1 \otimes
  \omega), \\
  \hlT &\colon \upsilon \oo \omega \mapsto (\upsilon \oo
  1)\hat\Delta_{ C}(\omega), &
  \hrT &\colon \upsilon \oo \omega \mapsto (1 \oo
  \omega)\hat\Delta_{ C}(\upsilon),
\end{align*}
which dualise the corresponding canonical maps of $\mathcal{A}$ similarly as in \eqref{eq:dual-mh}. To do so, we need to define various pairings between balanced tensor squares of $\hat{A}$ and balanced tensor squares of $A$.
This is done in subsection \ref{subsection:dual-comultiplication}.

Finally, we show in subsection \ref{subsection:dual-integrals} that the maps
\begin{align*}
  {_{C}\hat{\psi}_{C}} \colon a\cdot \phi \mapsto \ceps(a) \quad \text{and} \quad {_{B}\hat{
\phi}_{B}}\colon \psi\cdot a \mapsto \epsb(a)
\end{align*}
are partial integrals on $\widehat{\mathcal{A}}$, so that the latter becomes a measured regular multiplier Hopf algebroid again, and that the canonical map $A \mapsto (\hat{A})^{\vee}$ is an isomorphism from $\mathcal{A}$ to the dual of $\widehat{\mathcal{A}}$.

\subsection{The dual  as a bimodule}
\label{subsection:dual-bimodule}

The underlying algebra $\hat A$ of the dual $\widehat{\mathcal{A}}$ was introduced already in \cite[Section 6.2]{timmermann:integration}.  Before we recall its construction, let us look at the special case where $\mathcal{A}$ is unital.

If the algebras and maps comprising $\mathcal{A}$ are unital, we can regard the left and the right comultiplication as bimodule maps
\begin{align*}
  \Delta_{B} \colon \bAt_{B} \to \bAt_{B} \otimes \bAt_{B} \quad \text{and} \quad
  \Delta_{C} \colon \csA^{C} \to \csA^{C} \otimes \csA^{C},
\end{align*}
and equip the dual modules $(^{B}A)^{\vee}$, $(A_{B})^{\vee}$, $(_{C}A)^{\vee}$ and $(A^{C})^{\vee}$ with convolution products similarly as
in  \eqref{eq:dual-modules-convolution}. 
Now, the base weight $(\mu_{B},\mu_{C})$ allows us to identify \emph{one common subalgebra} inside these four convolution algebras   as follows.  On the spaces $(\bAt_{B})^{\sqcup}$ and $(\csA^{C})^{\sqcup}$ of factorizable functionals, which can be identified with subspaces of $(^{B}A)^{\vee}$ and $(A_{B})^{\vee}$ or  $(_{C}A)^{\vee}$ and $(A^{C})^{\vee}$, respectively, the  multiplications take the form
\begin{align*}
  \upsilon\cdot_{B} \omega &:= (\upsilon \underset{\mu_{B}}{\otimes} \omega) \circ \Delta_{B} \quad   \text{and} \quad
  \upsilon \cdot_{C} \omega := (\upsilon \underset{\mu_{C}}{\otimes} \omega) \circ \Delta_{C},
\end{align*}
where we use the relative tensor product of factorizable functionals defined in \eqref{eq:factorisation-tensor} and functoriality. Moreover, on the intersection,
\begin{align} \label{eq:asqcup}
  A^{\sqcup} &:=(\bAt_{B})^{\sqcup} \cap (\csA^{C})^{\sqcup} \subseteq \dual{A},
\end{align}
the two multiplications coincide, because the counit functional $\varepsilon$, which lies in $A^{\sqcup}$ by \cite[Remark 5.1 (3)]{timmermann:integration},  is a unit for both multiplications and the   mixed coassociativity condition implies  that for all 
$\upsilon,\omega \in A^{\sqcup}$,
\begin{align*}
  \upsilon \cdot_{B} \omega = \upsilon \cdot_{B} (\varepsilon \cdot_{C} \omega)  = (\upsilon \cdot_{B} \varepsilon) \cdot_{C} \omega = \upsilon \cdot_{C} \omega.
\end{align*}

Let us now turn to the general, non-unital case.
 Define $A^{\sqcup} \subseteq \dual{A}$ as in \eqref{eq:asqcup}. Then
 \begin{align} \label{eq:asqcup-intersections}
A^{\sqcup} = (_{B}A^{B})^{\sqcup} \cap (^{C}A_{C})^{\sqcup} 
=
  (_{B}A_{B})^{\sqcup} \cap (_{C}A_{C})^{\sqcup} =  (^{B}A_{B})^{\sqcup} \cap (_{C}A^{C})^{\sqcup}.
\end{align}
Indeed, a functional   $\omega$ on $A$ lies in $A^{\sqcup}$ if and only if it  there exist  module maps
\begin{align*}
  \bomega &\in \dual{(\bA)}, &
  \omegab &\in \dual{(\Ab)}, &
\comega &\in \dual{(\cA)}, & \omegac &\in \dual{(\Ac)}, \\
  ^{B}\omega &\in \dual{(\Asb)}, &
  \omega^{B} &\dual{(\sbA)}, &
  ^{C}\omega &\in \dual{(\Asc)}, & \omega^{C} &\in \dual{(\scA)}
\end{align*}
such that the composition with $\mu_{B}$ or $\mu_{C}$, respectively, gives $\omega$ again.  In that case,
\begin{align*}
  ^{B}\omega &= S_{C}\circ \omegac, &
  \omega^{B} &= S_{C} \circ \comega, &
  ^{C}\omega &= S_{B} \circ \omegab, &
  \omega^{C} &= S_{B} \circ \bomega.
\end{align*}

In the non-unital case, the space $A^{\sqcup}$ may be too large  to carry a convolution product. Instead, the underlying vector space of the dual algebra will be the  subspace 
\begin{align}
  \label{eq:dual-space}
\hat A:= A\cdot \phi = \phi \cdot A = A\cdot \psi = \psi \cdot A \subseteq A^{\sqcup}.
\end{align}
 The four spaces in the middle are equal by Theorem \ref{theorem:modular-automorphism} and Theorem \ref{theorem:modular-element}, and the inclusion $\hat A \subseteq A^{\sqcup}$ can  easily be seen using the  factorizations $\phi = \mu_{C} \circ \cphic$ and $\psi=\mu_{B} \circ \bpsib$.

By \cite[Lemma 3.6]{timmermann:integration},  the formulas
   \begin{align}
    \label{eq:dual-module}
                                 x\upsilon &:= \upsilon \cdot  t_{B}(x), & \upsilon x &:= \upsilon \cdot x, &y\upsilon &:= y \cdot \upsilon, & \upsilon y &:= t_{C}(y)\cdot \upsilon,
   \end{align}
  where $\upsilon\in A^{\sqcup}$ and $x\in B$, $y\in C$, turn $A^{\sqcup}$ into a $B$-bimodule and a $C$-bimodule. As such, it is  non-degenerate because $BA=AB=CA=AC=A$.

\begin{lemma} \label{lemma:dual-bimodule}
The subspace $\hat A \subseteq A^{\sqcup}$ is a $C$-sub-bimodule and a $B$-sub-bimodule. As such, it is idempotent and faithful. If $\mathcal{A}$ is locally projective, then also $\hat A$ is locally projective as a left and as a right module over $B$ and over $C$, respectively.
\end{lemma}
\begin{proof}
  To see that $\hat A$ is a sub-bimodule, note that
\begin{align} \label{eq:hata-module}
  \begin{aligned}
    y(a\cdot \phi) &= (ya) \cdot \phi, & (a\cdot \phi)y &= (t_{C}(y)a)
    \cdot \phi, \\ x(\phi \cdot a) &= \phi \cdot (at_{B}(x)), & (\phi
    \cdot a)x &= \phi \cdot (ax)
  \end{aligned}
\end{align}
for all $a \in A$, $x\in B$, $y\in C$.  Idempotence of $\hat A$   follows easily  from the relation $AB=BA=AC=CA=A$, and   faithfulness of  $\hat A$ from faithfulness of $\phi$ and $\psi$. To see that $\hat A$ is locally projective as a module, note that  \eqref{eq:hata-module} allows us to identify $\hat A$, regarded as a left or right module over $C$ or over $B$, respectively, with the  locally projective modules $_{C}A$, $A^{C}$, $^{B}A$ and $A_{B}$, respectively.
\end{proof}

 We write
  \begin{align*}
    _{B}(A^{\sqcup}), \quad (A^{\sqcup})_{B}, \quad _{C}(A^{\sqcup}), \quad (A^{\sqcup})_{C} \quad \text{and} \quad
_{B}\hat{A}, \quad \hat{A}_{B}, \quad _{C}\hat{A}, \quad \hat{A}_{C}
  \end{align*}
  when we regard $A^{\sqcup}$ or $\hat{A}$, respectively, as a left or right module over $B$ or over $C$ as in \eqref{eq:dual-module}.
Using similar notation as for $A$, we also   write
\begin{align*}
  (A^{\sqcup})^{{B}}, \quad ^{{B}}(A^{\sqcup}), \quad (A^{\sqcup})^{{C}}, \quad ^{{C}}(A^{\sqcup}) \quad \text{and} \quad
\hat{A}^{{B}}, \quad ^{{B}}\hat{A}, \quad \hat{A}^{{C}}, \quad ^{{C}}\hat{A}
\end{align*}
 when we regard $A^{\sqcup}$ or $\hat{A}$, respectively, as a left or right module over ${B}$ or ${C}$ such that  the product of an element $\omega \in A^{\sqcup}$ with $x\in B$ or $y\in C$ is
 \begin{align*}
 t_{C}^{-1} (x)\omega, \quad \omega t_{C}^{-1}(x), \quad
t_{B}^{-1}(y)\omega, \quad \omega t_{B}^{-1}(y),
 \end{align*}
respectively.

\subsection{The dual algebra}
\label{subsection:dual-algebra}
The spaces $A^{\sqcup}$ and $\hat A$ act by convolution on $A$ as follows.
\begin{proposition}[{\cite[Proposition 6.1]{timmermann:integration}}] \label{proposition:act}
 Let
  $(\mathcal{A},\mu_{B},\mu_{C},\bpsib,\cphic)$ be a measured regular multiplier Hopf
  algebroid and let $\omega \in A^{\sqcup}$ and $a\in A$. Then there exist unique multipliers $\omega \actright a \in M(A)$ and $a\actleft \omega \in M(A)$ such that for all $b\in A$,
\begin{align*}
    b(\omega \actright a)&:= (\id  \odot {^{B}\omega})((b \otimes 1)\Delta_{B}(a)), &
  (\omega \actright a)b&:= (\id \odot {_{C}\omega})(\Delta_{C}(a)(b\otimes 1)),  \\
  b(a\actleft \omega) &:= (\omegab \odot \id)((1 \otimes b)\Delta_{B}(a)), &
  (a \actleft \omega)b 
&:= (\omega^{C} \odot \id)(\Delta_{C}(a)(1 \otimes b)).
\end{align*} 
\end{proposition}
The point here is to prove that for all $a,b,c\in A$ and $\omega\in A^{\sqcup}$,
\begin{align*}
  b((\omega \actright a)c) = (b(\omega \actright a))c \quad \text{and} \quad b((a \actleft \omega)c) = (b(a\actleft \omega))c,
\end{align*}
  which can be done by  help of the counit functional $\varepsilon$ and mixed coassociativity.
\begin{lemma} \label{lemma:act}
  \begin{enumerate}
  \item Let $\upsilon,\omega \in \hat A$ and $a\in A$. Then
$\omega \actright a$ and $a\actleft \upsilon$ lie in $A$ and     
\begin{align}
      \label{eq:1}
 \omega \actright (a \actleft \upsilon) &= (\omega \actright a) \actleft \upsilon, & \varepsilon(\omega \actright a)     &= \omega(a), & \varepsilon(a \actleft \upsilon) &= \upsilon(a).
    \end{align}
  \item For all $a,b \in A$, 
    \begin{align} \label{eq:strong-invariance}
      ((\phi \cdot a) \actright b) = S((b\cdot \phi)\actright a) \quad\text{and} \quad b\actleft (a\cdot \psi) = S(a \actleft (\psi \cdot b)).
    \end{align}
\item For all $\omega \in A^{\sqcup}$, $a\in A$, $x,x'\in B$, $y,y'\in
    C$,
\begin{align}
\label{eq:act-bimodule}
\begin{aligned}
  \omega \actright (yay') &= y(\omega \actright a)y', & (xy\omega
  x'y') \actright a &= x(\omega \actright
  (x'at_{C}(y')))t_{C}(y),  \\
  (xax') \actleft \omega &= x(a \actleft \omega)x', & a \actleft
  (xy\omega x'y') &= t_{B}(x') ((t_{B}(x)ay) \actleft \omega)y'
\end{aligned}
\end{align}
  \item For all $\omega \in A^{\sqcup}$ and  $a\in A$,
    \begin{align} \label{eq:act-antipode} 
\omega \actright S(a) &= S(a\actleft (\omega \circ S)), & S(a) \actleft \omega &= S((\omega \circ S) \actright a).
\end{align}
\item If $(\mathcal{A},\mu_{B},\mu_{C},\bpsib,\cphic)$ is a measured multiplier Hopf $*$-algebroid, then
\begin{align}
  \label{eq:act-involution}
  (\omega \actright a)^{*} &= (* \circ \omega \circ *) \actright
                            (a^{*}), & (a \actleft \omega)^{*} &= (a^{*}) \actleft (* \circ
                                                               \omega \circ *).
\end{align}
\end{enumerate}
\end{lemma}
\begin{proof}
  The verification is straightforward; use
  the inclusions $\Delta_{C}(A)(A\otimes 1) \subseteq \AlA$ and $\Delta_{C}(A)(1\otimes A) \subseteq \AlA$, coassociativity and the counit axioms for (1), \eqref{dg:strong-invariance-right} and \eqref{dg:strong-invariance-left} for (2), the relations \eqref{eq:left-comultiplication-bilinear} and \eqref{eq:right-comultiplication-bilinear} for (3), diagram \eqref{dg:galois-antipode} for (4) and Definition \ref{definition:involution} (3) for (5). We therefore only carry out the proof of the second equation in \eqref{eq:act-bimodule} and leave the remaining assertions to the reader.  By \eqref{eq:factorisation-dual-bimodule-factorisation}, the functional $\omega'=xy\omega x'y'$ satisfies
\begin{align*}
  _{C}\omega'(c) &= \comega(x'ayt_{C}(y'))t_{C}^{-1}(x), &
  \omegab'(c) &= t_{B}^{-1}(y')\omegab(x't_{B}(x)cy) 
\end{align*}
for all $c\in A$, and hence, taking into account \eqref{eq:left-comultiplication-bilinear},
\begin{align*}
  (\omega' \actright a)b &= (\id \odot \comega')(\Delta_{C}(a)(b \otimes 1)) \\
  &= (\id \odot \comega)((x \otimes x')\Delta_{C}(a)(b\otimes yt_{C}(y'))) \\ &= x (\id \odot \comega)(\Delta_{C}(x'at_{C}(y'))(t_{C}(y)b)) = x(\omega \actright (x'at_{C}(y')))t_{C}(y)b
\end{align*}
for all $b\in A$. The other equations in  \eqref{eq:act-bimodule} follows similarly.
\end{proof}
 The space $\hat A$ becomes an algebra as follows.
\begin{theorem}[{\cite[Theorem 6.3]{timmermann:integration}}] \label{theorem:dual-algebra} Let
  $(\mathcal{A},\mu_{B},\mu_{C},\bpsib,\cphic)$ be a measured regular multiplier Hopf
  algebroid. Then:
  \begin{enumerate}
  \item The space $\hat{A}$ is  a non-degenerate, idempotent algebra with respect to the multiplication defined by $\upsilon\omega := \upsilon \circ (\omega \actright -) = \omega  \circ
(-\actleft \upsilon)$.
\item The space $A$ is a non-degenerate, idempotent and faithful $\hat A$-bimodule with respect to the left and right multiplication given by $\omega \otimes a \mapsto \omega \actright a$ and $a\otimes \omega \mapsto a \actleft \omega$.
\item The space $A^{\sqcup}$ is a non-degenerate and faithful $\hat A$-bimodule with respect to the left and right multiplication given by $\omega \otimes \upsilon \mapsto \upsilon(-\actleft \omega)$ and $\upsilon \otimes \omega \mapsto \upsilon(\omega \actright -)$.
  \end{enumerate}
\end{theorem}
Given the  preparations above, the proof is straightforward. 
For example, to see that  $\upsilon \omega \in \hat A$ whenever $\upsilon,\omega \in \hat A$, write $\omega=b\cdot \phi$ with $b\in A$ and use Lemma \ref{lemma:act} (1) and (2) to see that
\begin{align*}
      (\upsilon \omega) (a) = \upsilon((b \cdot \phi) \actright a) &=
    \upsilon S^{-1}((\phi \cdot a) \actright b) =  \phi(a (b \actleft \upsilon S^{-1}))
\end{align*}
for all $a\in A$, whence, by  Lemma \ref{lemma:act} (4), 
\begin{align} \label{eq:product-ahat}
  \upsilon (b\cdot \phi) = c \cdot \phi, \quad \text{where } c= (b \actleft \upsilon S^{-1}) = S(\upsilon \actright S^{-1}(b)).
\end{align}
A similar argument shows  that
\begin{align} \label{eq:product-ahat-2}
  (\psi\cdot b)\upsilon = \psi \cdot c, \quad \text{where } c = \upsilon S^{-1} \actright b.
\end{align}

 The multiplier algebra $M(\hat A)$  can be identified with a subspace of $A^{\sqcup}$ similarly as in the case of  multiplier Hopf algebras \cite{kustermans:analytic-2} as follows. Since $A$ is idempotent and non-degenerate as an $\hat A$-bimodule, we can regard it as an $M(\hat A)$-bimodule such that $T \actright (\omega \actright a)=(T\omega \actright a)$ and $(a\actleft \omega) \actleft T = a \actleft (\omega T)$ for all $a\in A$, $\omega \in \hat A$ and $T\in M(\hat A)$.
\begin{proposition} \label{proposition:dual-multipliers} Let
  $(\mathcal{A},\mu_{B},\mu_{C},\bpsib,\cphic)$ be a measured regular multiplier Hopf
  algebroid with dual algebra $\hat A$. Then every element of
  \begin{align*}
     A^{\sqcup}_{0} :=\{ \upsilon \in A^{\sqcup} : \upsilon \actright a \in A \text{ and } a\actleft \upsilon  \in A \text{ for all  } a \in A\}
  \end{align*}
  determines a multiplier of $\hat A$ and the map $A^{\sqcup}_{0}\to M(\hat A)$ is a linear isomorphism.
\end{proposition}
\begin{proof}
 Let $\upsilon \in A^{\sqcup}_{0}$, $\omega\in \hat A$, and write $\omega=b\cdot \phi$ with $b \in A$. Then the same calculation that lead to \eqref{eq:product-ahat} shows that the product $\upsilon \omega \in A^{\sqcup}$ is equal to $c\cdot \phi \in \hat A$ with $c=S(\upsilon\actright S^{-1}(b)) \in A$. A similar argument shows that $\omega\upsilon \in \hat A$.
  The resulting map $A^{\sqcup}_{0} \to M(\hat A)$ is well-defined and injective by  Theorem \ref{theorem:dual-algebra} (3), so we only need to show that this map  is surjective.  Let $T \in M(\hat A)$ and define $\upsilon,\upsilon'\colon A \to \C$ by
$\upsilon(a):=\varepsilon(T\actright a)$ and $\upsilon'(a):=\varepsilon(a\actleft T)$ for all $a \in A$. Then $\upsilon=\upsilon'$ because 
\begin{align*}
  \upsilon(\omega \actright a) = \varepsilon(T\omega \actright a) = (T\omega)(a) = \varepsilon(a \actleft T\omega) =  \omega(a \actleft T) = \varepsilon(\omega \actright a \actleft T) = \upsilon'(\omega \actright a)
\end{align*} 
for all $\omega \in \hat A$ and $a\in A$ by Lemma \ref{lemma:act} (1).  The relations $\varepsilon\in A^{\sqcup}$ and  \eqref{eq:act-bimodule} imply $\upsilon \in A^{\sqcup}$. Finally, the calculation above and the definition of the $\hat A$-bimodule structure of $A^{\sqcup}$ imply $(\upsilon\omega)(a) = \upsilon(\omega \actright a) = (T\omega)(a)$ for all $\omega \in \hat A$ and $a\in A$ so that $\upsilon\omega = T\omega$.
\end{proof}

The multiplication on $\hat A$ is compatible with the bimodule structures considered above:
\begin{lemma} 
 For all $\omega \in \hat{A}$, $\upsilon \in A^{\sqcup}$, $x\in B$ and $y\in C$, 
\begin{align} \label{eq:dual-bimodule-multiplier}
  (\omega x)\upsilon &= \omega(x\upsilon), &
  (\upsilon x)\omega &= \upsilon(x\omega), &
 (\omega y)\upsilon &= \omega(y\upsilon), &
  (\upsilon y)\omega &= \upsilon(y\omega) &
\end{align}
 In particular, we obtain embeddings of $B$ and $C$ into $M(\hat{A})$, and the images commute.
\end{lemma}
\begin{proof}
The relations \eqref{eq:dual-bimodule-multiplier} follow easily from \eqref{eq:act-bimodule}, for example,
 \begin{align*}
   ((\omega x)\upsilon)(a) &= \upsilon(a \actleft \omega x) = \upsilon(t_{B}(x)(a \actright \omega)) = (\omega(x\upsilon))(a), \\
  ((\upsilon x)\omega)(a) &=  (\upsilon x)(\omega \actright a) = \upsilon(x(\omega \actright a)) = \upsilon((x\omega) \actright a) = (\upsilon(x\omega))(a)
\end{align*}
for all $\upsilon \in A^{\sqcup}$, $\omega \in\hat A$, $a\in A$ and $x\in B$.
\end{proof}
Let us next consider the involutive case.
\begin{proposition} \label{proposition:dual-involution}
Let  $(\mathcal{A},\mu_{B},\mu_{C},\bpsib,\cphic)$ be a measured multiplier Hopf
$*$-algebroid.  Then $\hat A$ is a $*$-algebra with respect to the involution given by
\begin{align*}
  \omega^{*}(a):=\omega(S(a)^{*})^{*}
\end{align*}
 for all $\omega \in \hat A$ and $a \in A$, and the embeddings of $B$ and $C$ into $M(\hat A)$ become involutive.
\end{proposition}
\begin{proof}
Let $\omega,\upsilon \in \hat A$. Write  $\omega = a\cdot \phi$. Then  $(a  \cdot\phi)^{*}(b) = \phi(a^{*}S(b)) = (\phi \circ S)(bS^{-1}(a^{*})) = \phi(bS(a)^{*}\delta^{+})= (S(a)^{*}\delta{^{+}} \cdot \phi)(b)$ for all $b\in A$, where
$\delta^{+} \in M(A)$ is as in Theorem \ref{theorem:modular-element}, and hence $\omega^{*} \in \hat{ A}$.
Relation  \eqref{eq:counit-antipode-involution} implies that $(\omega^{*})^{*} =\omega$. To see that $(\upsilon^{*}\omega^{*})=\omega^{*}\upsilon^{*}$, note that for all $a\in A$,
\begin{align*}
  \omega^{*} \actright S(a)^{*} &= (* \circ \omega \circ * \circ S) \actright S(a)^{*} = S(a \actleft \omega)^{*}, & 
S(a)^{*} \actleft \upsilon^{*} &= S(\upsilon \actright a)^{*}
\end{align*}
 by \eqref{eq:act-antipode} and \eqref{eq:act-involution},
and hence $\omega^{*}\upsilon^{*} \actright S(a)^{*} = S(a \actleft \upsilon\omega)^{*} = (\upsilon\omega)^{*} \actright S(a)^{*}$.  Finally,
\begin{align*}
  (\omega^{*} x^{*})(a) = (\omega^{*})(ax^{*}) = \omega(S(x^{*})^{*}S(a)^{*}) = \omega(S^{-1}(x)S(a)^{*}) = (x\omega)^{*}(a),
\end{align*}
whence $\omega^{*}x^{*}=(x\omega)^{*}$ for all $x\in B$, and similarly $y^{*}\omega^{*}=(\omega y)^{*}$ for all $y\in C$.
\end{proof}

 By \cite[Lemma 6.2]{timmermann:integration}, we can also define an evaluation map $\dual{T} \in \dual{(\hat A)}$ for every  $T \in M(A)$ such that for all $a\in A$,
\begin{align*}
  \dual{T}(a\cdot \phi)  &= \phi(Ta), & \dual{T}(\phi \cdot a) &= \phi(aT), & \dual{T}(a\cdot \psi) &= \psi(Ta), & \dual{T}(\psi \cdot a) &= \psi(aT).
\end{align*}
We show that this functional lies in
\begin{align*}
   \hat{A}^{\sqcup} := ({_{C}\hat{A}_{C}}) \cap ({_{B}\hat{A}_{B}}).
\end{align*}
\begin{lemma} \label{lemma:eval}
Let $T \in M(A)$. Then the functional $\dual{T}$ lies in $(\hat A)^{\sqcup}$, and
         for all $a \in A$,
    \begin{align*}
      ({_{C}\dual{T}})(\phi \cdot a) &=  \sigma_{C}(\cphic(aT)), &
      (\dual{T}_{C})(\psi \cdot a) &= t_{B}(\bpsib(aT)), \\
      ({_{B}\dual{T}})(a \cdot \phi) &= t_{C}(\cphic(Ta)), &
      (\dual{T}_{B})(a \cdot \psi) &= \sigma^{-1}_{B}(\bpsib(Ta)).
    \end{align*}
\end{lemma}
\begin{proof}
  The functional $\dual{T}$ lies in $({_{C}\hat A_{C}})^{\sqcup}$ and the first two equations hold because
  \begin{align*}
    \dual{T}(y(\phi \cdot a))&= 
    \phi(aTy) =\mu_{C}(\cphic(aT)y) = 
    \mu_{C}(y\sigma_{C}(\cphic(aT))),
\\    \dual{T}((\psi \cdot a)y) &=  \psi(aTt_{C}(y))  \\ &=
 \mu_{B}(\bpsib(aT)t_{C}(y))
    = \mu_{B}(t_{B}^{-1}(y)\bpsib(aT)) = \mu_{C}(t_{B}(\bpsib(aT))y)
  \end{align*}
  for all $y\in C$ and $a\in A$. The remaining assertions follow similarly.
\end{proof}

 \subsection{The dual comultiplications}
\label{subsection:dual-comultiplication}
From now on, we suppose that \emph{$\mathcal{A}$  is locally projective}.

Consider the pairings
\begin{align} \label{eq:pairing-ala}
 (A^{\sqcup} \otimes A^{\sqcup}) \times (A^{C} \otimes \cA) &\to \C, &
(\upsilon \otimes \omega\mid a\otimes b) &= \upsilon(t_{C}(\comega(b))a) = \omega(\upsilon^{C}(a)b), \\ \label{eq:pairing-aca}
 (A^{\sqcup} \otimes A^{\sqcup}) \times (A_{C} \otimes \cA) &\to \C, &
  (\upsilon \otimes \omega\mid a\otimes b) &= \upsilon(a\comega(b)) = \omega(\upsilonc(a)b), \\ \label{eq:pairing-ara}
 (A^{\sqcup} \otimes A^{\sqcup}) \times (A_{B} \otimes \bAt) &\to \C, & 
(\upsilon \otimes \omega\mid a\otimes b) &= \upsilon(a^{B}\omega(b)) = \omega(bt_{B}(\upsilonb(a))), \\ \label{eq:pairing-aba}
(A^{\sqcup} \otimes A^{\sqcup}) \times (A_{B} \otimes \bsA) &\to \C, &
(\upsilon \otimes \omega\mid a\otimes b) &= \upsilon(a\bomega(b)) = \omega(\upsilonb(a)b).
\end{align}
Using   tensor products of factorizable functionals \eqref{eq:factorisation-tensor}, they take the simple form
   \begin{align} \label{eq:pairing}
     (\upsilon \otimes \omega \mid  a\otimes b):=(\upsilon \underset{\mu_{C}}{\otimes} \omega)(a \otimes b) \quad \text{or} \quad
     (\upsilon \otimes \omega \mid  a\otimes b):=(\upsilon \underset{\mu_{B}}{\otimes} \omega)(a \otimes b),
   \end{align}
respectively.   By Lemma \ref{lemma:eval},
 we can also regard elements of $A$ as factorizable functionals on $\hat A$ and hence define pairings between $A \otimes A$  on one side and $    \hAbA$, $  \hat A^{B} \otimes {_{B}\hat A}$,  $ \hAcA$ and $\hat A_{C} \otimes {_{\hat C}\hat A}$ on the other side.
\begin{lemma} \label{lemma:pairing}
  The formulas in \eqref{eq:pairing-ala}--\eqref{eq:pairing-aba} define non-degenerate pairings
\begin{align*}
    \hAbA \times  \AlA &\to \C, &
 \hAcA \times   \ArA  &\to \C, \\
  \hat A^{B} \otimes {_{B}\hat A} \times \AcA &\to \C, &
  \hat A_{C} \otimes {^{C}\hat A} \times \AbA &\to \C
\end{align*}
such that  for all $a,b\in A$ and $\upsilon,\omega \in \hat A$,
\begin{align} \label{eq:dual-pairing}
  (\upsilon \otimes \omega\mid a\otimes b) =  
(\dual{a} \underset{\mu_{B}}{\otimes} \dual{b})(\upsilon \otimes \omega) \quad \text{or} \quad
  (\upsilon \otimes \omega\mid a\otimes b) =  
(\dual{a} \underset{\mu_{C}}{\otimes} \dual{b})(\upsilon \otimes \omega),
\end{align}
respectively.
 \end{lemma}
 \begin{proof}
To see that   the pairings factorize as claimed,  we only need to prove \eqref{eq:dual-pairing}. We do so  for the first of the four cases; the remaining three cases are similar:
   \begin{align*} 
     (\dual{a} \underset{\mu_{B}}{\otimes} \dual{b})(\upsilon \otimes \omega) =
\mu_{B}(\dual{a}_{B}(\upsilon) \cdot {_{B}\dual{b}}(\omega))
&= \mu_{B}(\sigma_{B}^{-1}(\bupsilon(a))t_{C}(\comega(b))) \\
&= \mu_{B}(t_{C}(\comega(b))\bupsilon(a)) = \upsilon(t_{C}(\comega(b))a)
   \end{align*}
   by Lemma \ref{lemma:eval} and \eqref{eq:dual-module}.  To see that the pairings are non-degenerate, use the fact that elements of $\hat A$ separate the points of $A$, elements of $A$ separate the points of $\hat A$, relation \eqref{eq:dual-pairing} and the fact that all modules involved are locally projective by assumption on $\mathcal{A}$ and Lemma \ref{lemma:dual-bimodule}.
 \end{proof}

With the pairings obtained above, we embed various balanced tensor products of $\hat{A}$ dual spaces of certain balanced tensor products of $A$, and show that the duals of the canonical maps $\Tl,\Tr,\lT,\rT$ restrict to the balanced tensor products of $\hat{A}$.
 
Denote by $\End(A)$  the space of all linear endomorphisms of $A$.  
\begin{lemma} \label{lemma:compact}
  The four subspaces of $\End(A)$ spanned by all  maps of the form
 \begin{align*}
 a\mapsto b(a \actleft \upsilon), \quad a\mapsto b(\upsilon \actright a), \quad
a\mapsto (\upsilon \actright a)b \quad \text{or }   
a\mapsto (a \actleft \upsilon)b \quad \text{with } b\in A, \upsilon\in \hat{A},
 \end{align*}
respectively, coincide with the subspaces spanned by all maps of the form
\begin{align*}
  a\mapsto b\upsilonb(a), \quad  a\mapsto b\upsilonc(a), \quad a\mapsto \cupsilon(a)b \quad \text{or } a\mapsto \bupsilon(a)b \quad \text{with } b\in A, \upsilon \in \hat{ A},
\end{align*}
respectively.
\end{lemma}
\begin{proof}
We only show that maps of the form $a\mapsto b\upsilonb(a)$ and $a\mapsto b(a \actleft \upsilon)$ span the same subspace of $\End(A)$;
the remaining assertions follow similarly.
Let $b, c \in A$ and write $(1\otimes c)\Delta_{B}(b) =\sum_{i} b_{i} \otimes c_{i}$. Then by right-invariance of $\bpsib$,
\begin{align*}
  c\bpsib(ba)  &= (\bpsib \odot \id)((1 \otimes c)\Delta_{B}(ba)) \\ &= \sum_{i}
(\bpsib \odot \id)((b_{i} \otimes c_{i})\Delta_{B}(a)) = \sum_{i} c_{i}(a \actleft (\psi \cdot b_{i})).
\end{align*}
Since the map $\rT$ is bijective and $\hat A=\psi \cdot A$,  the assertion follows.
\end{proof}
\begin{proposition} \label{proposition:dual-canonical} Let $(\mathcal{A},\mu_{B},\mu_{C},\bpsib,\cphic)$ be a measured regular multiplier Hopf algebroid, where $\mathcal{A}$ is locally projective. Then there exist unique linear bijections $\hat T_{\rho}$, $\hat T_{\lambda}$, $_{\lambda}\hat T$, $_{\rho}\hat T$ that make the following diagrams commute,
   \begin{align*}
     \xymatrix@C=20pt@R=12pt{
      \hAcA \ar[r]^{\hTr} \ar@{^(->}[d] & \hAlA  \ar@{^(->}[d] &
      \hAbA \ar[r]^{\Sigma}\ar@{^(->}[d]  &  \hat A^{C} \otimes {^{C}\hat A} \ar[r]^{\hTl}  & \hAlA
      \ar@{^(->}[d]  
      \\
      \dual{(\ArA)} \ar[r]^{\dual{(\lT)}} & \dual{(\AcA)} &
      \dual{(\AlA)} \ar[r]^{\dual{(\Tl)}} & \dual{(A^{B} \otimes {^{B}A})} \ar[r]^{\dual{\Sigma}} & \dual{(\AcA)} \\
      \hAbA \ar[r]^{\hlT} \ar@{^(->}[d] & \hArA \ar@{^(->}[d] &
      \hAcA \ar[r]^{\Sigma} \ar@{^(->}[d] & \hat A^{B} \otimes {^{B}\hat A}  \ar[r]^{\hrT}  & \hArA \ar@{^(->}[d]
      \\
      \dual{(\AlA)} \ar[r]^{\dual{(\Tr)}} & \dual{(\AbA)} &
      \dual{(\ArA)} \ar[r]^{\dual{(\rT)}}&  \dual{(A^{C} \otimes {^{C}A})} \ar[r]^{\dual{\Sigma}} & \dual{(\AbA)} 
    } 
\end{align*}
that is, such  that for all $a,b\in A$ and $\upsilon,\omega \in \hat{ A}$,
\begin{align*}
  (\hat T_{\lambda}(\upsilon \otimes \omega)\mid a \otimes b) & = 
   (\omega \otimes \upsilon\mid T_{\lambda}(b\otimes a)) = \omega((\upsilon \actright a)b), 
   \\
   (\hat T_{\rho}(\upsilon \otimes \omega)\mid a\otimes b) &= (\upsilon \otimes \omega\mid {_{\lambda}T}(a\otimes b))
 = \upsilon(a(\omega \actright b)),  \\
   ({_{\lambda} \hat T}(\upsilon \otimes \omega)\mid a\otimes b) &= (\upsilon \otimes \omega\mid T_{\rho}(a\otimes b))
= \omega((a \actleft \upsilon)b), \\
   ({_{\rho}\hat T}(\upsilon \otimes \omega)\mid a\otimes b) &= (\omega \otimes \upsilon\mid {_{\rho}T}(b\otimes a))
= \upsilon(b(a  \actleft \omega)).
   \end{align*}
 \end{proposition}
 \begin{proof}
   Uniqueness is  clear. Let us prove existence of $_{\lambda}\hat T$; the other maps can be treated similarly. Choose $\upsilon,\omega \in \hat A$ and $a,b\in A$.   Then by \eqref{eq:pairing-ala},
  \begin{align*}
    (\upsilon \otimes \omega\mid \Tr(a \oo b)) &=         \omega((\upsilon^{C} \odot \iota)(\Delta_{C}(a)(1\otimes b)))= \omega((a \actleft \upsilon)b).
  \end{align*}
We write $\omega$ in the form $\psi \cdot c$, use Lemma  \ref{lemma:compact}, and find $c_{i} \in A$ and $\upsilon_{i} \in \hat A$ such that the expression above becomes equal to
\begin{align*}
\sum_{i}  \psi(c_{i}(\upsilon_{i})_{B}(a)b) = \sum_{i} ((\psi \cdot c_{i}) \otimes  \upsilon_{i}\mid a \otimes b)
\end{align*}
for all $a,b\in A$, where we used \eqref{eq:pairing-aba}. 
 \end{proof}
\begin{remark} \label{remark:slice-htr}
Later, we need the following relation.
For all $\upsilon,\omega \in \hat A$ and $b\in A$,
\begin{align*}
  (\id \odot {_{B}\dual{b}})(\hTr(\upsilon \otimes \omega)))(b) &= \upsilon (- (\omega \actright b)),
\end{align*}
because $(\id \odot {_{B}\dual{b}})(\hTr(\upsilon \otimes \omega))(a) =
    (\dual{a} \underset{\mu_{B}}{\otimes} \dual{b})(\hTr(\upsilon
    \otimes \omega)) = \upsilon(a(\omega \actright b))$  for all $a\in
    A$.
\end{remark}

The last missing ingredient for our first main  theorem is  following result.
\begin{lemma} \label{lemma:comult-aux}
Write ${\hat{T}}_{\rho}(\upsilon\otimes \omega)=\sum_{i} \upsilon_{i}\otimes \omega_{i}$. Then for all $b,c\in A$,
\begin{align*}
  \sum_{i} (\upsilon_{i} \actright b)(\omega_{i} \actright c) &=
\upsilon\actright(b(\omega\actright c)).
\end{align*}
  \end{lemma}
  \begin{proof}
Let $a\in A$ and $u\in \hat{A}$. Then the pentagon relation for $\lT$ \cite[Proposition 2.8]{timmermann:regular} implies
\begin{align*}
u(a(\upsilon\actright(b(\omega\actright c))))
&=(u \mbtimes \upsilon \mbtimes \omega)((\lT)_{12}(\lT)_{23}(a\otimes b\otimes c)) \\
&=(u \mbtimes \upsilon \mbtimes \omega)((\lT)_{23}(\lT)_{13}(\lT)_{12}(a\otimes b\otimes c)) \\
 &=
\sum_{i} (u \mbtimes \upsilon_{i} \mbtimes \omega_{i})((\lT)_{13}(\lT)_{12}(a\otimes b\otimes c))\\
&=    \sum_{i} u(a(\upsilon_{i} \actright b)(\omega_{i} \actright c)). \qedhere
\end{align*}
  \end{proof}

 Define the extended Takeuchi products
\begin{align*}
  \hat A_{C} \overline{\times} {^{C}\hat A} \subseteq \End(\hat A_{C} \otimes {^{C}\hat A})^{\op} \quad\text{and}\quad
  \hat A^{B} \overline{\times} {_{B}\hat A} \subseteq \End(\hat A^{B} \otimes {_{B}\hat A})
\end{align*}
 similarly as in Subsection \ref{section:multiplier-bialgebroids}.
 \begin{theorem} \label{theorem:dual-hopf-algebroid} Let $(\mathcal{A},\mu_{B},\mu_{C},\bpsib,\cphic)$ be a measured regular multiplier Hopf algebroid, where $\mathcal{A}$ is locally projective. Then there exist unique maps
 \begin{align*}
    \hat \Delta_{C} \colon \hat A \to \hat A_{C} \overline{\times} {^{C}\hat A} \quad \text{and} \quad
    \hat \Delta_{B} \colon \hat A \to \hat A^{B} \overline{\times} {_{B}\hat A}
  \end{align*}
 such that  for all $\upsilon,\omega\in \hat A$ and $a,b\in A$
  \begin{align} \label{eq:dual-deltab}
    ( \hat \Delta_{B}(\upsilon)(1 \otimes \omega)\mid a\otimes b) &=
    (\upsilon \otimes \omega\mid (a \otimes 1)\Delta_{B}(b)), \\ \label{eq:dual-deltac}
    ((\upsilon \otimes 1)\hat \Delta_{C}(\omega) \mid  a\otimes b) &= (\upsilon \otimes \omega\mid \Delta_{C}(a)(1 \otimes b)),
  \end{align}
  and the tuple $\hat{\mathcal{A}}:= (\hat A,C,B,t_{B}^{-1},t_{C}^{-1},\hat \Delta_{C},\hat \Delta_{B})$ is a regular multiplier Hopf algebroid which is locally projective.
\end{theorem}
We call $\hat{\mathcal{A}}$ the \emph{dual regular
multiplier Hopf algebroid}  of  $(\mathcal{A},\mu_{B},\mu_{C},\bpsib,\cphic)$.
\begin{proof}
By Lemma \ref{lemma:dual-bimodule}, $\hat{A}$ is locally projective.

In the following, we use the formulas given in  Proposition
\ref{proposition:dual-canonical} without further notice.

To prove existence of a linear map $\hat \Delta_{B}$ satisfying \eqref{eq:dual-deltab}, it suffices to show that
\begin{align*}
  \hTr(\upsilon \otimes \omega)(u \otimes 1) = \hTl(u \otimes \upsilon)(1 \otimes \omega)
\end{align*}
for all $u,\upsilon,\omega \in \hat A$. To prove this relation, note
 that for all $u',\upsilon',\omega' \in \hat{ A}$ and elementary
 tensors $a\otimes b\in \AcA$,
\begin{align*}
(\upsilon'u  \underset{\mu_{C}}{\otimes} \omega' \mid a\otimes b) &=
(\upsilon' u)(a\comega'(b)) = \upsilon'((u \actright a)\comega'(b)) =
 (\upsilon' \underset{\mu_{C}}{\otimes} \omega'\mid (u \actright a) \otimes
 b)  
\end{align*}
 and similarly for all  $a\otimes b\in \AbA$,
\begin{align*}
  (u'   \underset{\mu_{C}}{\otimes} \upsilon'\omega \mid a\otimes b) &= (u'  \underset{\mu_{C}}{\otimes} \upsilon' \mid a\otimes (\omega \triangleright b)),
\end{align*}
 whence for all $a,b\in A$,
\begin{align} \label{eq:dual-deltab-1}
  \begin{aligned}
    (\hTr(\upsilon \otimes \omega)(u \otimes 1)\mid a\otimes b)& =
    (\hTr(\upsilon \otimes \omega)\mid  (u\actright a) \otimes b)   \\
    &= \upsilon((u \actright a)(\omega \actright b))  \\
    &= (\hTl(u \otimes \upsilon)\mid a \otimes (\omega \actright b)) =
    (\hTl(u \otimes \upsilon)(1 \otimes \omega)\mid a\otimes b).
  \end{aligned}
\end{align}
  To see that the map $\hat \Delta_{B}$ is a homomorphism, write
\begin{align*}
  \hat\Delta_{B}(\upsilon)(1\otimes \omega)&=\sum_{i} \upsilon_{i}\otimes \omega_{i}.
\end{align*}
Then  \eqref{eq:dual-deltab-1} and  Lemma \ref{lemma:comult-aux} imply
\begin{align*}
( \hat\Delta_{B}(u)\hat\Delta_{B}(\upsilon)(1\otimes \omega)\mid b\otimes c) &=
\sum_{i}u((\upsilon_{i}\actright b)(\omega_{i}\actright c))\\
&=u(\upsilon\actright(b(\omega\actright c))) \\
&=(u\upsilon)(b(\omega \actright c)) =(\hat\Delta_{B}(u\upsilon)(1\otimes \omega)  \mid b\otimes c).
\end{align*}

Similar arguments prove existence of a  homomorphism $\hat \Delta_{C}$  satisfying \eqref{eq:dual-deltac}.

We claim that
$\hat{\Delta}_{B}(y\upsilon x) = (x\otimes 1)\hat{\Delta}_{B}(\upsilon)(1\otimes
y)$ for all $x\in B$, $y\in C$ and $\upsilon \in \hat{A}$.
 Fix $\omega\in \hat{A}$, $a,b\in
A$ and write $\hat{\Delta}_{B}(\upsilon)(1\otimes \omega) = \sum_{i}
\upsilon_{i} \otimes \omega_{i}$. Then
\begin{align*}
  (\hat{\Delta}_{B}(y\upsilon x)(1 \otimes \omega)\mid a\otimes b) &= \upsilon(xa (\omega\triangleright b)y) \\ &= \upsilon(xa (\omega \triangleright by)) =
(\hat{\Delta}_{B}(\upsilon)(1\otimes \omega)\mid xa\otimes by)
\end{align*}
  by \eqref{eq:act-bimodule} and \eqref{eq:dual-module},
and
\begin{align*}
  \sum_{i}(\upsilon_{i}x \otimes y\omega_{i}\mid a\otimes b) &= \sum_{i} \upsilon_{i}(xa {_{C}\omega_{i}}(by)) = \sum_{i}(\upsilon_{i} \otimes \omega_{i}\mid xa \otimes by)
\end{align*}
 by  \eqref{eq:pairing-aca}. The claim follows.
 
A similar argument shows that  $\hat{\Delta}_{C}(y\upsilon x) = (x\otimes 1)\hat{\Delta}_{C}(\upsilon)(1\otimes
y)$ for all $x\in B$, $y\in C$ and $\upsilon\in \hat{A}$.

Next, we prove  coassociativity  of $\hat\Delta_{B}$, which means that for all $u,\upsilon,\omega \in \hat A$, 
\begin{align*}
  (\hat \Delta_{B} \otimes\id)(\hat\Delta_{B}(\upsilon)(1\otimes \omega)) (u \otimes 1\otimes 1) =
  (\id \otimes \hat \Delta_{B})(\hat \Delta_{B}(\upsilon)(u \otimes 1))(1 \otimes 1 \otimes w).
\end{align*}
Since $\hat A$ is locally projective as a module over $B$ and $C$, it
suffices to show that this relation holds application of a slice map
of the form $(\dual{a})^{B} \odot \id \odot {_{B}\dual{c}}$, or, equivalently, that the maps
\begin{align*}
  \upsilon \mapsto ((\dual{a})^{B} \odot \id)(\hat \Delta_{B}(\upsilon)(u \otimes 1)) \quad \text{and} \quad
  \upsilon \mapsto (\id \odot {_{B}\dual{c}})(\hat \Delta_{B}(\upsilon)(1 \otimes \omega))
\end{align*}
commute. But \eqref{eq:dual-deltab-1} shows that these maps are given by
$\upsilon \mapsto \upsilon((u \actright a)-)$ and $\upsilon \mapsto \upsilon(-(\omega \actright c))$, respectively. Hence, $\Delta_{B}$ is coassociative.

Similar arguments prove coassociativity of $\hat \Delta_{C}$ and mixed coassociativity.

Finally, consider the subspaces
\begin{align*}
  {_{C}\hat I} &= \langle {_{C}f} (\omega) :  f \in \dual{({_{C}\hat A})}, \omega \in \hat A\rangle, &
  {\hat I^{C}} &= \langle f^{C}(\omega) : f \in \dual{({\hat A^{C}})}, \omega \in  \hat A\rangle, \\
  \hat I_{B} &= \langle f_{B}(\omega)  : f \in \dual{(\hat A_{B})}, \omega \in \hat A\rangle, &
  {^{B}\hat I} &= \langle {^{B}f}(\omega)  : f\in \dual{(^{B}\hat A)}, \omega \in \hat A\rangle
\end{align*}
of $C$ and $B$, respectively.  Using Lemma \ref{lemma:eval} and surjectivity of the maps $\cphic$ and $\bpsib$, we see that $  {_{C}\hat I}=C=  {\hat I^{C}}$ and  $\hat I_{B} =   {^{B}\hat I}$.

Since the canonical maps $\hTl,\hTr,\hlT,\hrT$ are bijective, we can conclude that $\hat{\mathcal{A}}$  is a regular multiplier Hopf algebroid as claimed.
\end{proof}
Let us next consider the involutive case.
\begin{theorem}
  If   $(\mathcal{A},\mu_{B},\mu_{C},\bpsib,\cphic)$  is a measured multiplier Hopf
  $*$-algebroid, then the dual $\hat{\mathcal{A}}=(\hat A,C,B,t_{B}^{-1},t_{C}^{-1},\hat\Delta_{C},\hat\Delta_{B})$ is a multiplier Hopf
  $*$-algebroid with involution given by $\omega^{*}(a):=\overline{\omega(S(a)^{*})}$ for all  $\omega \in\hat A$ and $a\in A$.  
\end{theorem}
\begin{proof}
By Proposition \ref{proposition:dual-involution}, the formula above defines an involution on $\hat A$, and it suffices to show that $(\ast \oo
  \ast) \circ \hTr \circ (\ast \oo \ast) = \hrT$.
Let $\upsilon \otimes \omega \in \hAcA$ and  $a\otimes b \in \ArA$. Then by
  definition of the involution and of the maps $\hTr$ and $\hrT$, 
  \begin{align*}
    (((\ast \oo \ast)\circ \hTr \circ (\ast \oo \ast)) \Sigma (\upsilon \oo
    \omega)\mid a\otimes b)
  \end{align*}
is  equal to 
  \begin{align*}
    (\upsilon \otimes \omega\mid ((* \oo *) \circ (S \oo
    S) \circ \Sigma \circ \lT 
    \circ (S^{-1} \oo S^{-1}) \circ (* \oo *))(a\otimes b)).
\end{align*}
But $\circ (S \oo
    S) \circ \Sigma \circ \lT 
    \circ (S^{-1} \oo S^{-1})  = \Tr \circ \Sigma$ by  \eqref{dg:galois-antipode}, and 
$(*\oo *)\circ  \Tr \circ \Sigma   \circ (* \oo *)= \rT \circ \Sigma $ by Definition \ref{definition:involution} (3). Therefore, the expression above is equal to
\begin{align*} 
(\upsilon \otimes \omega\mid  (\rT \Sigma)(a \otimes b))      &=  ((\hrT \Sigma)(\upsilon \otimes \omega)\mid a\otimes b). \qedhere
  \end{align*}
\end{proof}
Next, we determine the counit and antipode of the dual $\hat{\mathcal{A}}$. 

\begin{theorem} \label{theorem:dual-counits-antipode} Let
  $(\mathcal{A},\mu_{B},\mu_{C},\bpsib,\cphic)$ be a measured regular multiplier Hopf
  algebroid, where $\mathcal{A}$ is locally projective. Then the   antipode
  $\hat S$ and left and right counits $\hbeps$ and $\hepsc$  of the dual regular   multiplier Hopf algebroid
 $(\hat A,C,B,t_{B}^{-1},t_{C}^{-1},\hat \Delta_{C},\hat \Delta_{B})$
are given by
  \begin{align*} 
    \hat S(\omega) &= \omega \circ S, & \hbeps(a \cdot \phi) &= t_{C}(\cphic(a)),
    & \hepsc(\psi \cdot a) &= t_{B}(\bpsib(a)) 
  \end{align*}
  for all $\omega \in
    \hat A$ and $ a\in A$. In particular, the functionals $\mu_{B} \circ \hbeps$ and $\mu_{C} \circ \hepsc$ are equal and coincide with evaluation at $1\in M(A)$.
\end{theorem}
\begin{proof}
By Lemma \ref{lemma:eval}, evaluation at $1\in M(A)$ defines a functional $\heps \in (\hat A)^{\sqcup}$ such that
 $\hbeps(a\cdot \phi)=t_{C}(\cphic(a))$, $\hepsc(\psi \cdot a) = t_{B}(\bpsib(a)) $ and
$t_{C} (\hceps(a\cdot \phi)) = t_{C} \circ \sigma_{C}(\cphic(a)) =
\hbeps(a\cdot \phi) $ for all $a \in A$.   The map $\hbeps$ is the
left counit for $\hat A$ because by Remark \ref{remark:slice-htr}, 
\begin{align*}
  ((t_{C}\circ \hceps \odot \id)(\hTr(\upsilon \otimes \omega)))(b) &=
  (\heps \underset{\mu_{B}}{\otimes} \dual{b}) (\hTr(\upsilon \otimes \omega)) \\ &=
 \heps(\upsilon(-(\omega \actright b))) = \upsilon(\omega \actright b) = (\upsilon\omega)(b)
\end{align*}
for all $\upsilon,\omega \in \hat A$ and $b\in A$ 
and hence $(t_{C}\circ \hceps \odot \id)(\hTr(\upsilon \otimes \omega)) = \upsilon\omega$.  
A
similar  argument shows that $\hepsb$ is the right counit of $\hat{\mathcal{A}}$.

To prove the formula for  the antipode, we use \cite[Proposition 5.8]{timmermann:regular} to $\mathcal{A}$ and $\hat{\mathcal{A}}$, and find that the
 following diagrams commute:
  \begin{align*}
    \xymatrix@R=18pt{
      \AbA \ar[r]^{\rT\Sigma} \ar[d]_{\Tr} & \ArA  & 
&
      \hAbA \ar[r]^{\hrT\Sigma} \ar[d]_{\hTr} & \hArA\\
      \AlA  \ar[r]^{S\otimes \id}& \AcA \ar[u]_{\lT}& &
      \hAlA \ar[r]^{\hat S\otimes \id}& \hAcA \ar[u]_{\hlT} }
  \end{align*}
But by construction of the maps $\hTr,\hlT$ and $\hrT$, commutativity of the first diagram also implies commutativity of the second diagram with $\hat S$ replaced by the transpose of $S$, i.e.,
\begin{align*}
  (  \hat S(\upsilon) \otimes \omega\mid a \otimes b) &= ((\hlT)^{-1}  (\hrT \Sigma) (\hTr)^{-1}(\upsilon \otimes \omega)\mid a\otimes b) \\
  &= (\upsilon \otimes \omega\mid  (\lT)^{-1}  (\rT \Sigma)  (\Tr)^{-1}(a\otimes b)) = (\upsilon \otimes \omega\mid S(a) \otimes b)
\end{align*}
for all $a,b\in A$ and $\upsilon,\omega\in \hat A$, whence $(\hat S(\upsilon))(a)=\upsilon(S(a))$ for all $\upsilon \in \hat A$ and $a\in A$.
\end{proof}

\subsection{The full duality and biduality}
\label{subsection:dual-integrals}

To obtain a full duality and biduality of measured regular multiplier Hopf algebroids, it remains to construct total integrals on the dual. This can again be done in a similar way as in the case of multiplier Hopf algebras, see \cite{daele}.

\begin{theorem} \label{theorem:dual-measured} Let
  $(\mathcal{A},\mu_{B},\mu_{C},\bpsib,\cphic)$ be a measured regular multiplier Hopf
  algebroid, where $\mathcal{A}$ is locally projective, with dual regular multiplier Hopf algebroid
  $\hat{\mathcal{A}}=(\hat A,C,B,t_{B}^{-1},t_{C}^{-1},\hat \Delta_{C},\hat \Delta_{B})$ and define
$\hcpsic \colon \hat A \to C$ and  $\hbphib \colon \hat A \to B$ by
\begin{align*}
  \hcpsic( a\cdot \phi):= \ceps(a) \quad \text{and}\quad
  \hbphib(\psi \cdot a):= \epsb(a) \quad \text{for all } a\in A.
\end{align*}
 Then $(\hat{\mathcal{A}},\mu_{C},\mu_{B},\hcpsic,\hbphib)$ is a
 measured regular multiplier Hopf algebroid, where $\hat{\mathcal{A}}$
 is locally projective.
Its left total integral  $\hat\phi=\mu_{B} \circ \hbphib$ and right total integral and $\hat\psi=\mu_{C} \circ \hcpsic$ satisfy
  \begin{align} \label{eq:dual-left-integral} \hat \phi((\psi \cdot
    S(a))\omega) &= \omega(a), & \sigma^{\hat\phi}(\psi \cdot a) &=
    (\psi \circ S) \cdot S^{-2}(a), & \sigma^{\hat\phi}(y) &=
    (\sigma^{\psi})^{-1}(y), \\ \label{eq:dual-right-integral}
    \hat\psi (\omega (S(b) \cdot \phi)) &=
    \omega(b),   &
    (\sigma^{\hat\psi})^{-1}(a\cdot \phi)&= S^{-2}(a) \cdot (\phi \circ S), &
    \sigma^{\hat\psi}(x) &= (\sigma^{\phi})^{-1}(x)
  \end{align}
  for all $a,b\in A$, $x\in B$, $y\in C$ and $\omega \in \hat
  A$. 
If   $(\mathcal{A},\mu_{B},\mu_{C},\bpsib,\cphic)$ is a measured multiplier Hopf
  $*$-algebroid, then so is  $(\hat{\mathcal{A}},\mu_{C},\mu_{B},\hcpsic,\hbphib)$, and then
  \begin{align*}
    \hat \psi((a\cdot \phi)^{*}(a \cdot \phi)) &= \phi(a^{*}a)
    &&\text{and} &
    \hat\phi((\psi \cdot a)(\psi \cdot a)^{*}) &= \psi(aa^{*}) &&
    \text{for all } a\in A.
  \end{align*} 
\end{theorem}
In the situation above, we call
  $(\hat{\mathcal{A}},\mu_{C},\mu_{B},\hcpsic,\hbphib)$ the \emph{dual} of
  $(\mathcal{A},\mu_{B},\mu_{C},\bpsib,\cphic)$.
\begin{proof} 
By Theorem \ref{theorem:dual-counits-antipode},  $\mu_{C}\circ \hceps $ and $\mu_{B} \circ \hepsc$ coincide. 

The map $\hcpsic$ is $C$-bilinear because by \eqref{eq:dual-module},
\begin{align*}
  \hcpsic(y (a\cdot \phi) y') = \hcpsic(yt_{C}(y')a \cdot \phi) = \ceps(yt_{C}(y')a) = y\ceps(a)y'=y\hcpsic(a\cdot \phi)y'
\end{align*}
for all $y,y' \in C$ and $a\in A$.
 The functional $\hat\psi=\mu_{C} \circ \hcpsic$ on $\hat A$ is factorizable and
 \begin{align*}
   \hbpsi(a\cdot \phi)  &= t_{C}(\epsc(a)), &
   \hpsib(a\cdot \phi) &= (\sigma^{\phi})^{-1}(\epsb(a)),
 \end{align*}
because by \eqref{eq:dual-module} and Theorem \ref{theorem:modular-automorphism},
 \begin{align*}
   \hat\psi(x(a\cdot \phi)) &= \hat\psi(a\cdot \phi \cdot t_{B}(x)) = \hat\psi(at_{C}^{-1}(x) \cdot \phi) = 
   \eps(at_{C}^{-1}(x)) = \mu_{B}(xt_{C}(\epsc(a))), 
 \end{align*}
and
 \begin{align*}
      \hat\psi((a\cdot \phi)x') = \hat\psi(a\cdot \phi \cdot x') &= \hat\psi(a \sigma^{\phi}(x') \cdot \phi) \\ &= \varepsilon(a\sigma^{\phi}(x'))  =\mu_{B}(\epsb(a)\sigma^{\phi}(x'))  = \mu_{B}((\sigma^{\phi})^{-1}(\epsb(a))x').
\end{align*}

Taking $x'=(\sigma^{\phi})^{-1}(x)$ in the equations above, we find
\begin{align*}
  \hat \psi(x (a\cdot \phi) ) = \varepsilon(at_{C}^{-1}(x)) = \varepsilon(ax) =\varepsilon(a\sigma^{\phi}(x')) =\hat\psi((a \cdot \phi)x')
\end{align*}
which will give  the last equation in \eqref{eq:dual-right-integral} once existence of $\sigma^{\hat\psi}$ has been shown.

To see that $\hcpsic$ is a partial right integral, we use Remark \ref{remark:slice-htr} with  $\omega \in \hat A$, $\upsilon \in \hat A$ of the form $\upsilon= a\cdot \phi$ for some $a\in A$, and $b\in A$, and find
\begin{align*}
  ((t_{C} \circ \hcpsic  \odot \id)(\hat{\Delta}_{B}(\upsilon)(1 \otimes \omega)))(b) &=
(\hat{\psi} \underset{\mu_{B}}{\otimes} \dual{b})(\hat\Delta_{B}(\upsilon)(1 \otimes \omega))  \\ &=
  \hat\psi((\omega \actright b) \cdot \upsilon) = \hat\psi((\omega \actright b)a\cdot \phi) = \eps((\omega \actright b)a).
\end{align*}
We use  \eqref{eq:counits-multiplicative}, \eqref{eq:act-bimodule}  the counit property and \eqref{eq:dual-module}, and find
\begin{align*}
\varepsilon((\omega \actright b)a) = \varepsilon((\omega \actright b)\ceps(a)) &=\varepsilon(\omega \actright b\ceps(a))
\\ &=  \omega(b\ceps(a)) = \omega(b \hcpsic(a\cdot \phi)) = (\hcpsic(a\cdot \phi)\omega)(b).
\end{align*}
Since $b\in A$ was arbitrary, we can conclude that
\begin{align*}
  (t_{C} \circ \hcpsic  \odot \id)(\hat\Delta_{B}(\upsilon)(1 \otimes \omega))= \hcpsic(a\cdot \phi)\omega,
\end{align*}
showing that $\hcpsic$ is a partial right integral. It is full, that
is, $\hbpsi(A)=B=\hpsib(A)$, because $\ceps$ and $\epsc$ are
surjective by Remark \ref{remark:counits-surjective}.

Similar arguments show  that the functional $\hat\phi=\mu_{B} \circ \hbphib$ is factorizable and  that $\hbphib$ is a partial left integral.

To  prove the first equation in  \eqref{eq:dual-right-integral}, let $\upsilon \in \hat A$, $b\in A$ and write $\omega = b\cdot \phi$. Then
\begin{align*}
  \hat \psi(\upsilon (b\cdot \phi))  = \hat \psi((b \actleft \upsilon S^{-1}) \cdot \phi ) = \varepsilon(b \actleft \upsilon S^{-1}) ) = \upsilon(S^{-1}(b))
\end{align*}
 by \eqref{eq:product-ahat}.
Since $\phi$ is faithful, we can conclude that $\hat\psi$ is faithful
as well.  Moreover, taking $\upsilon=a\cdot \phi$ with $a\in A$ and  $\omega' = S^{-2}(b) \cdot (\phi \circ
    S)$,  we get
  \begin{align*}
    \hat \psi((a\cdot \phi)(b\cdot \phi)) &= \phi(S^{-1}(b)a) = (\phi
    \cdot S^{-1}(b)) (a) = \omega'(S^{-1}(a)) = \hat\psi(\omega'(a \cdot
    \phi)),
  \end{align*}
 whence the second  equation in \eqref{eq:dual-right-integral} follows. The third one was proven above already.

 If $(\mathcal{A},\mu_{B},\mu_{C},\bpsib,\cphic)$ is a measured
  multiplier Hopf $*$-algebroid,  then
  \begin{align*}
    \hat\psi((a\cdot \phi)^{*}(a\cdot \phi)) &= (a \cdot
    \phi)^{*}(S^{-1}(a)) =  \phi(S(S^{-1}(a))^{*}a)^{*} = \phi(a^{*}a)
  \end{align*}
  for all $a\in A$ by \eqref{eq:dual-right-integral}, and a similar
  calculation shows that $ \hat\phi((\psi \cdot a)(\psi \cdot a)^{*})
  = \psi(aa^{*})$ whence $\hat\psi$ and $\hat \phi$ are
  positive.
\end{proof}
The following biduality holds.
\begin{theorem} \label{theorem:biduality} Let
  $(\mathcal{A},\mu_{B},\mu_{C},\bpsib,\cphic)$ be a measured regular multiplier Hopf
  algebroid or a measured multiplier Hopf $*$-algebroid, where
  $\mathcal{A}$ is locally projective, and let
  $(\hat{\mathcal{A}},\mu_{C},\mu_{B},\hcpsic,\hbphib)$ be its
  dual. Then the map $A \to \dual{(\hat A)}$ given by $a\mapsto
  \dual{a}$ is an isomorphism from $(\mathcal{A},\mu_{B},\mu_{C},\bpsib,\cphic)$ to
  the bi-dual   $(\hhat{\mathcal{A}},\mu_{B},\mu_{C},{_{B}\hhat{\psi}_{B}},{_{C}\hhat{\phi}_{C}})$.
\end{theorem}
\begin{proof}
  The bidual regular multiplier Hopf algebra $\hhat{\mathcal{A}}$ has the form
  \begin{align*}
    \hhat{\mathcal{A}}=(\hhat{A},B,C,t_{B},t_{C},\hhat{\Delta}_{B},\hhat{\Delta}_{C}).
  \end{align*}
The map $j\colon A \to \dual{(\hat A)}$ given by $a\mapsto
  \dual{a}$ is a linear surjection from $A$ to $\hhat{A}$ because 
 $\hhat{A} = \hat{A} \cdot \hat\psi = (A \cdot \phi) \cdot \hat \psi$ and
   \begin{align*}
    ((S(b) \cdot \phi) \cdot \hat \psi)(\upsilon) =
    \hat\psi(\upsilon(S(b) \cdot \phi)) = \upsilon(b) = \dual{b}(\upsilon)
  \end{align*}
  for all $b\in A$ and $\upsilon \in \hat A$ by \eqref{eq:dual-right-integral}. The map $j$ is injective because elements of $\hat A$ separate the points of $A$.  To check that it is $B$-bilinear, it helps to write $\hat C=B$ and $\hat t_{\hat C}=t_{C}^{-1}$. We insert \eqref{eq:dual-module} twice and find that
\begin{align*}
  x\dual{b}x'(\upsilon) = \dual{b}(\upsilon x \hat{t}_{\hat C}(x')) =  \upsilon(x b t_{C}(\hat t_{\hat C}(x'))) = \upsilon(xbx')
\end{align*}
   for all $x,x'\in \hat C=B$, $b\in A$ and $\upsilon\in \hat A$.
 A similar calculation shows that $j$ is $C$-bilinear.

  Let us next consider the canonical maps of $\hhat{\mathcal{A}}$.  By
  construction and Lemma \ref{lemma:pairing}, the map
  $\hhat{T}_{\rho}$ satisfies
  \begin{align*}
    (\hhat{T}_{\rho}(\dual{a} \oo \dual{b})\mid \upsilon \oo \omega) =
    (\dual{a} \oo \dual{b}\mid \hlT(\upsilon \oo \omega)) =
    (\upsilon \oo \omega\mid \Tr(a \oo b))
  \end{align*}
  for all $a,b\in A$ and $\upsilon,\omega \in \hat{A}$. Therefore, the map $j$ identifies $\hhat{\Tr}$ with $\Tr$,
  and similar arguments show that it identifies the maps $\hhat{T}_{\lambda},
  {_{\lambda}\hhat{T}}$ and ${_{\rho}\hhat{T}}$ with $\Tl,\lT$ and
  $\rT$, respectively.  

In particular, by Proposition \ref{proposition:dual-canonical},
\begin{align*}
  \upsilon(a(\omega \triangleright b)) = (\upsilon \otimes \omega\mid \lT(a\otimes b)) &= ({_{\lambda}}\hat{\hat{T}}(\dual{a} \otimes \dual{b})\mid \upsilon \otimes \omega) \\&= \dual{b}((\upsilon \triangleleft \dual{a})\omega) = ((\upsilon \triangleleft \dual{a})\omega)(b) = (\upsilon \triangleleft  \dual{a})(\omega \triangleright b),
\end{align*}
showing that
\begin{align} \label{eq:dual-bimodule-explicit-1}
  \upsilon \triangleleft \dual{a} = \upsilon \cdot a.
\end{align}
In particular, $(\dual{a}\dual{b})(\upsilon) =\dual{b}(\upsilon
\triangleleft \dual{a}) = \upsilon(ab) = \dual{(ab)}(\upsilon)$ so
that $j$ is a homomorphism.

Therefore, $j$ is an isomorphism from $\mathcal{A}$ to $\hhat{\mathcal{A}}$.

 The total left integral $\hhat{\phi}$ of the bidual is given by
 \begin{align*}
   \hhat{\phi}(\dual{a}) = \hhat{\phi}((S(a) \cdot \phi)\cdot \hat
    \psi) &= {\hat\eps}((\sigma^{\hat\psi})^{-1}(S(a) \cdot \phi)) \\ &=
    \hat \eps(S^{-1}(a) \cdot (\phi \circ S)) =
    (\phi \circ S)(S^{-1}(a)) =\phi(a)
  \end{align*}
  for all $a\in A$, where we used the formula for $\sigma^{\hat\psi}$
obtained in the proof of Theorem \ref{theorem:dual-measured}. A similar argument shows
  that the total right integral $\hhat{\psi}$ satisfies $\hhat{\psi}(\dual{a})=\psi(a)$ for all $a\in A$.

  Finally, assume that $(\mathcal{A},\mu_{B},\mu_{C},\bpsib,\cphic)$ is a measured
  multiplier Hopf $*$-algebroid. Then $j$ is a $*$-isomorphism because for all $b\in A$ and $\upsilon \in \hat
  A$, 
  \eqref{eq:counit-antipode-involution} implies $(\dual{b})^{*}(\upsilon) = \dual{b}(\hat
  S(\upsilon)^{*})^{*}$ and $\hat S(\upsilon)^{*} = \ast \circ
  (\upsilon \circ S) \circ \ast \circ S = \ast \circ \upsilon \circ
  \ast$,  whence $(\dual{b})^{*}(\upsilon) =
  \upsilon(b^{*})$.
\end{proof}

\begin{remark} \label{remark:dual-bimodule-explicit}
  The $M(\hat A)$-bimodule structure on
$A$, obtained via the identification $A \to \hat{\hat A}$,
$a\mapsto \dual{a}$
 takes the
following form: for all $a \in A$ and $\upsilon \in A^{\sqcup}_{0} \cong M(\hat A)$,
\begin{align}
  \label{eq:dual-bimodule-explicit}
  \upsilon \cdot \dual{a} = \dual{(\upsilon \actright a)} \quad \text{and}\quad
  \dual{a} \cdot \upsilon = \dual{(a\actleft \upsilon)}.
\end{align}
 Indeed,  for all $\omega \in \hat A$, we have
$(\upsilon \cdot \dual{a})(\omega) = \dual{a}(\omega\upsilon) = (\omega\upsilon)(a) = \omega(\upsilon \actright a) = \dual{(\upsilon \actright a)}(\omega)$, and a similar calculation proves the second equation above. 
\end{remark}

\subsection{When is the dual unital?} 
In \cite{boehm:bijective}, Böhm and Szlach\'anyi develop a duality for (unital) Hopf algebroids with non-degenerate cointegrals. Existence of such cointegrals implies  existence of suitably non-degenerate partial left and right integrals,   and holds if and only if the inclusions of the base algebras $B$ and $C$ in the total algebra $A$ are Frobenius extensions \cite{boehm:integrals}.  To clarify the relation between the duality developed here with the theory in  \cite{boehm:bijective}, one should therefore study the interplay of cointegrals with base weights.   For the moment, we shall only show that the dual  of a measured multiplier Hopf algebroid is unital if and only if the latter has a two-sided cointegral. 
\begin{definition}
  Let $\mathcal{A}$ be a regular multiplier  Hopf algebroid with left counit $\ceps$ and right counit $\epsb$. A \emph{left cointegral} in $\mathcal{A}$ is an element $l \in A$ satisfying $al=\ceps(a)l$ for all $a\in A$. A \emph{right cointegral} is an element $r\in A$ satisfying $r a = r\epsb(a)$ for all $a\in A$. A \emph{cointegral} is an element of $A$ that is both a left and a right cointegral.
\end{definition} 
Let us call an element $\lambda\in \hat{A}$ a \emph{left integral} if $_{C}\lambda=\lambda_{C}$  is a partial left integral, and an element $\rho\in \hat{A}$ a \emph{right integral} if $_{B}\rho=\rho_{B}$ is a partial right integral.
\begin{lemma} \label{lemma:cointegral-integral}
  Let  $(\mathcal{A},\mu_{B},\mu_{C},\bpsib,\cphic)$ be a measured multiplier Hopf algebroid, where $\mathcal{A}$ is locally projective.
  \begin{enumerate} 
  \item An element $\lambda \in \hat{A}$ is a left cointegral if and only if it is  a left integral on $A$. 
\item An element $\rho\in \hat{A}$ is a right cointegral if and only  if it is a  right integral on $A$.
  \end{enumerate}
\end {lemma}
\begin{proof}
We only prove (1).
  Let $\lambda \in \hat{A}$,  $a,b\in A$ and $\omega=a\cdot \phi$. Then $(\omega\lambda)(b) =\omega(\lambda \triangleright b)$ and
  \begin{align*}
    (\hbeps(\omega)\lambda)(b) = (t_{C}(\cphic(a))\lambda)(b) &= \lambda(t_{B}t_{C}(\cphic(a))b)  \\
&= \mu_{C}({_{C}\lambda}(b)\cphic(a)) = \phi({_{C}\lambda}(b)a) = \omega({_{C}\lambda}(b)),
  \end{align*}
where we used the relation $t_{B}t_{C}=\sigma_{C}^{-1}$ \eqref{eq:modular-automorphism}. The equivalence in (1) follows.
\end{proof}
In particular, we see that if $A$ is unital, then $\hat{A}$ has a left and a right cointegral.
More precisely, modulo duality,  the  following equivalence holds: 
\begin{proposition}\label{proposition:cointegral-unital}
  Let $(\mathcal{A},\mu_{B},\mu_{C},\bpsib,\cphic)$ be a measured multiplier Hopf algebroid, where $\mathcal{A}$ is locally projective.
Then the following conditions are equivalent:
\begin{enumerate}
\item the dual $\widehat{\mathcal{A}}$ is unital;
\item there exists a left cointegral $l\in A$ such that $\cphic(l)=1_{C}$;
\item there exists a right cointegral $r\in A$ such that $\bpsib(r)=1_{B}$.
\end{enumerate}
In that case, $l$ and $r$ can be chosen such that $\hat{A} \triangleright l =A=r \triangleleft \hat{A}$.
\end{proposition}
\begin{proof}
We only prove the equivalence of (1) and (2); equivalence with (3) follows similarly. Instead of using Lemma \ref{lemma:cointegral-integral} and biduality, we proceed directly.

  Suppose (1) holds. Then $\varepsilon\in \widehat{\mathcal{A}}=A\cdot \phi$. Choose $l \in A$ such that $l\cdot \phi=\varepsilon$. Then  $\ceps = l \cdot \cphic$  and hence $\cphic(l)=\ceps(1_{A})=1_{C}$.  We show that   $l$ is a left cointegral.  Let $a,b\in A$. Then
  $\phi(ab\lambda) = \varepsilon(ab) = \varepsilon(a\ceps(b)) = \phi(a\ceps(b)\lambda)$
by \eqref{eq:counits-multiplicative} which implies  $b\lambda = \ceps(b)\lambda$ because $\phi$.

Conversely, assume $l\in A$ is as in (2). Then $(l\cdot \phi)(a) = \phi(al)=\phi(\ceps(a)l) = \mu_{C}(\ceps(a)\cphic(l)) \varepsilon(a)$ for all $a\in A$ and hence $\varepsilon =l\cdot \phi \in \hat{ A}$. Finally,  strong invariance \eqref{eq:strong-invariance} implies
  $a = (l \cdot \phi) \triangleright a = S((\phi \cdot a) \triangleright l)$ for all $a\in A$ and hence $\hat{A} \triangleright l = S^{-1}(A)=A$.
\end{proof}

\section{Morphisms}
\label{section:morphisms}

We now introduce morphisms of regular multiplier Hopf algebroids and show that they automatically preserve the antipode and are, in some sense, compatible with the counits. We pay special attention to morphisms into duals of measured regular multiplier Hopf algebroids because those   will allow us  to succinctly describe the  duals  of some examples considered in  Section \ref{section:examples}.

\subsection{Definition and examples of morphisms}
Let 
\begin{align} \label{eq:morphisms-a-d}
  \mathcal{A}=(A,B,C,t_{B},t_{C},\Delta_{B},\Delta_{C}) \quad \text{and} \quad \mathcal{D}=(D,E,F,t_{E},t_{F},\Delta_{E},\Delta_{F})
\end{align}
be two regular multiplier Hopf algebroids, let $\pi \colon D\to M(A)$ be a non-degenerate homomorphism and  suppose that, after extension of $\pi$ to multipliers,
\begin{align} \label{eq:mor-base-b}
    \pi(E)B&=B=B\pi(E), & \pi \circ t_{E} &= t_{B} \circ \pi,
    \\\pi(F)C&=C=C\pi(F), & \pi \circ t_{F} &= t_{C} \circ \pi. \label{eq:mor-base-c}
\end{align}

We shall make precise what it means  for $\pi$ to be compatible with the comultiplications of $\mathcal{D}$ and $\mathcal{A}$. Intuitively, this means that 
\begin{align} \label{eq:mor-delta-b}
(a'\pi(d') \otimes a''\pi(d'')) \Delta_{B}(\pi(d)) &= (a' \otimes a'')  (\pi \otimes \pi)((d' \otimes d'')\Delta_{E}(d)), \\ 
\label{eq:mor-delta-c}
\Delta_{C}  (\pi(d))(\pi(d')a' \otimes \pi(d'')a'') &= (\pi \otimes \pi)(\Delta_{F}(d)(d' \otimes d''))(a' \otimes a'')
\end{align}
for all $a',a''\in A$ and $d,d',d''\in D$, or, equivalently and more tersely, such that
\begin{align}
  \label{eq:mor-delta}
\Delta_{B}(\pi(d)) =(\pi {_{E}\overline{\times}^{E}}\pi)(\Delta_{E}(d)) \quad \text{and} \quad \Delta_{C}(\pi(d)) = (\pi {^{F}\overline{\times}_{F}} \pi)(\Delta_{F}(d))  
\end{align}
for all $d\in D$. However,  several expressions in these equations  have not yet been defined and require a careful interpretation. 

To make sense of the left hand sides in \eqref{eq:mor-delta-b} and \eqref{eq:mor-delta-c},  we  extend the comultiplications $\Delta_{B}$ and $\Delta_{C}$ to multipliers. 
We write $(\AlA)_{(A\oo A)}$ and $_{(A \oo A)}(\ArA)$ when we regard $\AlA$ and $\ArA$ as a right or as a left module over $A \otimes A$, respectively.  
\begin{lemma} \label{lemma:comult-extension}
  The maps $\Delta_{B}$ and $\Delta_{C}$ extend uniquely to
  homomorphisms
  \begin{align*} 
\Delta_{B}
    &\colon M(A) \to \End({_{(A \oo A)}(\ArA)})^{\op}, &
    \Delta_{C} &\colon M(A) \to \End((\AlA)_{(A \oo A)})
  \end{align*}
  such that for all $T\in M(A)$ and $a,b,c\in A$,
  \begin{align*}
(a\otimes b)\Delta_{B}(c)    \Delta_{B}(T) &= (a\otimes b) \Delta_{B}(cT), &
\Delta_{C}(T)\Delta_{C}(a)(b\otimes c) &= \Delta_{C}(Ta)(b\otimes c).
  \end{align*}
\end{lemma}
\begin{proof}
  Uniqueness and existence follow easily from bijectivity of the canonical maps, for example, $\Delta_{C}(T)$ has to coincide with and can be defined as $T_{\rho}(T \otimes 1)T_{\rho}^{-1}$.
\end{proof}

To make sense of the right hand sides in \eqref{eq:mor-delta-b} and \eqref{eq:mor-delta-c},  note that
\begin{align*}
  (\pi \otimes \pi)((d' \otimes d'')\Delta_{E}(d)) \in M(A)_{M(B)} \otimes {^{M(B)}M(A)}, \\
  (\pi \otimes \pi)(\Delta_{F}(d)(d' \otimes d'')) \in M(A)^{M(C)} \otimes {_{M(C)}M(A)}
\end{align*}
are well-defined by  \eqref{eq:mor-delta-b} and \eqref{eq:mor-delta-c}, so that we can multiply on the left or on the right, respectively, by $a' \otimes a'' \in A\otimes A$  to obtain well-defined elements in
$A_{M(B)} \otimes {^{M(B)}}A = \ArA$ or $A^{M(C)} \otimes A_{M(C)} = \AlA$, respectively.

By now, the left and the right hand sides of  \eqref{eq:mor-delta-b} and \eqref{eq:mor-delta-c} as well as the left hand sides in \eqref{eq:mor-delta} are well-defined. Let us now turn to the right hand sides in \eqref{eq:mor-delta}.

We write 
$A_{E}$, $^{E}A$, $_{F}A$, $A^{F}$ when we regard $A$ as a module over $E$ or $F$ such that
\begin{align*}
  a\cdot e&:= a\pi(e), & e\cdot a &:= a \pi(t_{E}(e)),   & f\cdot a& := \pi(f)a, & a\cdot f &:= \pi(t_{F}(f))a
\end{align*}
for all $a\in A$, $e\in E$ and $f\in F$. 
\begin{lemma}
There exist well-defined homomorphisms
  \begin{align*}
    \pi {_{E}\overline{\times}^{E}} \pi &\colon D_{E} \overline{\times} {^{E}D} \to \End(_{(A\otimes A)}(A_{E} \otimes {^{E}A}))^{\op}, \\
\pi {^{F}\overline{\times}_{F}} \pi &\colon D^{F}\overline{\times} {_{F}D} \to \End((A^{F} \otimes {_{F}A})_{(A\otimes A)})   
  \end{align*}
such that for all $t \in  D^{F}\overline{\times} {_{F}D}$, $s \in D_{E} \overline{\times} {^{E}D}$, $d,d'\in D$ and $a,a'\in A$,
\begin{align*}
(a'\pi(d) \otimes a\pi(d'))  (    \pi {_{E}\overline{\times}^{E}} \pi)(w) &= (a\otimes a')(\pi\otimes \pi)((d\otimes d')w), \\
  (\pi {^{F}\overline{\times}_{F}} \pi)(w)(\pi(d)a\otimes \pi(d')a') &= (\pi \otimes \pi)(w(d\otimes d'))(a\otimes a').
\end{align*}
\end{lemma}
\begin{proof}
We only prove the assertion concerning $\pi {^{F}\overline{\times}_{F}} \pi$; the other map can be treated similarly. Let $w\in D^{F} \overline{\times} {_{F}D}$. By definition of $D^{F}\overline{\times} {_{F}D}$,  
\begin{align*}
 (\pi \otimes \pi)(w(d \otimes 1))(a \otimes \pi(d')a') 
 = (\pi \otimes \pi)(w(1\otimes d'))(\pi(d)a \otimes a')
\end{align*}
for all $d,d'\in D$ and $a,a'\in A$. Hence,
there exists a well-defined linear map from $A\times A$ to $A^{F} \otimes {_{F}A}$  that sends every element of the form $(\pi(d)a,\pi(d')a')$  to the expression above.  The corresponding linear map from $A\otimes A$ to $A^{F} \otimes {_{F}A}$ descends  $A^{F} \otimes {_{F}A}$ because $w(t_{F}(u) \otimes 1) = w(1\otimes u)$ for all $u \in F$ by definition of $D^{F} \overline{\times} {_{F}D}$.
\end{proof}
Now also  the right hand sides in \eqref{eq:mor-delta} are well-defined, but they do not lie in the same space as the left hand sides. To remedy this, note that the quotient maps from $A\otimes A$ to $\AlA$ or $\ArA$, respectively, factorize trough the quotient maps
\begin{align} \label{eq:mor-quotient-maps}
q_{E} \colon  A_{E}\otimes{^{E}A} \to A_{B} \otimes {^{B}A} \quad \text{and} \quad  q_{F} \colon A^{F} \otimes {_{F}A} \to A^{C} \otimes {_{C}A},
\end{align}
 which induce canonical maps  
 from   $  \End({_{(A\otimes A)}(A_{B} \otimes {^{B}A})})$ and $  \End({_{(A\otimes A)}(A_{E} \otimes {^{E}A})}) $ to
 \begin{align} \label{eq:hom-eb}
  \Hom({_{(A\otimes A)}(A_{E} \otimes {^{E}A})}, {_{(A\otimes A)}(A^{B} \otimes {_{B}A})}),
 \end{align}
 and from $ \End((A^{C} \otimes {_{C}A})_{(A\otimes A)}) $ and $ \End((A^{F} \otimes {_{F}A})_{(A\otimes A)}) $ to
 \begin{align} \label{eq:hom-fc}
   \Hom((A^{F} \otimes {_{F}A})_{(A\otimes A)}, (A^{C} \otimes {_{C}A})_{(A \otimes A)}).
 \end{align}
To keep the notation in reasonable bounds, we shall apply these maps
where needed without explicit mentioning, for example, to make sense
of the two equations in \eqref{eq:mor-delta}. 
\begin{definition} \label{definition:morphism}
Let $\mathcal{D}=(D,E,F,t_{E},t_{F},\Delta_{E},\Delta_{F})$ and  $\mathcal{A}=(A,B,C,t_{B},t_{C},\Delta_{B},\Delta_{C})$ be regular multiplier Hopf algebroids.
  A \emph{morphism} from $\mathcal{D}$ to $\mathcal{A}$ is a non-degenerate homomorphism $\pi\colon D\to M(A)$ satisfying \eqref{eq:mor-base-b}, \eqref{eq:mor-base-c}, and  \eqref{eq:mor-delta}.
\end{definition}
Examples of morphisms will be given in Section \ref{section:examples}. In particular, we shall use morphisms  into duals of  measured regular multiplier Hopf algebroids to describe the comultiplications on  these duals in a succinct way.  

Let us now suppose that $(\mathcal{A},\mu_{B},\mu_{C},\bpsib,\cphic)$ is a measured regular multiplier Hopf algebroid and
 explain how one can test whether a homomorphism $\pi\colon D\to M(\hat{A})$  is a morphism from $\mathcal{D}$ to the dual $\widehat{\mathcal{A}}$ constructed in Section \ref{section:dual}.  

 We first show that  the extended comultiplications $\Delta_{B}$ and $\Delta_{C}$ on $M(A)$ are dual to the multiplication on $\hat A$ in a natural sense, but then apply this result with $\mathcal{A}$ and $\widehat{\mathcal{A}}$ switched.  Recall that  $L(M)$ and $R(M)$ denote the left and the right multipliers of a module $M$.   Similar arguments as in the case of Lemma \ref{lemma:pairing} show: 
\begin{lemma}
 There exist unique pairings
  \begin{align*}
    L((\AlA)_{(A\otimes A)}) \times (\hat A \otimes \hat A) \to \C
    \quad\text{and}\quad R(_{(A\otimes A)}(\ArA)) \times (\hat A \otimes
    \hat A) \to \C
  \end{align*}
  such that for all $T\in M(A)$, $a,b\in A$, $\upsilon,\omega \in \hat
  A$,
  \begin{align*}
    (T \mid  a\cdot \upsilon \otimes b\cdot \omega) &= (T(a\otimes b)\mid \upsilon \otimes \omega) = (\upsilon \underset{\mu_{C}}{\otimes} \omega)(T(a\otimes b)), \\
    (T\mid \upsilon\cdot a \otimes \omega \cdot b) &= ((a\otimes
    b)T\mid \upsilon \otimes \omega) = (\upsilon
    \underset{\mu_{B}}{\otimes} \omega)((a\otimes b)T).
  \end{align*}
  These pairings separate the points of $L((\AlA)_{(A\otimes
    A)})$ and of $ R(_{(A\otimes A)}(\ArA))$.
\end{lemma}
\begin{proposition} \label{proposition:morphism-dual}
Let $(\mathcal{A},\mu_{B},\mu_{C},\bpsib,\cphic)$ be a measured regular multiplier Hopf algebroid. Then 
  \begin{align*}
 (\Delta_{B}(T)\mid \upsilon \otimes \omega) = (T \mid  \upsilon \omega)  =    (\Delta_{C}(T)\mid \upsilon \otimes \omega)
  \end{align*}
for all $a \in M(A)$ and $\upsilon,\omega \in \hat A$.
\end{proposition}
\begin{proof}
  We only prove the second equation. Write $\upsilon=a\cdot \phi$ and $\omega=b\cdot \psi$ with $a,b\in A$ and choose $d_{i},e_{i} \in A$ such that $a\otimes b = \sum_{i} \Delta_{C}(d_{i})(e_{i} \otimes 1)$ in $\AlA$. Then
  \begin{align*}
    (\Delta_{C}(T)\mid  \upsilon \otimes  \omega) &=
\sum_{i}    (\phi \underset{\mu_{C}}{\otimes} \phi)(\Delta_{C}(Td_{i})(e_{i}\otimes 1)\\
&= \sum_{i} \phi(\cphic(Td_{i})e_{i}) = \sum_{i}\phi(Td_{i}\cphic(e_{i})).
  \end{align*}
For $T\in A$, the left hand side becomes $T^{\vee}(\upsilon \cdot \omega)$, whence $ \sum_{i}\phi(-d_{i}\cphic(e_{i})) )=\upsilon \cdot \omega$.
\end{proof}

\begin{corollary}
Let $(\mathcal{A},\mu_{B},\mu_{C},\bpsib,\cphic)$ be a measured regular multiplier Hopf algebroid, let $\mathcal{D}$ be a regular multiplier Hopf algebroid as above and
  let $\pi \colon D \to M(\hat{A})$ be a non-degenerate homomorphism satisfying \eqref{eq:mor-base-b} and \eqref{eq:mor-base-c}. Then $\pi$ is a morphism from $\mathcal{D}$ to $\hat{\mathcal{A}}$  if and only if 
for all $d,d',d'' \in D$ and $a,b\in A$,
  \begin{align}
((\pi \otimes \pi)((d' \otimes d'')\Delta_{E}(d))\mid a'\otimes a'') &=
\pi(d)((a' \actleft \pi(d')) (a'' \actleft \pi(d''))), \label{eq:morphism-test-deltart} \\  \label{eq:morphism-test-deltalt}
((\pi \otimes \pi)(\Delta_{F}(d)(d' \otimes d''))\mid a' \otimes a'') &=   \pi(d)((\pi(d') \actright  a') (\pi(d'') \actright a'')).
  \end{align}
\end{corollary}
\begin{proof}
By Proposition \ref{proposition:morphism-dual},
  \begin{align*}
\pi(d)((\pi(d')\actright a') ( \pi(d'') \actright a'')) &=
(\hat{\Delta}_{B}(\pi(d))\mid  (\pi(d')\actright  a') \otimes   (\pi(d'')\actright a'' )) \\
&= (\hat{\Delta}_{B}(\pi(d))(\pi(d') \otimes \pi(d''))\mid a' \otimes a'').
  \end{align*}
  Thus, \eqref{eq:morphism-test-deltalt} is equivalent to \eqref{eq:mor-delta-c}. A similar argument shows that \eqref{eq:morphism-test-deltart} is equivalent to \eqref{eq:mor-delta-b}.
\end{proof}

\subsection{Morphisms preserve the antipode and counits}
Let $\mathcal{A}$ and $\mathcal{D}$ be regular multiplier Hopf algebroids as in \ref{eq:morphisms-a-d} again and let $\pi\colon D \to M(A)$ be a morphism from $\mathcal{D}$ to $\mathcal{A}$. Denote by $\ceps$ and $\epsb$ the left and the right counit of $\mathcal{A}$, and by $\epse$ and $\feps$ the left and the right counit of $\mathcal{D}$.
\begin{lemma}
  For all $a\in A$ and $d\in D$,
  \begin{align*}
    \epsb(\pi(d)a) &= \epsb(\pi(\epse(d))a) &&\text{and} &
    \ceps(a\pi(d)) &= \ceps(a\pi(\feps(d))).
  \end{align*}
\end{lemma}
\begin{proof}
We only prove the second equation, and use a similar commutative diagram as in the proof of Proposition 3.5 (2) in \cite{timmermann:regular}:
  \begin{align*} \xymatrix@C=35pt@R=18pt{A \otimes D \otimes A
      \ar[r]^{\id \otimes \widetilde{T_{\rho}}^{\pi}} \ar[rd]_{\id \otimes
      m_{D,A}} \ar[dd]_{m_{A,D} \otimes
        \id} & A \otimes D^{F} \otimes {_{F}A} \ar[d]^{\id \otimes (\feps
        \odot \id)} \ar[r]^{(\widetilde{T_{\rho}}^{A})_{13}} & A^{C}_{E} \oo {_{E}D} \oo \cA \ar
      `r/0pt[r] [rdd]^{m_{E}
        \otimes \id} \ar[d]^{(\id \odot \feps) \otimes \id} & \\
      & A \otimes A \ar[d]_{m}  \ar[r]^{\widetilde{T_{\rho}}^{A}} & \AlA
      \ar[d]^{\ceps \odot \id} & \\  A\otimes A
      \ar `d/0pt[d] `r[rrrd]^{\widetilde{T_{\rho}}^{A}}  [rrr] \ar[r]^{m} &A \ar@{=}[r] &A &  \AlA \ar[l]_{\ceps
        \odot \id}  \\ &&&& }
  \end{align*} 
Here, $m_{A,D}$, $m_{D,A}$ and $m_{E}$ denote the  obvious multiplication maps, and
 $\widetilde{T_{\rho}}^{\pi}$ is defined by $d\otimes \pi(e)b \mapsto (\id \otimes \pi)(\Delta_{F}(d)(1\otimes e))(1\otimes b)$. All cells except for the upper right one commute by the assumptions on $\ceps,\feps$ and $\pi$. Since $\widetilde{\Tr}$ is surjective, we can conclude that the last cell must commute, which proves the assertion.
\end{proof}
\begin{remark}
  In informal Sweedler notation, the argument can be written as follows.
 Fix $a,b\in A$ and $d,e\in D$  and write $\Delta_{C}(a)$ and $\Delta_{F}(d)$ as $a_{(1)} \otimes a_{(2)}$ and $d_{(1)} \otimes d_{(2)}$, respectively.  The counit relation, applied to  $\Delta_{C}(a\pi(d))$, $\Delta_{C}(a)$ and $\Delta_{F}(d)$, implies
\begin{align*}
  \ceps(a_{(1)}\pi(d_{(1)}))a_{(2)}\pi(d_{(2)})\pi(e)b =  a\pi(d)\pi(e)b 
&= \ceps(a_{(1)})a_{(2)}\pi(\feps(d_{(1)}))\pi(d_{(2)}e)b.
\end{align*}
Since
 the canonical maps of $\mathcal{D}$ are bijective and $\pi(\feps(D)) \subseteq M(C)$, we obtain
 \begin{align*}
   \ceps(a_{(1)}\pi(d'))\pi(e')b &=\ceps(a_{(1)})a_{(2)}\pi(\feps(d'))\pi(e')b = \ceps(a_{(1)}t_{B}(\pi(\feps(d'))))a_{(2)}\pi(e')b
 \end{align*}
 for all $d',e'\in D$.
 Now, we use the relation $\pi(D)A=A$, bijectivity of the canonical maps of $\mathcal{A}$, and the relation $\ceps(a't_{B}(y))=\ceps(a'y)$, where $y\in C$, and conclude that 
 \begin{align*}
   \ceps(a'\pi(d'))b' = \ceps(a't_{B}(\pi(\feps(d'))))b' = \ceps(a'\pi(\feps(d')))b'
 \end{align*}
for all $a',b'\in A$ and $d'\in D$.
\end{remark}

\begin{theorem} \label{theorem:morphism-antipode}
  Let $  \mathcal{A}=(A,B,C,t_{B},t_{C},\Delta_{B},\Delta_{C})$ and $\mathcal{D}=(D,E,F,t_{E},t_{F},\Delta_{E},\Delta_{F})$ be regular multiplier Hopf algebroids with antipodes $S_{A}$ and $S_{D}$, respectively, and let $\pi\colon D \to M(A)$ be a morphism from $\mathcal{D}$  to $\mathcal{A}$. Then 
  $\pi \circ S_{D} = S_{A} \circ \pi$.
\end{theorem}
\begin{proof}
  Let us first outline the argument using informal Sweedler notation. 
Fix $a,b\in A$ and $d\in D$. By the lemma above,
  \begin{align*}
    S_{A}(\pi(d_{(1)})a_{(1)}) \pi(d_{(2)})a_{(2)}b &= \epsb(\pi(d)a)b = \epsb(\pi(\epse(d))a)b = S_{A}(a_{(1)})\pi(\epse(d))a_{(2)}b.
  \end{align*}
Since the canonical maps of $\mathcal{A}$ and $\mathcal{D}$ are surjective, we can first conclude that 
\begin{align*}
  S_{A}(\pi(d_{(1)})a')\pi(d_{(2)}e)b' = S_{A}(a')\pi(\epse(d)e)b' = S_{A}(a')\pi(S_{D}(d_{(1)})d_{(2)}e)b'
\end{align*}
for all $a',b'\in A$ and $e\in D$, and then that
$S_{A}(a')S_{A}(\pi(d'))b''  = S_{A}(a') \pi(S_{D}(d'))b''$
for all $d'\in A$ and $b''\in A$.

To make this more precise, consider the following diagram:
\begin{align*}
  \xymatrix@C=32pt{
&    A_{B}^{F} \otimes {^{F}D} \otimes \bA \ar[r]^{(T_{\rho}^{A})_{13}} 
\ar `l/0pt[l] `d[ldddd]^{m^{\op}_{E}\otimes \id}  [dddd]   \ar[d]_{(\id \odot t_{E}\circ\epse) \otimes \id}&
    A^{C} \otimes D_{E} \otimes {_{C,E}A} \ar[r]^(0.45){(T_{\rho}^{\pi})_{23}} \ar[d]_{\id \otimes (\epse \odot \id)} &
    A^{C} \otimes {^{C}M(A)^{C}} \otimes \cA \ar `r/0pt[r] `d[rdddd]_{\bar m_{C}^{\op} \otimes \id} [dddd] \ar[ld]^{\id \otimes \bar m_{C}(S_{A}\otimes \id)} \ar[dd]_{S_{A} \otimes S_{A} \otimes \id} & \\
&\AbA \ar[r]^{T_{\rho}^{A}} \ar[dd]_{\epsb \odot \id} & \AlA  \ar[d]^{S_{A} \otimes \id}  &&
\\
& & \AcA \ar[ld]_{m_{C}} & \ar[l]_(0.55){\id \otimes \bar m_{C}} \Ac \oo {_{C}M(A)_{C}} \oo \cA \ar[d]^{\bar m_{C} \otimes \id} &
\\
& A && \AcA \ar[ll]_{m_{C}}& \\
& \AbA \ar[u]^{\epsb\odot \id} \ar[rr]^{T_{\rho}^{A}} && \AlA \ar[u]_{S_{A} \otimes \id}&
}
\end{align*}
Here, $m_{E}^{\op}$, $\bar m_{C}$ and $\bar m_{C}^{\op}$ denote the obvious multiplication maps, and $T^{\pi}_{\rho}$ maps an element  $d\otimes \pi(e)b$ to $(\pi \otimes \pi)(\Delta_{F}(d)(1\otimes e))(1\otimes b)$. The outer cell commutes because $\pi$ is a morphism, the cell on the left hand side commutes by the preceding lemma,  and all of the other cells except for the triangle commute as well.  Since $T_{\rho}^{A}$ is surjective, we can conclude that this triangle commutes as well, whence
\begin{align*}
  \pi(\epse(d))\pi(e)a = \bar m_{C}(S_{A} \otimes \id)(T^{\pi}_{\rho}(d\otimes \pi(e)a))
\end{align*}
for all $d,e\in D$ and $a\in A$.
Writing $\Delta_{F}(d)(1\otimes e)=\sum_{i} d'_{i} \otimes e'_{i}$, this relation  takes the form
\begin{align*}
\sum_{i}\pi(S_{D}(d'_{i})e'_{i})a= \sum_{i} S_{A}(\pi(d'_{i}))\pi(e'_{i})a.
\end{align*}
Since $T^{D}_{\rho}$ is surjective, the claim follows.
\end{proof}

\section{Examples}
\label{section:examples}

As a special case, the duality theory developed in the preceding sections restricts to a duality of regular weak multiplier Hopf algebras with integrals, which correspond with certain measured regular multiplier Hopf algebras.  We carefully explain this relation, summarizing and extending the results in \cite{daele:relation}, and give a self-contained description of the resulting duality for weak multiplier Hopf algebras.

Next, we  determine the dual objects for measured regular multiplier Hopf algebroids associated  in \cite{timmermann:integration} to \'etale, locally compact, Hausdorff groupoids, to actions of Hopf algebras,  and to   braided-commutative Yetter-Drinfeld algebras.  To keep the treatment in reasonable bounds, we introduce a few simplifying assumptions, for example, to avoid a discussion of Radon-Nikodym derivatives. The latter will be taken up in \cite{timmermann:opalg}.

\subsection{Duality of weak multiplier Hopf algebras with integrals}

A regular weak multiplier Hopf algebra \cite{daele:weakmult} consists of a non-degenerate, idempotent algebra $A$ and a
homomorphism $\Delta \colon A\to M(A\otimes A)$ satisfying the following conditions:
\begin{enumerate}
\item for all $a,b\in A$, the following elements belong to $A\otimes A$:
      \begin{align*} &\Delta(a)(1 \otimes b), && \Delta(a)(b\otimes 1), & &(1 \otimes
b)\Delta(a), && (b \otimes 1)\Delta(a);
    \end{align*}
  \item $\Delta$ is coassociative in the sense that for all $a,b,c\in A$,
      \begin{align*} (a\otimes 1 \otimes 1)(\Delta \otimes \iota)(\Delta(b)(1 \otimes c)) =
(\iota\otimes \Delta)((a \otimes 1)\Delta(b))(1 \otimes 1 \otimes c);
    \end{align*}
  \item $\Delta$ is full in the sense that there are no strict subspaces $V,W\subset A$
satisfying
      \begin{align*} 
        \Delta(A)(1 \otimes A) &\subseteq V\otimes A &&\text{or} & (A \otimes
1)\Delta(A) &\subseteq A\otimes W;
    \end{align*}
  \item there exists an idempotent $E \in M(A \otimes A)$ such that
      \begin{align*} \Delta(A)(1 \otimes A) &=E(A \otimes A) = \Delta(A)(A\otimes 1)
\end{align*}
and
\begin{align*} (A \otimes
1)\Delta(A)&=(A\otimes A)E = (1\otimes A)\Delta(A);
    \end{align*}
  \item the idempotent $E$ in condition (4) satisfies
      \begin{align*} (\Delta \otimes \iota)(E) =(E \otimes 1)(1\otimes E)=(1 \otimes
E)(E\otimes 1) = (\iota \otimes \Delta)(E),
    \end{align*}
   where $\Delta \otimes \iota$ and $\iota \otimes \Delta$ are extended
to homomorphisms $M(A \otimes A) \to M(A \otimes A \otimes A)$ such that $1 \mapsto E
\otimes 1$ or $1 \mapsto 1 \otimes E$, respectively;
\item there exist idempotents $F_{1},F_{3}\in M(A \otimes A^{\op})$ and $F_{2},F_{4} \in M(A^{\op}
\otimes A)$ such that
     \begin{align*}
 E_{13}(F_{1} \otimes 1) &=E_{13}(1 \otimes E),  & (1 \otimes
F_{2})E_{13} &= (E \otimes 1)E_{13}, \\
 (F_{3} \otimes 1)E_{13} &=(1 \otimes E)E_{13},  & E_{13}(1 \otimes
F_{4}) &= E_{13}(E \otimes 1),
    \end{align*}
   where $E_{13} \in M(A \otimes A\otimes A)$ acts like $E$ on the first
and third tensor factor, and such that the kernels of the linear maps
$T_{1},T_{2},T_{3},T_{4} \colon A\otimes A \to A\otimes A$ defined by
     \begin{align*}
      T_{1}(a\otimes b)&=
\Delta(a)(1 \otimes b), & T_{2}(a\otimes b)&=
(a \otimes 1)\Delta(b), \\
T_{3}(a\otimes b)&=
(1 \otimes b)\Delta(a), & T_{4}(a\otimes b)&=
\Delta(b)(a \otimes 1)
    \end{align*}
are given by
\begin{align*}
  \ker(T_{i})& =(A \otimes 1)(1-F_{i})(1\otimes A)  \quad \text{for } i=1,2, \\
\ker(T_{i}) &=(1 \otimes A)(1-F_{i})(A\otimes 1)  \quad \text{for } i=3,4,
\end{align*}
 \item there exists a linear map $\varepsilon\colon A\to \C$ called the \emph{counit} such
that for all $a,b\in A$,
     \begin{align*} (\varepsilon \odot \iota)(\Delta(a)(1\otimes b)) &= ab= (\iota
\odot \varepsilon)((a \otimes 1)\Delta(b)).
    \end{align*}
 \end{enumerate}

Let  $(A,\Delta)$ be regular weak multiplier Hopf algebra. Then
the comultiplication extends to a homomorphism $\Delta\colon M(A) \to
M(A\otimes A)$ such that $1\mapsto E$. Moreover, there exists an
antipode, which is a bijective anti-homomorphism $S\colon A\to A$ such that the maps
  \begin{align*}
    \begin{aligned} R_{1} &\colon A\otimes A \to A\otimes A, & a\otimes b &\mapsto a_{(1)} \otimes S(a_{(2)})b, \\ R_{2} &\colon A\otimes A\to A\otimes A, & a \otimes b
&\mapsto aS(b_{(1)}) \otimes b_{(2)}
    \end{aligned}
  \end{align*}
 are well-defined and satisfy $T_{i}R_{i}T_{i}=T_{i}$ and
$R_{i}T_{i}R_{i}=R_{i}$ for $i=1,2$; see \cite{daele:weakmult}.  
Using this antipode, one defines source and target maps
$\varepsilon_{s},\varepsilon_{t}\colon A\to M(A)$ by
\begin{align*}
 \varepsilon_{s}(a) &= S(a_{(1)})a_{(2)}, & \varepsilon_{t}(a) &=a_{(1)}S(a_{(2)}).
\end{align*}

For the purpose of this section, we adopt the following terminology.
\begin{definition} \label{def:frobenius}
  A \emph{Frobenius tuple} is a pair of algebras $B,C$ with anti-isomorphisms $t_{B}\colon B\to C$ and $t_{C} \colon C\to B$ and functionals $\mu_{B} \in B^{\vee}$  and $\mu_{C} \in C^{\vee}$ such that there exists an idempotent $E \in M(B\otimes C)$ satisfying the following conditions:
  \begin{enumerate}
  \item $E(1\otimes y) = E(t_{C}(y) \otimes 1)$  and $(x\otimes 1)E = (1\otimes t_{B}(x))E$ for all  $y\in C$ and $x\in B$;
    \item  $m(\id \otimes t_{B}^{-1})(E) = 1$ in $M(C)$ and $m(t_{C}^{-1} \otimes \id)(E) = 1$ in $M(B)$;
  \item  $(\mu_{B} \odot \id)(E)=1$ in $M(C)$ and  $(\id \odot \mu_{C})(E)=1$ in $M(B)$.
  \end{enumerate}
\end{definition}
Suppose given a Frobenius tuple with $E\in M(B\otimes C)$ as above. Then $E$ is a \emph{separability idempotent} in the sense of \cite{daele:separability}, and 
\begin{itemize}
\item[(F1)] the algebras $B$ and $C$ have local units \cite[Proposition
  1.10]{daele:separability}; in particular, they are firm;
   \item[(F2)] every idempotent $B$-submodule or $C$-submodule
   splits \cite[Section 4]{daele:separability}; in particular, every
   idempotent module over $B$ or $C$ is locally projective.
\end{itemize}
Let $A$ be a non-degenerate, idempotent algebra with $B,C \subseteq M(A)$ as commuting subalgebras such that $AB=BA=AC=CA=A$. We regard $A$ as a module over $B$ and $C$ as in Section \ref{section:multiplier-bialgebroids}, and denote by
 \begin{align*}
  \pi_{B} \colon A\otimes A \to \ArA \quad \text{and} \quad
  \pi_{C} \colon A\otimes A \to \AlA
\end{align*}
the canonical quotient maps.
\begin{lemma} \label{lemma:frobenius}
\begin{enumerate}
\item Every functional on $A$ is factorizable, that is,  $A^{\sqcup}=A^{\vee}$.
\item  There are well-defined linear maps
\begin{align*}
  \iota_{B}& \colon \ArA \to A\otimes A, \ a\otimes b \mapsto (a\otimes b)E, \\
\iota_{C}& \colon \AlA \to A\otimes A, \ a\otimes b\mapsto E(a\otimes b);
\end{align*}
and $\pi_{B} \circ \iota_{B}$ and $\pi_{C} \circ \iota_{C}$ are the identity on $\ArA$ or $\AlA$, respectively.
\item For all $\upsilon,\omega \in A^{\vee}=A^{\sqcup}$, 
  \begin{align*}
    (\upsilon \otimes \omega) \circ \iota_{B} &= \upsilon \underset{\mu_{B}}{\otimes} \omega  \text{ in } (\ArA)^{\vee}, &
    (\upsilon \otimes \omega) \circ \iota_{C} &= \upsilon \underset{\mu_{C}}{\otimes} \omega \text{ in } (\AlA)^{\vee}.
  \end{align*}
\end{enumerate}  
\end{lemma}
\begin{proof}
  (1) As in \cite[Proof of Lemma 3.3]{timmermann:integration}, one  checks that for every $\omega \in A^{\vee}$, the formulas
  \begin{align} \label{eq:frobenius-factorize-b}
    \omega_{B} (a) &= (\omega \otimes t_{B}^{-1})((a\otimes 1)E), & {_{B}\omega}(a) &= (\omega \otimes t_{C})(E(a\otimes 1)),  \\  \label{eq:frobenius-factorize-c}
    \omega_{C}(a) &=  (t_{B} \otimes \omega)((1\otimes a)E), & \comega(a)&=(t_{C}^{-1} \otimes \omega)(E(1\otimes a))
  \end{align}
 define module maps $\omega_{B},{_{B}\omega},\omega_{C},{_{C}\omega}$ whose composition with $\mu_{B}$ or $\mu_{C}$, respectively, is $\omega$.

(2) Straightforward; see also \cite[Proposition 3.9]{daele:relation}.

(3) We only prove the first relation. Let $\upsilon,\omega \in A^{\vee}$ and $a,b\in A$. By \eqref{eq:frobenius-factorize-b},
\begin{align*}
  (\upsilon \underset{\mu_{B}}{\otimes} \omega)(a\otimes b)=  \omega(bt_{B}(\upsilon_{B}(a))) &= \omega(bt_{B}((\upsilon \otimes t_{B}^{-1})((a\otimes 1)E))) \\ & = \omega(b(\upsilon \otimes \id)((a\otimes 1)E)) \\ & = (\upsilon\otimes \omega)((a\otimes b)E) = (\upsilon\otimes \omega)(\iota_{B}(a\otimes b)). \qedhere
\end{align*}
 \end{proof}
To formulate our main results, we need the notion of integrals on weak multiplier Hopf algebras \cite{daele:larson}.
\begin{definition}
    Let $(A,\Delta)$ be a regular weak multiplier Hopf algebra. We call a functional $\phi$ on $A$ a
\emph{left integral on $(A,\Delta)$} if
\begin{align*}
  (\id \odot \phi)((b\otimes 1)\Delta(a)) &= (\id \odot \phi)((b\otimes 1)F_{2}(1\otimes a))  \quad \text{for all } a,b\in A,
\end{align*}
and a functional $\psi$ on $A$ a  \emph{right integral on $(A,\Delta)$} if
\begin{align*}
(\psi  \odot \id)(\Delta(a)(1\otimes b)) &= (\psi \odot \id)((a\otimes 1)F_{1}(1\otimes b))  \quad \text{for all } a,b\in A.
\end{align*}
A \emph{regular weak multiplier Hopf algebra with integrals} is a regular weak multiplier Hopf algebra $(A,\Delta)$ with a right integral $\psi$ and a left integral $\phi$ satisfying
\begin{align} \label{eq:wmha-full}
  \begin{aligned}
    (\phi \odot \id)((A\otimes 1)E) = C = (\phi\odot \id)(E(A\otimes 1)), \\
    (\id \odot \psi)(E(1\otimes A)) = B = (\id \odot
    \psi)((1\otimes A)E).
  \end{aligned}
\end{align}
 \end{definition}

\begin{theorem} \label{theorem:wmha-had}
  Let $(A,\Delta,\psi,\phi)$ be a regular weak multiplier Hopf algebra with integrals.  Then there exists
 a unique measured  regular multiplier Hopf algebroid
\begin{align*}
  (\mathcal{A},\mu_{B},\mu_{C},\bpsib,\cphic), \quad \text{where } \mathcal{A}=(A,B,C,t_{B},t_{C},\Delta_{B},\Delta_{C}),
\end{align*}
such that for all $a,b,c\in A$, $x\in B$, $y\in C$,
    \begin{gather*}
      \begin{aligned}
        B&=\varepsilon_{s}(A), & C&=\varepsilon_{t}(A), & t_{B}(x)&=S^{-1}(x),  & t_{C}(y)&=S^{-1}(y),
      \end{aligned}  \\
 (b\otimes c)\Delta_{B}(a)
      =\pi_{B} ((b\otimes c)\Delta(a))  \quad \text{and} \quad
      \Delta_{C}(a)(b\otimes c) = \pi_{C}(\Delta(a)(b\otimes c)), \\
      \mu_{B}(\varepsilon_{s}(a)) = \varepsilon(a) =
\mu_{C}(\varepsilon_{t}(a)),  \quad        
\mu_{B} \circ \bpsib = \psi,  \quad  \mu_{C} \circ \cphic = \phi.
\end{gather*}
Moreover,   $(B,C,t_{B},t_{C},\mu_{B},\mu_{C})$ is a Frobenius tuple and $\mathcal{A}$ is locally projective.
\end{theorem}
\begin{proof}
First, $\mathcal{A}$ is a regular multiplier Hopf algebroid
by \cite[Theorem 4.8]{daele:relation}, and locally projective by properties (F1) and (F2) noted after Definition \ref{def:frobenius}. The integrals $\phi$ and $\psi$ factorize through partial integrals $\cphic$ and $\bpsib$  by \cite[Lemma 3.3]{timmermann:integration} and its right-handed counterpart, $(\mu_{B},\mu_{C})$ is a counital base weight by \cite[Example 5.1]{timmermann:integration}, and  $\phi$ and $\psi$ are full by \eqref{eq:wmha-full}.
\end{proof}
To prove the converse, we need the following result.
 \begin{proposition} \label{theorem:local-units} Let $(A,B,C,t_{B},t_{C},\Delta_{B},\Delta_{C})$ be a regular multiplier Hopf algebroid and suppose that $(B,C,t_{B},t_{C},\mu_{B},\mu_{C})$ is a Frobenius tuple. Then  $A$ has local units.
\end{proposition}
\begin{proof}
  We  follow the proof of \cite[Proposition 2.2]{daele:actions}. Let $a_{1},\ldots,a_{n} \in A$. We suppose that the vector space $V=\{(aa_{1},\ldots,aa_{n}) : a\in A\} \subseteq A^{n}$ does not contain $(a_{1},\ldots,a_{n})$ and obtain a contradiction as follows. Choose $\omega^{(1)},\ldots,\omega^{(n)} \in A^{\vee}$ such that $\sum_{i} \omega^{(i)}(aa_{i})=0$ for all $a\in A$ but $\sum_{i} \omega^{(i)}(a_{i}) \neq 0$. Since $B$ has local units and $\epsb$ is surjective, we can find a $b\in A$ such that $\epsb(b)a_{i} = a_{i}$ for all $i$. Write
  \begin{align*}
  \Delta_{C}(b)(1\otimes a_{i}) =\sum_{j} p_{ij} \otimes q_{ij}
  \end{align*}
with $p_{ij},q_{ij} \in A$. Then by  \eqref{dg:antipode},
\begin{align*}
   \sum_{j} \omega^{(i)}(S(p_{ij})q_{ij}) = \omega^{(i)}(\epsb(b)a_{i}) = \omega^{(i)}(a_{i}).
\end{align*}
We shall show that  the left hand side is $0$ for all $i$, which is  a contradiction.  

Fix $i$. By (F1) and (F2), we can choose $\upsilon^{(k)}_{C} \in (A_{C})^{\vee}$ and $e^{(k)} \in A$ such that 
\begin{align*}
  \sum_{k} e^{(k)}\upsilon^{(k)}_{C}(S(p_{ij})) = S(p_{ij})
\end{align*}
  for all $j$.  Write $\omega^{(i,k)} = \omega^{(i)} \cdot e^{(k)}$ and $\upsilon^{(k)}=\mu_{C}\circ \upsilon^{(k)}_{C}$. Then
  \begin{align*}
    \sum_{j} \omega^{(i)}(S(p_{ij})q_{ij}) &= \sum_{j,k} \omega^{(i)}(e^{(k)}\upsilon^{(k)}_{C}(S(p_{ij}))q_{ij}) 
= \sum_{j,k}\omega^{(i,k)}(\upsilon^{(k)}_{C}(S(p_{ij}))q_{ij}).
\end{align*}
Since $\omega^{(i,k)} \in A^{\vee}=A^{\sqcup}$, we can rewrite this expression in the form
\begin{align*}
  \sum_{j,k} \upsilon^{(k)}(S(p_{ij}){_{C}\omega^{(i,k)}}(q_{ij}))
= \sum_{j,k} (\upsilon^{(k)} \circ S)(t_{C}({_{C}\omega^{(i,k)}}(q_{ij}))p_{ij}).
  \end{align*}
  But  $\sum_{j}t_{C}({_{C}\omega^{(i,k)}}(q_{ij}))p_{ij}=0$ for each $k$ because for every $\chi \in \hat{A}$,
  \begin{align*}
    \sum_{j}\chi(t_{C}({_{C}\omega^{(i,k)}}(q_{ij}))p_{ij}) &= (\chi \underset{\mu_{C}}{\otimes} \omega^{(i,k)})(\Delta_{C}(b)(1\otimes a_{i})) \\ &= \omega^{(i,k)}((b\triangleleft \chi)a_{i}) = \omega^{(i)}(e^{(k)}(b\triangleleft \chi)a_{i}) =0. \qedhere
  \end{align*}
\end{proof}

\begin{theorem} \label{theorem:had-wmha}
  Let $(\mathcal{A},\mu_{B},\mu_{C},\bpsib,\cphic)$ be a measured regular multiplier Hopf algebroid, where $\mathcal{A}=(A,B,C,t_{B},t_{C},\Delta_{B},\Delta_{C})$ and  where $(B,C,t_{B},t_{C},\mu_{B},\mu_{C})$ is a Frobenius tuple. Then there exists a unique homomorphism $\Delta\colon A\to M(A\otimes A)$ satisfying 
    \begin{align*}
      \Delta(a)(b\otimes c) = \iota_{C}(\Delta_{C}(a)(b\otimes c)) \quad \text{and} \quad
(b\otimes c)      \Delta(a) = \iota_{B}((b\otimes c)\Delta_{B}(a))
    \end{align*}
for all $a,b,c\in A$, and $(A,\Delta,\psi,\phi)$ is a regular weak multiplier Hopf algebra with integrals.
  \end{theorem}
  \begin{proof}
 Existence of $\Delta$ and the fact that $(A,\Delta)$ is a regular weak multiplier Hopf algebra follow from \cite[Theorem 5.11]{daele:relation}, $\phi$ and $\psi$ are a left and a right integral for $(A,\Delta)$ by \cite[Lemma 3.3]{timmermann:integration} and its right-handed counterpart, and \eqref{eq:wmha-full} follows from the relations $\bphi(A)=B=\phib(A)$, $\cpsi(A)=C=\psic(A)$ and \eqref{eq:frobenius-factorize-b}, \eqref{eq:frobenius-factorize-c}.
  \end{proof}
 As a corollary, we obtain the following duality of regular weak  multiplier Hopf algebras.
\begin{corollary}
  Let $(A,\Delta,\psi,\phi)$ be a regular weak multiplier Hopf algebra with integrals. Then there exists 
a regular weak multiplier Hopf algebra $(\hat{A},\hat{\Delta},\hat{\psi},\hat{\phi})$ with counit $\hat{\varepsilon}$ and antipode $\hat{S}$ such that
$\hat{A} = A\cdot \phi = \psi \cdot A \subseteq \dual{A}$  as a vector space and
  \begin{gather*}
    \begin{aligned}
      ((a\cdot \phi) (b\cdot \phi))(c) &= (\phi \otimes \phi) (\Delta(c)(a\otimes b)), \\
      (\hat{\Delta}(\upsilon)(1\otimes \omega)\mid a\otimes b) &= (\upsilon \otimes \omega\mid (a\otimes 1)\Delta(b)), \\
      ( (\upsilon \otimes 1)\hat{\Delta}(\omega)\mid a\otimes b) &=
      (\upsilon \otimes \omega\mid \Delta(a)(1\otimes b)),
    \end{aligned}
\\
  \hat{\phi}(\psi \cdot a) = \varepsilon(a) = \hat{\psi}(a\cdot \phi), \quad \hat{\varepsilon}(a\cdot \phi)=\phi(a), \quad \hat{S}(\upsilon) = \upsilon \circ S
\end{gather*}
for all  $a,b\in A$ and $\upsilon,\omega \in \hat{A}$, where $(-\mid -)$ denotes the canonical pairing of $\hat{A} \otimes \hat{A} \subseteq \dual{A} \otimes\dual{A}$ with $A\otimes A$. 
\end{corollary} 
\begin{proof}
We subsequently use  Theorem \ref{theorem:wmha-had}, Theorem \ref{theorem:dual-measured} and Theorem \ref{theorem:had-wmha} to pass from
 $(A,\Delta,\psi,\phi)$  to a  measured regular multiplier Hopf algebroid, its dual, and the associated regular weak multiplier Hopf algebra with integrals $(\hat{A},\hat{\Delta},\hat{\psi},\hat{\phi})$. The formulas given for the multiplication, comultiplication, integrals, counit and antipode  on $\hat{A}$ follow immediately from Lemma \ref{lemma:frobenius} (3) and Theorem \ref{theorem:dual-counits-antipode} and \ref{theorem:dual-measured}.
\end{proof}
In the situation above, we call $(\hat{A},\hat{\Delta},\hat{\psi},\hat{\phi})$ the \emph{dual} of $(A,\Delta,\psi,\phi)$. 
With the obvious notion of isomorphism between weak multiplier Hopf algebras with integrals,
we obtain the following biduality result as a corollary to Theorem \ref{theorem:biduality}:
\begin{corollary}
  Let $(A,\Delta,\psi,\phi)$ be a regular weak multiplier Hopf algebra with integrals and denote by $(\hat{\hat{A}},\hat{\hat{\Delta}},\hat{\hat{\psi}},\hat{\hat{\phi}})$ the dual of $(\hat{A},\hat{\Delta},\hat{\psi},\hat{\phi})$. Then the natural map $A \to \dual{\hat{A}}$ restricts to an isomorphism $A  \to \hat{\hat{A}}$ of regular weak multiplier Hopf algebras with integrals.
\end{corollary}

\subsection{\'Etale groupoids}

Let $G$ be a locally compact,  Hausdorff groupoid which is \emph{\'etale} in the sense that the source and the target maps $s,t\colon G\to G^{0}$ are open and local homeomorphisms \cite{renault}, and
let $\mu$ be a  Radon measure $\mu$ on  the space of units $G^{0}$ with full support. Suppose that $\mu$ is \emph{continuously quasi-invariant} in the sense of \cite[Example 3.6]{timmermann:integration}, that is, for every open $G$-set $U\subseteq G$, the two measures on $U$ obtained by pulling back $\mu$ along $t$ or $s$ are related by a continuous Radon-Nikodym derivative.  

Then  there exists a measured  multiplier Hopf $*$-algebroid 
\begin{align*}
  (\mathcal{A},\mu_{B},\mu_{C},\bpsib,\cphic), \quad \text{where } \mathcal{A}=(A,B,C,t_{B},t_{C},\Delta_{B},\Delta_{C}),  
\end{align*}
as follows \cite[Example 5.2]{timmermann:integration}.
First,
\begin{align*}
 A&=C_{c}(G), & B&=s^{*}(C_{c}(G^{0})), & C &=t^{*}(C_{c}(G^{0})),  &
  t_{B} \circ s^{*} &= t^{*}, & t_{C} \circ t^{*} = s^{*},
\end{align*}
where $C_{c}(G)$ and $C_{c}(G^{0})$ denote the algebras of compactly supported continuous functions on
$G$ and on $G^{0}$, respectively, and  $s^{*},t^{*} \colon C_{c}(G^{0}) \to M(C_{c}(G))$
the pull-back of functions along $s$ and $t$. The comultiplications $\Delta_{B}$ and $\Delta_{C}$ are given by
\begin{align*} (\Delta_{C}(f)(g \otimes h))(\gamma,\gamma')
=f(\gamma\gamma')g(\gamma)h(\gamma') = ((g\otimes h)\Delta_{B}(f))(\gamma,\gamma')
  \end{align*}
  for all $f,g,h\in A, \gamma,\gamma' \in G$, where we used the canonical
 isomorphisms
\begin{align*}
 \ctA \otimes \csA \cong C_{c}(\GstG) \cong
 \bAs\otimes \bAt
\end{align*}
that identify an elementary tensor $f\otimes g$ with the function $(\gamma,\gamma') \mapsto f(\gamma)g(\gamma')$, and $\GstG$ denotes the composable pairs of elements of $G$.  
Finally, 
\begin{align} \label{eq:base-weight-groupoid-functions}
  \mu_{B}(s^{*}(f)) = \int_{G^{0}} f \intd \mu = \mu_{C}(t^{*}(f))
\end{align}
and
\begin{align} \label{eq:integrals-groupoid-functions}
  (\cphic(f))(\gamma) &= \sum_{\substack{t(\gamma')=t(\gamma)}} f(\gamma'), &
  (\bpsib(f))(\gamma) &= \sum_{\substack{s(\gamma')=s(\gamma)}} f(\gamma')
\end{align}
for all $f\in C_{c}(G)$ and $\gamma\in G$.

 The total integrals $\phi:=\mu_{C} \circ \cphic$ and $\psi:=\mu_{B} \circ \bpsib$ correspond to integration with respect to the measures $\nu$ and $\nu^{-1}$ on $G$ defined by
   \begin{align} \label{eq:dnu}
  \int_{G} f \intd \nu  &=  \int_{G^{0}} \sum_{r(\gamma)=u} f(\gamma) \intd\mu(u), &
\int_{G} f \intd \nu^{-1} &=\int_{G^{0}} \sum_{s(\gamma) = u} f(\gamma) \intd\mu(u),
\end{align}
respectively,
and  the assumption that $\mu$ is continuously quasi-invariant is equivalent to the relation
$\nu  =  D\nu^{-1}$ for  some (unique) invertible $D \in C(G)$.

We show that the dual measured multiplier Hopf $*$-algebroid
\begin{align*}
  (\widehat{\mathcal{A}},\hat{\mu}_{C},\hat{\mu}_{B},{_{C}\hat{\psi}_{C}},{_{B}\hat{\phi}_{B}}), \quad \text{where }  \widehat{\mathcal{A}}=(\hat{A},C,B,\hat{t}_{C},\hat{t}_{B},\hat{\Delta}_{C},\hat{\Delta}_{B}),
\end{align*}
coincides with the  one associated to the convolution algebra of $G$ in \cite[Section  8.1]{timmermann:integration}.

By definition, the dual algebra $\hat{A}$ consists of all functionals on $C_{c}(G)$ of the form
  \begin{align*}
    \hat{f}:= fD^{-1/2}\cdot \phi  = fD^{1/2} \cdot \psi \colon g \mapsto \int_{G} fg d\nu^{0},
  \end{align*}
  where $f\in C_{c}(G)$ and $\nu^{0}=D^{-1/2}\nu=D^{1/2}\nu^{-1}$. The reason for introducing $D^{-1/2}$ in the definition of $\hat{f}$ will become clear in a moment. 
 
The multiplication on $\hat{A}$ is given by convolution, that is, for all $f,f' \in C_{c}(G)$,
\begin{align} \label{eq:gpd-convolution}
  \hat{f} \ast \hat{f'} = \hat{h}, \quad \text{where } h(\gamma'')=  \sum_{\gamma\gamma'=\gamma''}  f(\gamma)f'(\gamma') \text{ for all } \gamma\in G.
\end{align}
Indeed, short calculations show that for all $g\in C_{c}(G)$,
  \begin{align} \label{eq:gpd-factorise}
          ({_{B}\hat{f}}(g))(\gamma) &= 
      ({\hat{f}_{B}}(g))(\gamma) = \sum_{s(\gamma')=s(\gamma)} f(\gamma')D^{1/2}(\gamma')g(\gamma'), \\ \label{eq:gpd-act}
    (g \triangleleft \hat{f})(\gamma') &=\sum_{s(\gamma)=t(\gamma')}
    f(\gamma)D^{1/2}(\gamma)g(\gamma\gamma'),
  \end{align}
and  the last equation and the cocycle relation for $D$ imply
\begin{align*}
  ((g\triangleleft \hat{f}) \triangleleft \hat{f'})(\gamma'') = \sum_{\substack{s(\gamma)=t(\gamma') \\  s(\gamma')=t(\gamma'')}} f(\gamma)D^{1/2}(\gamma)f'(\gamma')D^{1/2}(\gamma') g(\gamma\gamma'\gamma'') = (g \triangleleft \hat{h})(\gamma'').
\end{align*}
 
 The involution on $\hat A$ is defined by 
  \begin{align*}
    (\hat f)^{*}(g) &= \overline{\int_{G} f\overline{S(g)} d\nu^{0}} =
    \int_{G} \overline{f(\gamma)} g(\gamma^{-1}) \intd\nu^{0}(\gamma) =
    \int_{G} \overline{f(\gamma^{-1})}g(\gamma)
    \intd\nu^{0}(\gamma),
  \end{align*}
whence
\begin{align} \label{eq:gpd-involution}
  (\hat{ f})^{*} &= \widehat{(f^{\dag})}, \quad \text{where }                   f^{\dag}(\gamma) =\overline{f(\gamma^{-1})}                   \text{ for all } \gamma\in G.
 \end{align}

  The embedding of $C$  into $M(\hat{A})$ is given by
  \begin{align*}
    t^{*}(h) \hat{f}  &= \widehat{t^{*}(h)f} =
    \hat h \ast \hat f, &
    \hat{f} t^{*}(h) &= 
    \widehat{s^{*}(h)f} = \hat f \ast \hat h 
   \end{align*}
   for all $f\in C_{c}(G)$ and $h\in C_{c}(G^{0})$, where $\ast$ denotes the convolution product. Hence, we can identify $C$ with $\widehat{C_{c}(G^{0})} \subseteq M(\hat{A})$. 
 Similar calculations show that the same is true for $B$ and that with respect to these identifications,
 $\hat S_{B}$ and $\hat S_{C}$ reduce to the identity.

Let us next compute the right comultiplication $\hat{\Delta}_{C}$ on $\hat{A}$.  
First, we can identify
   \begin{align} \label{eq:gpd-identify-hala}
      \hArA \cong C_{c}(\GssG) \quad \text{and} \quad \hAlA \cong C_{c}(\GttG)
   \end{align}
such that an elementary tensor $\hat f \otimes \hat g$ corresponds to the function $(\gamma,\gamma') \mapsto f(\gamma) g(\gamma')$, where $\GssG$ and $\GttG$ denote the pairs of all $(\gamma,\gamma')$ with $s(\gamma)=s(\gamma')$ or $t(\gamma)=t(\gamma')$, respectively.
Next, we use Proposition \ref{proposition:dual-canonical},  equation \eqref{eq:gpd-act}, and the cocycle property of $D$, and find
\begin{align*}
  ((\hat{f} \otimes 1)\hat{\Delta}_{C})(\hat{f'})&\mid g\otimes g') = 
\hat{f}'((g\triangleleft \hat{f})g') \\ 
&=\int_{G} \sum_{s(\gamma) = t(\gamma')} f(\gamma)D^{1/2}(\gamma)g(\gamma\gamma')g'(\gamma')f'(\gamma')d\nu^{0}(\gamma') \\
&=\int_{G} \sum_{s(\gamma'')=s(\gamma')} f(\gamma''\gamma'{}^{-1})D^{1/2}(\gamma'')D^{-1/2}(\gamma')g(\gamma'')g'(\gamma')f'(\gamma') d\nu^{0}(\gamma').
\end{align*}
On the other hand, the pairing $(\hArA) \times (A_{B} \otimes \bA) \to \C$  is given by
\begin{align*}
  (h\otimes h'\mid g\otimes g') &= \int_{G} \sum_{s(\gamma'')=s(\gamma')} h(\gamma'')g(\gamma'') h'(\gamma')D^{1/2}(\gamma'') g'(\gamma') d\nu^{0}(\gamma'),
\end{align*}
and comparing the two expressions we find that with respect to the identification  \eqref{eq:gpd-identify-hala},
\begin{align*}
  ((\hat{f} \otimes 1)\hat{\Delta}_{C}(\hat{f'})(\gamma'',\gamma') &= f(\gamma''\gamma'^{-1}) f'(\gamma')D^{-1/2}(\gamma').
\end{align*}
A similar calculation shows that
\begin{align*}
  (\hat{\Delta}_{B}(\hat{f})(1\otimes \hat{f'}))(\gamma,\gamma'') &= f(\gamma)f'(\gamma^{-1}\gamma'')D^{1/2}(\gamma).
\end{align*}
Thus, the multiplier Hopf $*$-algebroid $\widehat{\mathcal{A}}=(\hat{A},C,B,t_{B}^{-1},t_{C}^{-1},\hat{\Delta}_{C},\hat{\Delta}_{B})$ coincides with the modified multiplier Hopf $*$-algebroid in \cite[Section 8.1]{timmermann:integration}. Finally, the partial integrals $_{C}\hat{\psi}_{C}$ and $_{B}\hat{\phi}_{B}$ are given by
\begin{align*}
  _{C}\hat{\psi}_{C}(\hat{ f}) = \ceps(fD^{-1/2}) \equiv \widehat{f|_{G^{0}}} \equiv
   \epsb(fD^{1/2}) = { _{B}\hat{\phi}_{B}(\hat{f})},
\end{align*}
where we identified $B,C$ with the subalgebra $\widehat{C_{c}(G^{0})}$ of $M(\hat{ A})$ and used the relations $\ceps(g)=t^{*}(g|_{G^{0}})$ and $\epsb(g)=s^{*}(g|_{G^{0}})$ \cite[Example 2.3]{timmermann:integration}.

 \subsection{A two-sided crossed product}
 \label{subsection:two-sided}
Let $(H,\Delta_{H})$ be a regular multiplier Hopf algebra with a left integral $\phi_{H}$ and a right integral $\psi_{H}$, and denote its counit and antipode by $\varepsilon_{H}$ and $S_{H}$, respectively. Let $C$ be a unital algebra that is a left $H$-module algebra and write $h\triangleright y$ for the action of an element $h\in H$ on an element $y\in C$. Thus,
\begin{align*}
  h \triangleright (h' \triangleright y) &= (hh') \triangleright y, & h \triangleright (yy') &= (h_{(1)} \triangleright y)(h_{(2)} \triangleright y'), & h\triangleright 1_{C} &= \varepsilon_{H}(h)1_{C}
\end{align*}
for all $h,h' \in H$ and $y\in C$. We suppose that $H\triangleright C=C$. Finally,  let $\mu_{C}$ be a faithful functional on $C$ that is normalized in the sense that $\mu_{C}(1_{C})=1$, invariant with respect to the action of $H$ in the  sense that
$\mu_{C}(h\triangleright y) =\varepsilon_{H}(h)\mu_{C}(y)$
for all $h\in H$ and $y\in C$, and admits a modular automorphism $\sigma_{C}$. 

 Then one can define a measured regular multiplier  Hopf algebroid 
\begin{align*}
  (\mathcal{A},\mu_{B},\mu_{C},\bpsib,\cphic), \quad \text{where } \mathcal{A}=(A,B,C,t_{B},t_{C},\Delta_{B},\Delta_{C}),
\end{align*}
as follows \cite[Examples 2.7 and 5.6]{timmermann:integration}. First,
\begin{align*}
  B=C^{\op}, \quad   t_{C}(y) = y^{\op} \quad\text{and} \quad t_{B}(y^{\op})  = \sigma^{-1}(y) \quad\text{for all } y\in C.
\end{align*}
Note that $B$ carries a natural right action of $(H,\Delta_{H})$, given by
\begin{align*}
  x^{\op} \triangleleft h := (S_{H}^{-1}(h)  \triangleright x)^{\op}.
\end{align*}
  The algebra $A$ is the vector space $C \otimes H \otimes B$ with the multiplication given by
 \begin{align*}
 (y \otimes h \otimes x^{\op})(y' \otimes h' \otimes x'^{\op}) &= y(h_{(1)} \actright y') \otimes h_{(2)}h'_{(1)} \otimes (x^{\op} \actleft h'_{(2)})x'^{\op}.
 \end{align*}
 The algebras $C$ and $B$ embed naturally into $M(A)$, and we identify them with their images in $M(A)$. Moreover, we identify $H$ with the subalgebra $1_{C} \otimes H\otimes 1_{B} \subseteq A$.   The remaining ingredients of the measured  multiplier Hopf algebroid are given by
     \begin{gather*}
      (a\otimes a')\Delta_{B}(yhx) = ayh_{(1)} \otimes a'h_{(2)}x, \quad
      \Delta_{C}(yhx)(a \otimes a') = yh_{(1)}a \otimes h_{(2)}xa', \\
  \mu_{B}(y^{\op}) = \mu_{C}(y), \quad
    \bpsib(yhx) =\mu_{C}(y)\psi_{H}(h)x, \quad
 \cphic(yhx) =y\phi_{H}(h)\mu_{B}(x)
 \end{gather*}
for all $y\in C$, $x\in B$, $h\in H$ and $a,a'\in A$.
 The counits and antipode of $\mathcal{A}$ are given by
    \begin{align*}
       \ceps(x^{\op}hy) &=  (h \actright y)x, & \epsb(x^{\op}hy)
&=\sigma(y)^{\op}(x^{\op} \actleft h), &
S(yhx^{\op}) &= xS_{H}(h)\sigma(y)^{\op},
    \end{align*} 
and the modular automorphism $\sigma_{B}$ of $\mu_{B}$  is given by $\sigma_{B}(y^{\op})=\sigma_{C}^{-1}(y)^{\op}$ for all $y\in C$. 
Note that the canonical embedding of $H$ into $A$ is a morphism of multiplier Hopf algebroids.

This example goes back to \cite{connes:rankin}.  In the case where $C$  is a Frobenius algebra in the sense that it is part of a Frobenius tuple $(C,B,t_{B},t_{C})$, and that the action of $H$ is compatible with this tuple in a suitable sense, one can also endow the vector space $C \otimes H\otimes B$ with the structure of a   weak multiplier Hopf algebra, see  \cite[Proposition 3.6]{daele:weakmult2}. It is tedious to check but evident that the  associated multiplier Hopf algebroid  is isomorphic to the one constructed above.

Let us now determine  the dual multiplier Hopf algebroid
\begin{align*}
  \widehat{\mathcal{A}} = (\hat{A},C,B,t_{B}^{-1},t_{C}^{-1},\hat{\Delta}_{C},\hat{\Delta}_{B}).
\end{align*}
For trivial $H$, the dual algebra of the  weak (multiplier) Hopf algebra obtained in \cite[Proposition 3.6]{daele:weakmult2} is described nicely in \cite[Lecture 5]{daele-woronowicz}.

Denote by $\mathcal{K}$ the algebra spanned by all linear endomorphisms of $C$ of the form
\begin{align*}
  |x\rangle\langle y| \colon z \mapsto \mu_{C}(zy)x,
\end{align*}
where $x,y \in C$, and  by $(\hat{H},\Delta_{\hat{H}})$ the dual multiplier Hopf algebra of $(H,\Delta_{H})$ and by   $\varepsilon_{\hat{H}}$ and $\phi_{\hat{H}}$ its counit and the dual left integral on $\hat{H}$.
\begin{proposition} \label{proposition:chb-dual}
The dual algebra $\hat{A}$ is isomorphic to $\mathcal{K} \otimes \hat{H}$ via
  \begin{align*}
  hy  \cdot \phi \cdot x^{\op} \mapsto |y\rangle\langle x| \otimes h\cdot \phi_{H}.
  \end{align*}
  With respect to this identification, the embeddings of $B$ and of $C$  into $M(\hat{A})$ are given by
  \begin{align*}
  z^{\op} (|y\rangle\langle x| \otimes h\cdot \phi_{H})  &= |yz\rangle\langle x| \otimes h\cdot \phi_{H},  \\
  z (|y\rangle\langle x| \otimes h\cdot \phi_{H}) &=
|(S^{-1}_{H}(h_{(1)}) \triangleright z)y\rangle\langle x| \otimes h_{(2)} \cdot \phi_{H}.
\end{align*}
 The map $\hat{ H} \to \hat{A} \cong \mathcal{K}\otimes \hat{H}$ given by $\upsilon \mapsto \tilde{\upsilon}:= |1_{C} \rangle\langle 1_{C}| \otimes \upsilon$
is a morphism of multiplier Hopf algebroids. The counit functional $\hat{\varepsilon}$ and the dual right integral $\hat{\phi}$ of $\widehat{\mathcal{A}}$ are given by
\begin{align*}
  \hat{\varepsilon}(|x\rangle\langle y| \otimes \upsilon) &= \mu_{C}(x)\mu_{C}(y) \varepsilon_{\hat{H}}(\upsilon), &
                                                                                                                  \hat{\phi}(|x\rangle\langle y| \otimes \upsilon) &= \mu_{C}(yx) \phi_{\hat{H}}(\upsilon).
\end{align*}
\end{proposition}
 \begin{proof}
As a vector space, $\hat{A}$ is spanned by functionals of the form $hy\cdot \phi \cdot x^{\op}$, where $x,y\in C$ and $h \in H$.  Such a functional  acts on the left on $A$  as follows:
\begin{align*}
  (hy\cdot \phi\cdot x^{\op}) \triangleright x'^{\op}h'y' &=
                                                   (\iota \odot \cphic)  ((1\otimes x^{\op})\Delta_{C}(x'^{\op}h'y')(1\otimes hy)) \\
&= t_{C}(\cphic(x^{\op}x'^{\op}h'_{(2)}hy))h'_{(1)}y' 
 \\
&= \mu_{C}(x'x)y^{\op}  h'_{(1)}\phi_{H}(h'_{(2)}h)  y' \\
&= (|y\rangle\langle x| x')^{\op}  \, ((h\cdot \phi_{H}) \triangleright h') \, y'.
\end{align*}
Therefore, the map $hy\cdot \phi \cdot x^{\op} \mapsto |y\rangle\langle x| \otimes h\cdot \phi_{H}$ is an isomorphism from $\hat{A}$ to $\mathcal{K}\otimes \hat{H}$.

The embeddings of $B$ and of $C$  into $M(\hat{ A})$ defined by \eqref{eq:dual-module}   take the form
\begin{align*}
  z^{\op} (|y\rangle\langle x| \otimes h\cdot \phi_{H})  &\equiv (hy \cdot \phi \cdot x^{\op}) \cdot t_{B}(z^{\op})
\\ &= (hy\cdot \phi \cdot x^{\op}\sigma^{-1}_{C}(z)) = (hyz \cdot \phi \cdot x^{\op}) \equiv |yz\rangle\langle x| \otimes h\cdot \phi_{H}, \\
  z (|y\rangle\langle x| \otimes h\cdot \phi_{H}) &\equiv z\cdot (hy \cdot \phi \cdot x^{\op}) \\
&= 
h_{(2)}(S^{-1}_{H}(h_{(1)}) \triangleright z)y \cdot \phi \cdot x^{\op}  \equiv
|(S^{-1}_{H}(h_{(1)}) \triangleright z)y\rangle\langle x| \otimes h_{(2)} \cdot \phi_{H}.
\end{align*}

The map $\upsilon\mapsto \tilde{\upsilon}$  is an algebra homomorphism because $\mu_{C}$ is normalized. 
We show that this map is compatible with the left comultiplications on $\hat{H}$ and on $\hat{A}$.

Let $h\in H$ and  $\upsilon=h\cdot \phi_{H}\in \hat{ H}$. Then $\tilde{\upsilon} = h\cdot \phi$, $_{C}\tilde{\upsilon}=h\cdot \cphic$, and since $\mu_{C}$ is invariant,
\begin{gather} 
  \tilde{\upsilon}(x'^{\op}h'y') = \mu_{C}(y')\mu_{C}(x')\upsilon(h'), \quad _{C}\tilde{\upsilon}(x'^{\op}h'y) = (h'_{(1)} \triangleright y')\mu_{C}(x')\upsilon(h'_{(2)}), \label{eq:chb-ctilde} \\ \tilde{\upsilon} \triangleright x'^{\op}h'y' = \mu_{C} (x') (\upsilon \triangleright h')y' \nonumber
\end{gather}
for all $x',y'\in C$ and $h' \in H$. Fix $\upsilon,\upsilon' \in \hat{H}$ and let $a=x^{\op}hy$, $a'=x'^{\op}h'y'$ in $A$.
  By Proposition \ref{proposition:dual-canonical} and the calculations above,
\begin{align*}
(\hat{\Delta}_{B}(\tilde{\upsilon})(1 \otimes \tilde{\upsilon'})\mid a\otimes a') &=
  \tilde{\upsilon}(a (\tilde{\upsilon'} \triangleright a')) \\ &= \tilde{\upsilon}(x^{\op}hy (\upsilon' \triangleright h') y') \mu_{C}(x') \\
&= \mu_{C}(x')\mu_{C}(x) \mu_{C}(y(h'_{(1)} \triangleright y'))  \upsilon(hh'_{(2)})\upsilon'(h'_{(3)}),
\end{align*}
where we used the preceding calculations and invariance of $\mu_{C}$. 
 Write
 $\Delta_{\hat{H}}(\upsilon)(1 \otimes \upsilon') = \sum_{i} \omega_{i} \otimes \omega'_{i}$ with $\omega_{i},\omega'_{i} \in \hat{ H}$. Then the expression above can be written in the form
 \begin{align*} 
(\hat{\Delta}_{B}(\tilde{\upsilon})(1 \otimes \tilde{\upsilon'})\mid a\otimes a') &=
    \sum_{i}\mu_{C}(x')\mu_{C}(x) \mu_{C}(y(h'_{(1)} \triangleright y')) \omega_{i}(h)\omega'_{i}(h'_{(2)}).
 \end{align*}
On the other hand, using \eqref{eq:chb-ctilde}, we find that 
\begin{align*}
  (\tilde{\omega}_{i} \otimes \tilde{\omega}'_{i}\mid a\otimes a') &= \tilde{\omega}_{i}(x^{\op}hy {_{C}\tilde{\omega}'_{i}}(x'^{\op}h'y')) \\
&= \tilde{\omega}_{i}(x^{\op}hy (h'_{(1)}\triangleright y')) \mu_{C}(x')\omega'_{i}(h'_{(2)}) \\
&= \mu_{C}(x)\mu_{C}(x') \mu_{C}(y(h'_{(1)} \triangleright y'))\omega_{i}(h)\omega'_{i}(h'_{(2)}),
\end{align*}
where the pairing on the left hand side is the pairing of $\hat{A}^{B} \otimes {_{B}\hat{A}}$ and $\cAs \otimes \csA$, see Lemma \ref{lemma:pairing}. Comparing with the equation above, we find that
\begin{align*}
  \hat{\Delta}_{B}(\tilde{\upsilon})(1 \otimes \tilde{\upsilon'}) = \sum_{i} \tilde{\omega}_{i} \otimes \tilde{\omega}'_{i}.
\end{align*}

 The map $\upsilon\mapsto \tilde{\upsilon}$   intertwines the antipode of $\hat{H}$ and the antipode of $\hat{A}$ because
\begin{align*}
  \tilde{\upsilon}(S(x^{\op}hy)) = 
\tilde{\upsilon}(\sigma_{C}(y)^{\op}S_{H}(h)x)  =\mu_{C}(y)\mu_{C}(x)\upsilon(S_{H}(h)) = \widetilde{\upsilon\circ S_{H}}(x^{\op}hy)
\end{align*}
for all $\upsilon\in \hat{H}$ and $x,y\in C$. In particular, it follows that this map is also compatible with the right comultiplications on $\hat{H}$ and $\hat{A}$.

 The formulas given for the counit functional and the left integral on $\hat{A}$ follow immediately from the definitions.
 \end{proof}

\subsection{Braided-commutative Yetter-Drinfeld algebras}

Let $(H,\Delta_{H})$ be a regular Hopf algebra with an integral $\phi_{H}$ that is normalized in the sense that $\phi_{H}(1_{H})=1$, and denote its counit and antipode by $\varepsilon_{H}$ and $S_{H}$, respectively. Let $C$ be a unital braided-commutative Yetter-Drinfeld algebra, that is, $C$ is a left $H$-module algebra and a right $H^{\op}$-comodule algebra such that the left action, written $(h,y) \mapsto h\actright y$, and the coaction, written $y\mapsto y_{(0)} \otimes y_{(1)}$, satisfy the  Yetter-Drinfeld condition
\begin{align} \label{eq:yd-condition}
  (h_{(2)} \actright y)_{(0)} \otimes (h_{(2)} \actright y) _{(1)} h_{(1)} &= h_{(1)} \actright y_{(0)} \otimes h_{(2)}  y_{(1)}
\end{align}
and the braided commutativity condition
\begin{align} \label{eq:yd-braided}
  yy' &= y'_{(0)} (y'_{(1)} \actright y) = (S_{H}(y_{(1)}) \actright y')y_{(0)}.
\end{align}
 Suppose furthermore that $\mu_{C}$ is a faithful functional on $C$ that is normalized, that is, $\mu_{C}(1_{C})=1$, and $H$-invariant in the sense that
\begin{align} \label{eq:yd-invariance}
  \mu_{C}(y_{(0)}) y_{(1)} = \mu_{C}(y)1_{H} \quad \text{and} \quad \mu_{C}(h \actright y) &= \varepsilon_{H}(h)\mu_{C}(y)
\end{align}
for all $h\in H$ and $y\in C$.

Denote by $C\# H$ the smash product for the action, which is the vector space $C\otimes H$ with multiplication given by $(y\otimes h)(y' \otimes h') = y (h_{(1)} \actright y')\otimes h_{(2)}h'$, and identify $C$ and $H$ with the subalgebras $C\# 1_{H}$ and $1_{C} \# H$ of $C\# H$. 

 Then one can define a (unital) measured regular multiplier  Hopf algebroid 
\begin{align*}
  (\mathcal{A},\mu_{B},\mu_{C},\bpsib,\cphic), \quad \text{where } \mathcal{A}=(A,B,C,t_{B},t_{C},\Delta_{B},\Delta_{C}),
\end{align*}
as follows; see  Theorem 4.1 in \cite{militaru}, Example 4.1.5 in \cite{boehm:hopf} and \cite[Examples 2.6 and 5.5]{timmermann:integration}:
\begin{gather*}
  A = C\# H, \quad B = \{ y_{(0)} \# y_{(1)} : y\in C\} \subseteq C\# H, \quad C \equiv \{ y \# 1 : y\in C\}, \\
  \begin{aligned}
     t_{B}(y_{(0)} \# y_{(1)}) &=
    S_{H}^{-1}(y_{(1)}) \actright y_{(0)},& t_{C}(y) &= y_{(0)} \# y_{(1)}, \\
  (a\otimes a')  \Delta_{B}(hy) &= ah_{(1)}y \otimes a'h_{(2)}, &
  \Delta_{C}(y h)(a \otimes a') &= yh_{(1)}a \otimes h_{(2)}a', 
  \end{aligned} \\
\mu_{B}(y_{(0)} \# y_{(1)}) = \mu_{C}(y), \quad
\cphic(yh) = y\phi_{H}(h), \quad
\bpsib(hx) = \phi_{H}(h)x
\end{gather*}
for all $y\in C$, $x\in B$, $h\in H$, and $a,a'\in A$.
The antipode and counits of $\mathcal{A}$ are  given by
\begin{align} \label{eq:yd-counits-antipode}
  S(yh) &= S_{H}(h)S_{H}^{2}(y_{(1)}) y_{(0)}, & \epsb(hx) &= \varepsilon_{H}(h) x, &  \ceps(yh) &= y\varepsilon_{H}(h)
\end{align}
for all $y\in C$, $h\in H$ and $x\in B$. 

Note that the embedding of $H$ into $A$ is a morphism in the sense of Definition \ref{definition:morphism}.

Let us now determine  the dual multiplier Hopf algebroid
\begin{align*}
  \widehat{\mathcal{A}} = (\hat{A},C,B,t_{B}^{-1},t_{C}^{-1},\hat{\Delta}_{C},\hat{\Delta}_{B}).
\end{align*}
Denote by $\hat{H} = H\cdot \phi_{H}$ the dual multiplier Hopf algebra of $H$, by $\varepsilon_{\hat{H}}$ its counit, and by $\psi_{\hat{H}}$ the dual right integral, given by $h\cdot \phi_{h} \mapsto \phi_{H}(h)$. 

  The algebra $B$ carries  a left action of $\hat{H}$,  given by
  \begin{align*}
    \upsilon \triangleright t_{C}(y) &= t_{C}(y_{(0)})\upsilon(y_{(1)})
  \end{align*}
for all $\upsilon\in \hat{H}$ and $y\in C$. Using the relation $y_{(0)} \otimes y_{(1)}S_{H}(y_{(2)})=y\otimes 1_{H}$, one finds that $\hat{H} \triangleright B = B$, and short calculations show that for all $y,y' \in C$ and $\upsilon,\omega\in \hat{H}$,
\begin{align*}
  \upsilon \triangleright (\omega \triangleright t_{C}(y)) &= t_{C}(y_{(0)})\upsilon(y_{(1)})\omega(y_{(2)}) = (\upsilon\omega) \triangleright t_{C}(y), \\
                                                                                                                                 \upsilon \triangleright t_{C}(y)t_{C}(y') &= t_{C}(y'_{(0)}y_{(0)})\upsilon(y_{(1)}y'_{(1)})  = (\upsilon_{(1)} \triangleright t_{C}(y))(\upsilon_{(2)} \triangleright t_{C}(y')), \\
\upsilon \triangleright t_{C}(1_{C})  &= \varepsilon_{\hat{H}}(\upsilon) t_{C}(1_{C}).
\end{align*}
 Thus, $B$ becomes a left $\hat{H}$-module algebra; see also \cite[Propositions 3.1]{daele:galois}. 
Moreover, by \cite[Propositions 3.2]{daele:galois}  $C$ carries a unique right coaction
\begin{align*}
  C \to M(C\otimes \hat{H}), \ y \mapsto y_{\langle 0\rangle} \otimes y_{\langle 1\rangle},
\end{align*}
such that $h \triangleright y = y_{\langle 1\rangle}(h) y_{\langle 0\rangle} $ for all $h\in H$, $y\in C$, and the resulting map
\begin{align*}
  B \to M(B\otimes \hat{H}^{\op}), \ t_{C}(y) \mapsto t_{C}(y_{\langle 0\rangle}) \otimes y_{\langle 1\rangle}
\end{align*}
turns $B$ into a right $\hat{H}^{\op}$-module algebra.
\begin{proposition} \label{proposition:yd-dual}
The algebra $\hat{A}$ is isomorphic to $B \# \hat{H}$ via
\begin{align*}
  (hy \cdot \phi) \mapsto t_{C}(y) \# (h\cdot \phi_{H}).
\end{align*}
 With respect to this identification, the embeddings of $B$ and $C$ into $M(\hat{A})$ are given by
\begin{align*}
x(x' \# \upsilon)  &=  xx' \#  \upsilon, &
                                                               y(x \# \upsilon) &=xt_{C}(y_{\langle 0\rangle}) \# y_{\langle 1\rangle}\upsilon.
\end{align*}
The map $\hat{H} \to \hat{A} \cong B\# \hat{H}$ given by $\upsilon \mapsto \tilde{\upsilon} := 1_{C}\#\upsilon$
is a morphism of multiplier Hopf algebroids. The counit functional $\hat{\varepsilon}$ and the dual right integral $\hat{\phi}$ of $\widehat{\mathcal{A}}$ are given by
\begin{align*}
\hat{\varepsilon}(t_{C}(z)\# \upsilon) &= \mu_{C}(z)\varepsilon_{\hat{H}}(\upsilon), & \hat{\psi}(t_{C}(z) \# \upsilon) &= \mu_{C}(z)\psi_{\hat{H}}(\upsilon).
\end{align*}
 \end{proposition}
 \begin{proof}
We proceed along the same lines as in the proof of Proposition \ref{proposition:chb-dual}.

 Define a map $\hat{H} \to \hat{A}$, $\upsilon\mapsto \tilde{\upsilon}$, by $h\cdot \phi_{H} \mapsto h\cdot \phi$ for all $h\in H$. We shall see in a moment that then $\tilde{\upsilon} \in \hat{A}$ can be identified with $1_{C}\# \upsilon \in B\# \hat{H}$. 

The map $\upsilon\mapsto \tilde{\upsilon}$ is a homomorphism because for every element $\upsilon=h\cdot \phi_{H}$ in $\hat{H}$, and all $y'\in C$, $h'\in H$,
 \begin{align} \label{eq:yd-act}
   \tilde{\upsilon} \triangleright (y'h') =  (\iota \odot \cphic)(\Delta_{C}(y'h')(1 \otimes h)) = t_{C}(\cphic(h'_{(2)}h)) (y'h'_{(1)}) = y' (\upsilon \triangleright h').
 \end{align}
By \eqref{eq:dual-module}, 
\begin{gather*}
  t_{C}(z)\tilde{\upsilon} = h\cdot \phi \cdot t_{B}(t_{C}(z)) =  hz \cdot \phi,                                       \quad
       (t_{C}(z)\tilde{\upsilon})(y'h') =\phi(y'h'hz)=\mu_{C}(y'z)\upsilon(h'), \\
  (\tilde{\upsilon}t_{C}(z))(y'h') = \phi(t_{C}(z)y'h'h) = \phi(y't_{C}(z)h'h) = \mu_{C}(y'z_{(0)})\upsilon(z_{(1)}h').
 \end{gather*}
Since  $t_{C}(z_{(0)}) \otimes \upsilon(z_{(1)}-) = (\upsilon_{(1)} \triangleright t_{C}(z)) \otimes \upsilon_{(2)}$, we can conclude that
\begin{align*}
  \tilde{\upsilon}t_{C}(z) = t_{C}(\upsilon_{(1)} \triangleright t_{C}(z)) \widetilde{\upsilon_{(2)}}.
\end{align*}
Consequently, the map $t_{C}(z)\#\upsilon \mapsto t_{C}(z)\tilde{\upsilon}$ is a homomorphism from $B\#\hat{H}$ to $\hat{A}$. Since $\hat{A}$ is spanned by functionals of the form $ hz\cdot \phi =t_{C}(z)\tilde{\upsilon}$, this map is an isomorphism.  By  \eqref{eq:dual-module}, 
\begin{align*}
  (y(t_{C}(z)\tilde{\upsilon}))(y'h') &=(t_{C}(z)\tilde{\upsilon})(y'h'y) 
\\ &= (t_{C}(z)\tilde{\upsilon})(y'(h'_{(1)} \triangleright y)h'_{(2)}) 
= \mu_{C}(y'(h'_{(1)} \triangleright y)z) \upsilon(h'_{(2)})
\end{align*}
and hence $y(t_{C}(z)\# \upsilon) = t_{C}(z)t_{C}(y_{\langle 0\rangle}) \# y_{\langle 1\rangle}\upsilon$.

Next, we show  that the map $\upsilon\mapsto \tilde{\upsilon}$ is compatible with the left comultiplications on $\hat{H}$ and $\hat{A}$ . Let $\upsilon,\upsilon' \in \hat{ H}$, $y,y'\in C$ and $h,h'\in H$. Then by Proposition \ref{proposition:dual-canonical},
\begin{align*}
  (\hat{\Delta}_{B}(\tilde{\upsilon})(1\otimes\tilde{\upsilon'})\mid yh\otimes y'h') &=
  \tilde{\upsilon}(yh(\tilde{\upsilon'} \triangleright y'h'))  \\ &=   \tilde{\upsilon}(yhy'(\upsilon'\triangleright h')) = \mu_{C}(y(h_{(1)} \triangleright y'))\upsilon(h_{(2)}h'_{(1)})\upsilon'(h'_{(2)}).
\end{align*}
On the other hand, write $\Delta_{\hat{H}}(\upsilon)(1 \otimes \upsilon') = \sum_{i} \omega_{i} \otimes \omega'_{i}$. Then
\begin{align*} 
  (\tilde{\omega_{i}} \otimes \tilde{\omega_{i}'}\mid yh\otimes y'h') &=  \omega_{i}(yhy')\omega'_{i}(h') \\& = \mu_{C}(y(h_{(1)} \triangleright y')) \omega_{i}(h_{(2)})\omega'_{i}(h') 
=\mu_{C}(y(h_{(1)} \triangleright y')) \upsilon(h_{(2)}h'_{(1)})  \upsilon'(h'_{(2)}),
\end{align*}
where the pairing on the left hand side is the pairing of $\hat{A}^{B} \otimes {_{B}\hat{A}}$ and $\cAs \otimes \csA$, see Lemma \ref{lemma:pairing}. Comparing with the equation above, we find that
\begin{align*}
  \hat{\Delta}_{B}(\tilde{\upsilon})(1 \otimes \tilde{\upsilon'}) = \sum_{i} \tilde{\omega}_{i} \otimes \tilde{\omega}'_{i}.
\end{align*}

Next, we claim that the map $\upsilon \mapsto \tilde{\upsilon}$ intertwines the antipodes on $\hat{H}$ and on $\hat{A}$. Indeed,  since  $\mu_{C}$   is invariant with respect to the action of $H$ and the coaction of $H^{\op}$, 
\begin{align*}
  \tilde{\upsilon}(S(yh)) = \tilde{\upsilon}(S_{H}(h)S_{H}^{2}(y_{(1)})y_{(0)}) = \mu_{C}(y)\upsilon(S_{H}(h)) = 
\widetilde{(\upsilon\circ S_{H})}(yh)
\end{align*}
for all $\upsilon\in\hat{H}$, $y\in C$ and $h\in H$. 

Consequently, the map $\upsilon\mapsto \tilde{\upsilon}$ is also compatible with the right comultiplications and hence a morphism of multiplier Hopf algebroids. 

Again, the formulas given for the counit functional and the left integral on $\hat{A}$ follow immediately from the definitions.
 \end{proof}
 \subsection*{Acknowledgments} The author would like to thank Alfons Van Daele for inspiring and fruitful discussions, and the referee for many helpful comments.

     \bibliographystyle{abbrv}
\def\cprime{$'$}

\end{document}